\documentclass[a4paper,11pt]{article}
\usepackage[latin1]{inputenc}
\usepackage[T1]{fontenc}

\usepackage{array}
\usepackage{fancyhdr}
\usepackage{longtable}
\usepackage{tabularx}
\usepackage{rotating}

\usepackage{makeidx}
\usepackage{amsmath}
\usepackage{amsthm}
\usepackage{amssymb}
\usepackage[mathscr]{eucal}
\usepackage{enumerate}
\usepackage[dvips]{epsfig}
\usepackage{float}  
\usepackage{ifthen} 
\usepackage{extarrows}  
\usepackage{stmaryrd}   
\usepackage{arydshln}     
\usepackage{multirow}
\usepackage{mathtools}
\usepackage[dvips]{geometry} 
\usepackage{listings}
\usepackage{subfig}
\usepackage{dsfont}
\usepackage{xspace}

\usepackage{hyperref}

\theoremstyle{plain}
\newtheorem{thm}{Theorem}[section]
\newtheorem{prop}[thm]{Proposition}
\newtheorem{lma}[thm]{Lemma}

\newtheorem{cor}[thm]{Corollary}
\theoremstyle{definition}

\theoremstyle{remark}

\newcommand{\Prob}{\operatorname{P}}
\newcommand{\Erw}{\operatorname{E}}

\newcommand{\tr}{\operatorname{tr}}

\newcommand{\Span}{\operatorname{span}}

\newcommand{\costs}{\operatorname{cost}}
\newcommand{\err}{\operatorname{err}}
\newcommand{\ccosts}{\bar{c}}
\newcommand{\Oo}{\mathcal{O}}

\newcommand{\Y}{Y}

\newcommand{\YY}{Y^{N,K,M}}
\newcommand{\WW}{W^{K,M}}
\newcommand{\E}[1]{\mathrm{E}\left[\left\| #1 \right\|_H^2\right]}

\newcommand{\sI}{\sum_{l=0}^{m-1} \int_{t_l}^{t_{l+1}}}
\newcommand{\sII}{\sum_{l=0}^{m-2} \int_{t_l}^{t_{l+1}}}
\newcommand{\s}{\sum_{l=0}^{m-2}}
\newcommand{\su}{\sum_{l=0}^{m-1}}
\newcommand{\I}{\int_{t_l}^{t_{l+1}}}
\newcommand{\e}[1]{e^{A(#1)}}

\newcommand{\MIL}{\operatorname{MIL}}
\newcommand{\ESRK}{\operatorname{CDFM}}
\newcommand{\SESRK}{\operatorname{DFMM}}
\newcommand{\RKS}{\operatorname{RKS}}
\newcommand{\LIE}{\operatorname{LIE}}
\newcommand{\EES}{\operatorname{EES}}
\newcommand{\SCHEME}{\operatorname{SCHEME}}
\newcommand{\EDFM}{derivative-free Milstein scheme\xspace}

\title{{Enhancing the Order of the Milstein Scheme for Stochastic Partial
Differential Equations with Commutative Noise}}

\author{Claudine Leonhard\thanks{e-mail: leonhard@math.uni-luebeck.de}
\ \ and Andreas R\"o\ss ler\thanks{e-mail: roessler@math.uni-luebeck.de}
\bigskip
\\
\small{Institute of Mathematics, Universit\"at zu L\"ubeck,} \\
\small{Ratzeburger Allee 160, 23562 L\"ubeck, Germany} 
}

\date{}

\begin{document}

\maketitle

\begin{abstract}
We consider a higher-order Milstein scheme for stochastic partial
differential equations with trace class noise which fulfill a
certain commutativity condition. A novel technique 
to generally improve the order of convergence of
Taylor schemes for stochastic partial differential equations is 
introduced. The key tool is an efficient approximation of the 
Milstein term by particularly tailored nested derivative-free 
terms. For the resulting derivative-free Milstein scheme 
the computational cost is, in general, considerably reduced by some power.
Further, a rigorous computational cost model is considered
and the so called effective order of convergence
is introduced 
which allows to directly compare various numerical schemes 
in terms of their efficiency.
As the main result, we prove for a broad class of stochastic partial differential equations,
including equations with operators that do not need to be pointwise multiplicative,
that the effective order of convergence 
of the proposed derivative-free Milstein scheme is significantly higher than 
for the original Milstein scheme.
In this case, the derivative-free Milstein scheme outperforms the Euler scheme
as well as the original Milstein scheme due to the
reduction of the computational cost. 
Finally, we present some numerical examples that confirm the theoretical results. 
\end{abstract}
\section{Introduction}
Stochastic partial differential equations (SPDEs) are a powerful
tool in modeling various phenomena from biology to finance. 
Since analytical solutions to
these equations are, in general, not computable, there is a high demand
for numerical schemes to approximate these processes. 

%
In this work, we are concerned with semilinear parabolic SPDEs with commutative noise on a
probability space
$(\Omega, \mathcal{F}, \Prob)$ and on the time interval $[0,T]$ for some $T\in(0,\infty)$ with
some filtration
$(\mathcal{F}_t)_{t\in[0,T]}$ fulfilling the usual conditions. These
SPDEs are of the following general form
\begin{equation} \label{semiSPDE}
  \mathrm{d} X_t = \left(A X_t + F(X_t) \right) \, \mathrm{d}t +
  B(X_t) \, \mathrm{d}W_t, \quad \quad X_0 = \xi.
\end{equation}
The solution process $(X_t)_{t \in [0,T]}$ is $H_{\gamma}$-valued
for some suitable $\gamma \in [0,1)$ and
$(W_t)_{t \in[0,T]}$ is a $U$-valued $Q$-Wiener process.
Details on the operators,
spaces, and processes will be given in Section \ref{frame}. 

%
Even though there has been a lot of research on numerical methods for
stochastic differential equations in infinite dimensions over the 
years, for example, \cite{MR3004668, MR3484400, MR1402994,MR1644183,
MR1953619, MR2830608, MR2471778,
MR2677551, MR2094572, MR3047942, MR2324415,
MR2136207, MR2182132},
methods with a high order of convergence and derivative-free schemes
remain rare, see \cite{MR2996432, MR3027891, MR3534472, MR3081485, MR3248844,
MR3320928, MR2677551} and \cite{MR3011387}, respectively.
The numerical approximation of SPDEs requires the discretization of
both the time and space domains as well as the infinite dimensional
stochastic process. With regard to space, most schemes work with a
spectral Galerkin method or a finite element discretization to
obtain a finite dimensional system of stochastic differential
equations, see  \cite{MR3004668, MR3320928, MR2646102, MR2136207},
or \cite{MR2182132}, for example.
Concerning the approximation with respect to the temporal direction,
the linear implicit Euler method is the benchmark, see \cite{MR2728973,
MR1341554, MR2377271}, or \cite{MR1825100}. 

%
Recently, it was shown by A.~Jentzen and P.~E.~Kloeden \cite{MR2471778}
that a higher order of
convergence can be obtained when employing schemes which are
developed on the basis of the mild solution of \eqref{semiSPDE},
that is,
\begin{equation} \label{mild}
  X_t = e^{At} \, \xi + \int_0^t e^{A(t-s)} \, F(X_s) \, \mathrm{d}s
  + \int_0^t e^{A(t-s)} \, B(X_s) \, \mathrm{d}W_s \qquad \Prob\text{-a.s.}
\end{equation}
for $t \in [0,T]$. Based on this finding, the exponential Euler
scheme~\cite{MR2471778}, the Milstein scheme for SPDEs in
\cite{MR3320928}, and the numerical scheme in \cite{MR3047942} have
been built. In the present paper, we focus on the Milstein scheme
proposed by A.~Jentzen and M.~R\"ockner \cite{MR3320928}
and derive a scheme which is free of derivatives, therefore easier
to compute and in general
more efficient when considering errors versus cost. This
results in a higher effective order of convergence
compared to the original Milstein scheme, the exponential or the
linear implicit Euler scheme.

%
In order to make our main result more clear, we first consider the
Milstein scheme for finite dimensional stochastic differential
equations (SDEs).
Let $n, k \in \mathbb{N}$
and let $(W_t)_{t \in [0,T]}$ be a $k$-dimensional Brownian
motion with respect to $(\mathcal{F}_t)_{t \in [0,T]}$. Furthermore,
assume $a \colon \mathbb{R}^n \to \mathbb{R}^n$ and $b = (b_1,
\ldots, b_k) \colon \mathbb{R}^n \to \mathbb{R}^{n \times k}$ with
$b_j(x) = (b_{1,j}(x), \ldots, b_{n,j}(x))^T$, $j \in \{1, \ldots,
k\}$, $x \in \mathbb{R}^n$, to be Lipschitz continuous functions.
Then, the $n$-dimensional system of SDEs
\begin{align*}
    \mathrm{d} X_t = a(X_t) \, \mathrm{d}t
    + \sum_{j=1}^k b_j(X_t) \, \mathrm{d}W_t^j
\end{align*}
for $t \in [0,T]$ with initial value $X_0 = \xi \in \mathbb{R}^n$
has a unique solution
\cite{MR1214374}. Let an equidistant discretization of the time
interval $[0,T]$ with step size $h = \frac{T}{M}$ for some $M \in
\mathbb{N}$ and $t_m = m \, h$ for $m \in \{0,\ldots,M\}$ be
given. Further, let $\Delta W_m^j = W_{t_{m+1}}^j-W_{t_m}^j$ for all $j\in\{1,\ldots,k\}$. Then,
the stochastic double integrals can be expressed as
 \begin{align*}
      \int_{t_m}^{t_{m+1}} \int_{t_m}^s \, \mathrm{d}W_u^j \,
      \mathrm{d}W_s^i
      + \int_{t_m}^{t_{m+1}} \int_{t_m}^s \, \mathrm{d}W_u^i \,
      \mathrm{d}W_s^j
      = \Delta W_m^i \, \Delta W_m^j
 \end{align*}
for $i, j \in \{1,\ldots, k\}$ with $i \neq j$ and $m\in\{0,\ldots,M-1\}$, where the right-hand
side can be easily simulated. For now, we assume the SDE to be
commutative, that is,
\begin{equation*}
    \sum_{r=1}^n b_{r,j} \frac{\partial b_{l,i}}{\partial x_r}
    = \sum_{r=1}^n b_{r,i} \frac{\partial b_{l,j}}{\partial x_r}
\end{equation*}
for $l \in \{1, \ldots, n\}$ and $i,j \in \{1,\ldots, k\}$. Then,
for the commutative SDE system, the Milstein scheme can be
reformulated as $Y_0^M=\xi$ and
\begin{align*}
    Y_{m+1}^M &= Y_m^M + h \, a(Y_m^M) +  \sum_{j=1}^k b_j(Y_m^M) \,
    \Delta W_m^j
    + \frac{1}{2} \sum_{i,j=1}^k
    \left( \frac{\partial b_{l,i}}{\partial x_r}(Y_m^M)
    \right)_{1 \leq l,r \leq n} \, b_j(Y_m^M) \,
    \big( \Delta W_m^i \, \Delta W_m^j \big) \\
    &\quad -\frac{h}{2} \sum_{j=1}^k
    \left( \frac{\partial b_{l,j}}{\partial x_r}(Y_m^M)
    \right)_{1 \leq l,r \leq n} \, b_j(Y_m^M) ,
\end{align*}
for $m \in \{0,\ldots, M-1\}$, which is easy to implement because no
double integrals have to be simulated, see \cite{MR1214374} for more
details. Compared to the Euler-Maruyama method having strong order
$1/2$, the Milstein scheme attains strong order $1$ in this case.
However, for the Jacobian $\big( \frac{\partial b_{l,i}}{\partial
x_r}(Y_m^M) \big)_{1 \leq l,r \leq n}$ one has to evaluate $n^2$
scalar (nonlinear) functions at $Y_m^M$ for $i \in \{1, \ldots, k\}$ in each
time step. Thus, for an approximation at time $T$ one has to
evaluate $\Oo(n^2 \, k \, M)$ scalar nonlinear functions due to
the Jacobian matrix. If $n$ and $k$ are moderately large, e.g.,
$n=k=30$, already $30^3=27000$ function evaluations are necessary
for the Jacobian in each step, which needs significant computation time. On
the other hand, one step of the Euler-Maruyama scheme is much
cheaper because for the function $b$ only $30^2=900$ scalar
(nonlinear) functions have to be evaluated whereas an evaluation of
the Jacobian is not necessary. In general, the Euler-Maruyama scheme
needs one evaluation of the drift $a$ and the function $b$ in each
step which results in only $\Oo(n \, k \, M)$ evaluations of scalar
(nonlinear) functions for an approximation at time $T$ however, with a 
low order of convergence only. This problem is well known and a
special technique overcoming this trade-off in the SDE
setting has been introduced by one of the authors
\cite{MR2338537,MR2505871,MR2669396}. Especially, in case of
commutative noise, strong order 1.0 schemes with only $\Oo(n \, k \, M)$
evaluations of scalar functions are proposed in \cite{MR2669396}.

%
In the infinite dimensional setting, one has to be much more careful
as the number of function evaluations in the Milstein scheme is
'cubic' with respect to the dimensions of the finite dimensional
projection subspaces. The dimensions $N$ and $K$ of these subspaces have to
increase to obtain higher approximation accuracy.
The Milstein scheme for
SPDE~\eqref{semiSPDE} proposed by A.~Jentzen and M.~R\"ockner~\cite{MR3320928} reads as
$\YY_0 = P_N \xi$ and
\begin{equation}\label{Milstein}
    \begin{split}
    \YY_{m+1} &= P_N \Big( e^{Ah} \Big( \YY_m + h F(\YY_m) +
    B(\YY_m) \Delta \WW_m \\
    &\quad + \frac{1}{2} B'(\YY_m) \big( B(\YY_m) \Delta \WW_m,
    \Delta \WW_m \big) \\
    &\quad - \frac{h}{2} \sum_{\substack{j \in
    \mathcal{J}_K \\ \eta_j \neq 0}}
    \eta_j \, B'(\YY_m) \big( B(\YY_m) \tilde{e}_j, \tilde{e}_j
    \big) \Big) \Big) 
    \end{split}
\end{equation}
for $m\in\{0, \ldots, M-1\}$. Details on the operators and the notation can be 
found in Section~\ref{Numerics}.
In the examples in \cite{MR3320928},
A.~Jentzen and M.~R{\"o}ckner solve the issue of high dimensionality by
restricting the operator $F$ to be of the form $(F(v))(x)=f(x,v(x))$
and the operator $B$ to be in a class which is pointwise multiplicative in the
$Q$-Wiener process, that is, $(B(v) u)(x) = b(x,v(x)) \cdot u(x)$ for
all $x \in (0,1)^d$, $u, v \in H=U=L^2((0,1)^d, \mathbb{R})$, $f, b
\colon (0,1)^d \times \mathbb{R} \to \mathbb{R}$, and $d \in \{1,2,3 \}$.
Thereby, the authors avoid computational costs which are 'cubic' in the
dimensions of the problem for each step. Moreover, the scheme is
also applicable if this restriction does not hold, however, then the
computational cost also become 'cubic' in the dimensions of the
projection subspaces. 

%
Further, a derivative-free version of the Milstein scheme
for SPDEs is derived in \cite{MR3011387} under certain
conditions. However, this scheme is not applicable to general equations of type~\eqref{semiSPDE} 
but restricted to SPDEs that are pointwise multiplicative in the $Q$-Wiener process.
In particular, this scheme makes use of a bilinear approximation operator
for the derivative in the Milstein scheme
which needs to fulfill some special conditions stated as Assumption~2.5
in \cite{MR3011387}. 
In contrast, this assumption is not required to be fulfilled by the
scheme that we propose in the following. 
Finally, we want to point out that there are plenty of applications 
from various disciplines modeled by SPDEs that do not 
belong to the special setting of pointwise multiplicative operators, see 
\cite{MR2825337,MR2174871,MR3090646,MR3305472,MR2097528,MR2520127,MR3511288,schnoerr2016cox}, for example. 
For these equations, the original Milstein scheme in
\cite{MR3320928} cannot be applied efficiently due to its cubic computational 
cost, nor can the derivative-free version in \cite{MR3011387} be used
at all.

%
In this paper, we present
a different approach to dealing with the problem of
high dimensionality in the numerical approximation of SPDEs.
This approach leads to a method that is derivative-free and 
efficiently approximates SPDEs of type \eqref{semiSPDE} where the
operator $B$ is not restricted to be pointwise multiplicative in the 
$Q$-Wiener process. 
For the special case of a pointwise multiplicative operator, our
new approach has the same effective order of convergence as the schemes proposed in
\cite{MR3320928} and \cite{MR3011387} since the computational cost is of
the same order of magnitude. However, to treat this special class is not our
main goal and in the general case we can improve the effective order of
convergence compared to the Milstein scheme in \cite{MR3320928}.
Recently, a special technique to reduce the computational costs by a factor
depending on the dimensions of the considered SDE system to be
solved was proposed for the first time by A.~R\"o{\ss}ler for
finite dimensional SDEs, see \
\cite{MR2338537,MR2505871,MR2669396}, for example. This technique opened the door
for the efficient application of higher-order schemes in the case of
high dimensional SDE systems. Here, the idea is to carry over this
approach to the infinite dimensional setting of
SPDEs where it becomes even more powerful because one can
achieve an improvement of the order of convergence.
In this work, we derive a scheme which is efficiently applicable
to a broad class of SPDEs.
We approximate the derivative and reduce the large
number of function evaluations by choosing the approximation
operator carefully. The resulting \EDFM approximates the mild solution
\eqref{mild} of \eqref{semiSPDE} with the same theoretical order of
convergence with respect to the spatial and time discretizations as
the schemes in \cite{MR3320928} and, in the special case of 
pointwise multiplicative operators, as
the scheme in \cite{MR3011387}. However, the
computational cost is reduced by one order of magnitude for a
general class of semilinear SPDEs with commutative noise
and the effective order of convergence can thus be increased. 

%
The structure of the paper is as follows. First, we lay the
theoretical foundation and present the setting in which we work.
In Section \ref{Numerics}, we introduce the enhanced \EDFM and state
convergence results. Then, an information
based model for computational cost is proposed in order to compare
the quality of different numerical schemes for SPDEs.
We show that the computational cost for the \EDFM is significantly lower in
comparison to the original Milstein scheme. Although having a higher order
of convergence, the \EDFM possesses a computational cost
of the same order of magnitude
as the linear implicit Euler and the exponential Euler scheme in each time step.
Based on this model of computational cost, we compare the effective order
of convergence of the introduced \EDFM
with that for some recent numerical schemes and state our main result.
Finally, we present a proof of convergence for the proposed scheme.
\section{Framework for the considered SPDEs} \label{frame}
Let $(H,\langle \cdot,\cdot \rangle_H)$ and $(U,\langle \cdot,\cdot
\rangle_U)$ denote real separable Hilbert spaces.
Further, let $Q \in L(U)$ be a nonnegative and symmetric trace class
operator, i.e., for some finite or countable index set $\mathcal{J}$,
it holds that
\begin{equation*}
  \tr(Q) = \sum_{j \in \mathcal{J}} \langle Q \tilde{e}_j,
  \tilde{e}_j \rangle_U
  < \infty,
\end{equation*}
where $\{ \tilde{e}_j : \ j \in \mathcal{J}\}$ is an orthonormal
basis of eigenfunctions of $Q$ in $U$ such that there exist
eigenvalues $(\eta_j)_{j \in \mathcal{J}}$ with $\eta_j \in
[0,\infty)$ and $Q \tilde{e}_j = \eta_j \tilde{e}_j$ for all $j \in
\mathcal{J}$, see \cite[Proposition~2.1.5]{MR2329435}, for example.
Then, $(U_0, \langle \cdot, \cdot \rangle_{U_0})$ with $U_0 :=
Q^{\frac{1}{2}}(U)$ and $\langle u,v \rangle _{U_0} = \langle
Q^{-\frac{1}{2}} u, Q^{-\frac{1}{2}} v \rangle _U$ for all $u,v \in
U$ is a separable Hilbert space. Here, we denote by $T^{-1} \colon
T(U) \to \ker(T)^{\bot}$ the pseudoinverse of a linear operator $T
\in L(U)$ if $T$ is not one-to-one, see \cite[Appendix~C]{MR2329435}.
In the following, let $(W_t)_{t \in [0,T]}$ be a $U$-valued
$Q$-Wiener process with respect to the filtration
$(\mathcal{F}_t)_{t \in [0,T]}$ that fulfills the usual conditions,
which is defined on the probability space $(\Omega, \mathcal{F},
\Prob)$.
For some fixed $T \in (0,\infty)$, we study the following equation
\begin{equation} \label{SPDE}
  \begin{split}
    \mathrm{d} X_t & = \left( A X_t + F(X_t) \right) \, \mathrm{d}t
    + B(X_t) \, \mathrm{d}W_t, \quad t \in (0,T], \\
    X_0  & = \xi,
  \end{split}
\end{equation}
where the linear operator $A$ is the infinitesimal generator of
a $C_0$-semigroup.
Moreover, let $F$ be
the drift coefficient which may be a nonlinearity, let $B$
be a Hilbert-Schmidt operator-valued coefficient, and let
$\xi$ be a random initial value. 

%
In the following, we consider the space $(L(U,H)_{U_0}, \| \cdot \|_{L(U,H)})$
with $L(U,H)_{U_0}:= \{\left.T\right|_{U_0} : \ T\in L(U,H)\}$ which
is a dense subset of $L_{HS}(U_0,H)$~\cite{MR2329435}.
For the analysis of convergence of the \EDFM, we make the
following assumptions which are similar 
to those for the original
Milstein scheme proposed in \cite{MR3320928}. For easy comparison of the presented results, we adopt
the notation used in~\cite{MR3320928}:
\\ \\
(A1) For the linear operator $A \colon D(A) \subset H \to H$, there exist
eigenfunctions $(e_i)_{i \in \mathcal{I}}$ in $H$ and eigenvalues
$(\lambda_i)_{i \in \mathcal{I}}$ with $\lambda_i \in (0,\infty)$
and $\inf_{i \in \mathcal{I}} \lambda_i
>0$, such that $-A e_i= \lambda_i \, e_i$ for all $i \in \mathcal{I}$,
where $\mathcal{I}$ is a finite or countable index set, and such
that the eigenfunctions constitute an orthonormal basis of $H$. The
domain of $A$ is defined as $D(A) = \{u \in H : \ \sum_{i \in
\mathcal{I}} |\lambda_i|^2 \, |\langle u,e_i \rangle_H |^2 <
\infty\}$ and for all $x \in D(A)$, it holds that
\begin{equation*}
    A x = \sum_{i \in \mathcal{I}} -\lambda_i \, \langle x,e_i
    \rangle_H \, e_i .
\end{equation*}
%
Here, $A$ is the generator of an analytic semigroup $\{ S(t) : \ t \geq
0 \}$ of linear operators in $H$ which are denoted as $S(t) = e^{A
t}$ for $t \geq 0$ \cite{MR2028503}. For $\rho \in [0,\infty)$, 
we denote the domain of the fractional power of $-A \colon D(A) \to H$ as
$H_{\rho} := D((-A)^{\rho})$ with norm $\|u \|_{H_{\rho}} := \|
(-A)^{\rho} u \|_H$ for $u \in H_{\rho}$. These domains
are real Hilbert spaces with the relation $H_{\rho_2} \subset
H_{\rho_1} \subset H$ for $\rho_2 \geq \rho_1 \geq 0$ \cite{MR1873467}.
\\ \\
(A2) Let $F \colon H_{\beta} \rightarrow H$ for some $\beta \in [0,1)$,
and we assume the mapping to be twice continuously Fr\'{e}chet
differentiable with $\sup_{v \in H_{\beta}} \|F'(v)\|_{L(H)} <
\infty$ and $\sup_{v \in H_{\beta}} \|F''(v)\|_{L^{(2)}(H_{\beta},H)} <
\infty$. 
\\ \\
(A3) Let $B \colon H_{\beta} \rightarrow  L(U,H)_{U_0}$, and assume $B$ to be
twice continuously Fr\'{e}chet differentiable such that it holds that
$\sup_{v \in H_{\beta}} \|B'(v)\|_{L(H,L(U,H))} < \infty$, 
$\sup_{v \in H_{\beta}} \|B''(v)\|_{L^{(2)}(H, L(U,H))} < \infty$.
Furthermore, let $B(H_{\delta}) \subset L(U,H_{\delta})$ for some $\delta
\in (0,\tfrac{1}{2})$ and assume that there exists a constant $C >0$
such that
\begin{align*}
    \| B(u) \|_{L(U,H_{\delta})} &\leq C ( 1 +
    \| u \|_{H_{\delta}} ) , \\
    \| B'(v) P B(v) - B'(w) P B(w) \|_{L_{HS}^{(2)}(U_0,H)} &\leq C
    \| v - w \|_{H} , \\
    \| (-A)^{-\vartheta} B(v) Q^{-\alpha} \|_{L_{HS}(U_0,H)} &\leq C
    (1 + \| v \|_{H_{\gamma}})
\end{align*}
for all $u \in H_{\delta}$, $v, w \in H_{\gamma}$,
where
$\alpha \in (0,\infty)$, $\vartheta \in \left( 0,
\frac{1}{2} \right)$, $\gamma \in \left[ \max( \beta, \delta),
\delta + \frac{1}{2} \right)$, and for
any projection $P \colon H \to \tilde{H}$ 
of $H$ onto $\tilde{H} = \Span \{ e_i : \ i \in \tilde{\mathcal{I}}\} \subset H$ 
with a finite index set $\tilde{\mathcal{I}} \subset \mathcal{I}$
as well as for the case that $P$ is the identity.
Note that $\beta \in [0,
\delta + \tfrac{1}{2})$. Here, let 
$L^{(2)}(H,L(U,H)) = L(H,L(H,L(U,H)))$ 
and let for all $v \in H_{\beta}$ the mapping 
$B'(v) B(v) \colon U_0 \times U_0 \to
H$ with $\big( B'(v) B(v) \big) (u, \tilde{u}) = \big( B'(v) (B(v)
u) \big) \tilde{u}$ for $u, \tilde{u} \in U_0$ 
be a bilinear
Hilbert-Schmidt operator in $L_{HS}^{(2)}(U_0,H) =
L_{HS}(U_0,L_{HS}(U_0,H))$.
Moreover, for all $v \in H_{\beta}$, the operator $B'(v)B(v) \in
L_{HS}^{(2)}(U_0,H)$ is assumed to be symmetric, i.e., the operator
fulfills the commutativity condition
\begin{equation} \label{Comm}
    \big( B'(v) (B(v) u) \big) \tilde{u} =
    \big( B'(v) (B(v) \tilde{u}) \big) u
\end{equation}
for all $u, \tilde{u} \in U_0$. 
\\ \\
(A4) The initial value $\xi \colon \Omega \to H_{\gamma}$ is assumed to
be an $\mathcal{F}_0$-$\mathcal{B}(H_{\gamma})$-measurable random
variable such that
$E[ \| \xi \|_{H_{\gamma}}^4] < \infty$ is fulfilled. 
\\ \\
Note that Assumption (A3) is partially different from the assumptions
in \cite{MR3320928} where $B \colon H_{\beta} \to L_{HS}(U_0,H)$
and some slightly differing conditions on the derivatives of $B$ are
imposed. Because $L(U,H)_{U_0}$ is a dense subset of $L_{HS}(U_0,H)$
and since it holds that $\sup_{v \in H_{\beta}} \|B'(v)\|_{L(H,L_{HS}(U_0,H))} \leq
(\tr (Q))^{1/2} \, \sup_{v \in H_{\beta}} \|B'(v)\|_{L(H,L(U,H))} < \infty$,
an operator for which (A3) holds also fulfills the setting in
\cite{MR3320928}. We require these modified conditions
in some parts of our proof of convergence where we cannot employ 
It{\^o}'s isometry, see \eqref{proof-eqn-Ito}, for example.
Further, since $H_{\beta}$ is a dense subset of $H$,
it follows
that ${B} \colon H_{\beta} \to L(U,H)$ can be continuously
extended to a globally Lipschitz continuous mapping $\tilde{B}
\colon H \to L(U,H)$. In the following, to keep the presentation
simple, it is not distinguished between $B$ and $\tilde{B}$. The
same applies to $F$ respectively. 

%
If Assumptions (A1)--(A4) are fulfilled, then there exists a unique mild solution for 
SPDE~\eqref{SPDE}, see A.~Jentzen and M.~R\"ockner~\cite{MR2852200,MR3320928}.
\begin{prop}[Existence and uniqueness of the mild solution]\label{ExSol}
    Let Assumptions (A1)--(A4) be fulfilled. Then, there exists an
    up to modifications unique predictable mild solution $X
    \colon [0,T] \times \Omega \to H_{\gamma}$ for \eqref{SPDE} with
    $\sup_{t \in [0,T]} \Erw[ \|X_t \|_{H_{\gamma}}^4 + \| B(X_t)
    \|_{L_{HS}(U_0,H_{\delta})}^4 ] < \infty$ and
    \begin{equation}
        X_t = e^{At} \xi + \int_0^t e^{A(t-s)} F(X_s) \, \mathrm{d}s
        + \int_0^t e^{A(t-s)} B(X_s) \, \mathrm{d}W_s \quad
        \Prob\text{-a.s.}
    \end{equation}
    for all $t \in [0,T]$ with
    \begin{equation*}
        \sup_{\substack{s,t \in [0,T] \\ s \neq t}}
        \frac{\left( \mathrm{E}[ \|X_t-X_s\|_{H_r}^p ] \right)^{\frac{1}{p}}}
        {|t-s|^{\min(\gamma-r, \frac{1}{2})}}
        <\infty
    \end{equation*}
    for every $r \in [0,\gamma]$ and $p \in [2,4]$. Furthermore, the process $(X_t)_{t
    \in [0,T]}$ is continuous with respect to $\big(\Erw[ \|
    \cdot \|_{H_{\gamma}}^4 ] \big)^{1/4}$.
\end{prop}
\section{The enhanced derivative-free Milstein scheme}
\label{Numerics}
In order to derive a numerical scheme for SPDEs, we project the infinite
dimensional state space onto a finite dimensional subspace and
discretize the time interval. In the following, let
$(\mathcal{I}_N)_{N \in \mathbb{N}}$ and $(\mathcal{J}_K)_{K \in
\mathbb{N}}$ be sequences of finite subsets such that $\mathcal{I}_N
\subset \mathcal{I}$ and $\mathcal{J}_K \subset \mathcal{J}$ for all
$K,N \in \mathbb{N}$.
For $N \in \mathbb{N}$, let $P_N \colon H \to H_N$ denote the projection
of the infinite dimensional space $H$ onto the finite dimensional
subspace $H_N = \Span \{ e_i : \ i \in \mathcal{I}_N\} \subset H$
defined by
\begin{equation*}
    P_N v = \sum_{i \in \mathcal{I}_N} \langle v,e_i \rangle_H \,
    e_i
\end{equation*}
for $v \in H$.
Analogously, for $K \in \mathbb{N}$, let $(W_t^K)_{t \in [0,T]}$
denote the projection of the $U$-valued $Q$-Wiener process
$(W_t)_{t \in [0,T]}$ onto the finite dimensional subspace $U_K =
\Span \{ \tilde{e}_j : \ j \in \mathcal{J}_K \} \subset U$ defined
by
\begin{equation*}
    W_t^K = 
    \sum_{\substack{j \in \mathcal{J}_K \\ \eta_j \neq 0}} \langle
    W_t,\tilde{e}_j \rangle_{U} \, \tilde{e}_j
    = \sum_{\substack{j \in \mathcal{J}_K \\ \eta_j \neq 0}}
    \sqrt{\eta_j} \, \beta_t^j \, \tilde{e}_j \quad
    \Prob\text{-a.s.},
\end{equation*}
where $(\beta_t^j)_{t \in [0,T]}$ are independent real-valued
Brownian motions for $j \in \mathcal{J}_K$ with $\eta_j \neq 0$.
As the next step, we consider a discretization of the time domain. For
legibility, the interval $[0,T]$ is divided into $M \in \mathbb{N}$
equally spaced subsets of length $h = \tfrac{T}{M}$ with $t_m = m
\, h$ for $m \in \{ 0, \ldots, M \}$. In particular, we make use of the
increments
\begin{equation*}
    \Delta \WW_m := W_{t_{m+1}}^K - W_{t_m}^K
    = \sum_{\substack{j \in \mathcal{J}_K \\ \eta_j \neq 0}}
    \sqrt{\eta_j} \, \Delta \beta_m^j \, \tilde{e}_j \quad
    \Prob\text{-a.s.}
\end{equation*}
with $\Delta \beta_m^j = \beta_{t_{m+1}}^j-\beta_{t_m}^j$
$\Prob$-a.s.\ for $m \in \{0, \ldots, M-1\}$, $j\in\mathcal{J}_K$.
We assume commutativity as stated in Assumption
(A3), which allows us to rewrite
\begin{align}\label{DoubleInt}
    &e^{A(T-t)} \int_{t}^{T} B'(X_t) \Big( \int_{t}^s B(X_t)
    \, \mathrm{d}W_r^K \Big) \, \mathrm{d}W_s^K \nonumber \\
    &\quad = e^{A(T-t)} \Big( \frac{1}{2} B'(X_t) \big( B(X_t) (W_T^K -
    W_t^K), (W_T^K-W_t^K) \big) - \frac{T-t}{2} \sum_{\substack{j \in
    \mathcal{J}_K \\ \eta_j \neq 0}} \eta_j \, B'(X_t) \big( B(X_t)
    \tilde{e}_j, \tilde{e}_j \big) \Big) 
\end{align}
for $t \in [0,T]$ such that the iterated stochastic integral can be split into two
parts and simulation becomes straightforward, see~\cite{MR3320928}
for a proof. 

%
For some arbitrarily fixed $N$, $K$, and $M$, let $(\YY_m)_{0 \leq m
\leq M}$ with $\mathcal{F}_{t_m}$-$\mathcal{B}(H)$-measurable random
variables $\YY_m \colon \Omega \to H_N$ denote the discrete time
approximation process for $(X_{t_m})_{0 \leq m \leq M}$.
Now, we introduce a scheme which does not employ the derivative of
$B$ and therefore allows for a more efficient application to a broader
class of SPDEs than the Milstein scheme proposed in~\cite{MR3320928}. 
The main ingredient for the reduction of the computational cost is to
apply a specially tailored approximation of the derivative 
of the operator $B$. The crucial point is to avoid the use of any 
bilinear operators or their naive approximation that would boost the 
computational cost. Roughly speaking, the idea of discretizing the nonlinear 
operator $B'(Y)$ for any $Y \in H_{\beta}$ using standard difference 
quotients in each direction of the orthonormal basis
\begin{align*}
    B'(Y) \big( {e_k}, \tilde{e}_j \big)
    &\approx \frac{1}{h} \big( B(Y + h {e}_k) - B(Y) \big) \tilde{e}_j
\end{align*}
for all $k \in \mathcal{I}_N$
would result in $N+1$ necessary evaluations of the nonlinear operator $B$.
This is not efficient as $N$ is not a fixed number but has to increase
for higher precision in the infinite dimensional case.
Therefore, instead of first approximating the operator $B'(Y)$ itself
and then applying the approximate operator to some arguments $(u,\tilde{e}_j)$ in order 
to calculate $B'(Y)(u,\tilde{e}_j)$, a much more efficient idea is to directly approximate
the value $B'(Y)(u,\tilde{e}_j)$ by 
\begin{align*}
    B'(Y) \big( u, \tilde{e}_j \big)
    &\approx \frac{1}{h} \big( B(Y + h u) - B(Y) \big) \tilde{e}_j ,
\end{align*}
especially if only one fixed evaluation of $B'(Y)(\cdot,\tilde{e}_j)$ is needed.
The crucial point is that it is relatively cheap to directly approximate directional
derivatives by finite differences. Here, only two evaluations of 
the nonlinear operator $B$ are necessary, independent of the dimension $N$. 

%
Following ideas for ordinary SDEs in
\cite{MR2669396}, we propose a scheme which is characterized by
moving one of the sums into the argument. Thereby, fewer function
evaluations are necessary which results in a higher effective order of
convergence. For SPDE~\eqref{SPDE} with commutative noise \eqref{Comm} and
some arbitrarily fixed $N$, $K$, and $M$, we define
the enhanced \EDFM ($\ESRK$) as $\YY_0 = P_N \xi$ and
\begin{equation} \label{ESRK-scheme-orig}
    \begin{split}
    \YY_{m+1} &= P_N \Big( e^{Ah} \Big( \YY_m + h F(\YY_m) +
    B(\YY_m) \Delta \WW_m \\
    &\quad + \frac{1}{\sqrt{h}} \Big( B \Big(\YY_m + \frac{1}{2}
    \sqrt{h} \, P_N B(\YY_m) \Delta \WW_m
    \Big) - B(\YY_m) \Big) \Delta \WW_m
    \\
    &\quad + \sum_{\substack{j \in \mathcal{J}_K \\ \eta_j \neq 0}}
    \bar{B}(\YY_m,h,j)
    \Big) \Big)
    \end{split}
\end{equation}
for $m\in\{0, \ldots, M-1\}$
with $\bar{B}$ given by
\begin{equation} \label{ESRK-scheme-orig-bar-B}
  \bar{B}(\YY_m,h,j) = B \Big(\YY_m - \frac{h}{2} \sqrt{\eta_j} \, P_N B(\YY_m)
  \tilde{e}_j \Big) \sqrt{\eta_j} \, \tilde{e}_j - B(\YY_m) \sqrt{\eta_j} \, 
  \tilde{e}_j .
\end{equation}
It is important to note that the proposed \EDFM uses a special
approximation of the derivative in the original Milstein scheme
which turns out to be very efficient. In particular, approximating the
derivative in the way it is done in the enhanced \EDFM does not
influence the error estimate significantly. Apart from constants, it
can be proved to be the same as for the Milstein scheme. The main
result of this article is given as follows:
\begin{thm} \label{MainTh}
    Let Assumptions (A1)--(A4) be fulfilled. Then, there exists a
    constant $C \in (0,\infty)$ independent of $N$, $K$, and $M$ such
    that for $(\YY_m)_{0 \leq m \leq M}$, defined by the enhanced \EDFM in
    \eqref{ESRK-scheme-orig}--\eqref{ESRK-scheme-orig-bar-B}, it holds that
    \begin{equation*}
    \Big( \Erw \Big[ \big\| X_{t_m} - \YY_m \big\|_H^2
    \Big] \Big)^{\frac{1}{2}}
    \leq C \Big( \Big( \inf_{i \in \mathcal{I} \setminus
    \mathcal{I}_N} \lambda_i \Big)^{-\gamma}
    + \Big( \sup_{j \in \mathcal{J} \setminus \mathcal{J}_K}
    \eta_j \Big)^{\alpha} + M^{-\min(2(\gamma-\beta),\gamma)}
    \Big)
    \end{equation*}
    for all $m \in \{0, \ldots, M\}$ and all $N,K,M \in
    \mathbb{N}$.
\end{thm}
For the proof of Theorem~\ref{MainTh}, we refer the reader to Section
\ref{Proof}. 

%
Thus, under very similar Assumptions (A1)--(A4) as for the Milstein
scheme in \cite{MR3320928} it is possible to prove the same order of
convergence for the enhanced \EDFM. Moreover, as for the Milstein scheme, it
is straightforward to approximate the exponential term
$e^{At}$ by, e.g., $(I-At)^{-1}$, $t \in [0,T]$, see \cite{MR3248844}.
\section{Computational cost and effective order of convergence}
\label{CC}
Convergence results where the order of convergence depends directly
on the sets $\mathcal{I}_N$, $\mathcal{J}_K$ and on the parameter
$M$ like in Theorem~\ref{MainTh} are important to understand the
dependence of the error on the dimensionality of the approximation
spaces.
However, in order to judge the quality of an algorithm, we are
mainly interested in its error and cost. That is why it is important
to consider the order of convergence with respect to the
computational cost, that is, errors versus computational cost, which
we call the effective order of convergence, see also
\cite{MR2669396}. Since measured computation time may depend on the
implementation of an algorithm, an established theoretical cost
model as in \cite{MR1144521} is applied to be more objective. 
\subsection{A computational cost model}
Let $V$ be a real vector space. If $v \in V$ is part of the
considered problem to be solved, then an algorithm needs some
information about $v$ which can be seen as a call of an oracle or of a black
box. As (linear) information we consider the evaluation of any
(linear) functional $\phi \colon V \to \mathbb{R}$ and denote the
space of such functionals as $V^*$. Clearly, evaluating $\phi \in
V^*$ produces some computational cost, say $\costs(\phi)=c>0$.
Typically, each arithmetic operation or evaluation of sine, cosine,
the exponential function etc.\ produces cost of one unit whereas the
evaluation of a functional $\phi$ produces cost $c \gg 1$. Assuming
$c \gg 1$, the informational cost dominates the cost for
arithmetic operations in the algorithm. That is why we concentrate
on the cost for evaluating functionals $\phi \in V^*$, see also, for example,
\cite{MR1144521}. Typical examples in case of a Hilbert space
$V$ are $\phi_i(v) = \langle v, u_i \rangle_V$ for some $u_i \in V$,
$i\in\{1, \ldots, n\}$, $n\in\mathbb{N}$ with $\costs(\phi_1, \ldots, \phi_n) = c n$.
Moreover, if $V$ is the space of mappings $f \colon H \to
\mathbb{R}$, then one can consider the Dirac functional $\delta_x \in
V^*$ with $\delta_x f = f(x)$ for some $x \in H$. So, for $x_1,
\ldots, x_n \in H$ one can get the function evaluations $f(x_1),
\ldots, f(x_n)$ with $\costs(\delta_{x_1}, \ldots, \delta_{x_n}) = c
n$. In addition, we assume that each independent realization of an
$N(0,1)$-distributed random variable can be simulated with cost one.

%
Assume that, e.g., $|\mathcal{I}_N|=N$, $|\mathcal{J}_K|=K$, and that
$\eta_j \neq 0$ for all $j \in \mathcal{J}_K$  and all $K,N \in
\mathbb{N}$ which is the worst case for the computational effort.
For an implementation of the considered algorithms, it is usual to
identify $H_N$ by $\mathbb{R}^N$ applying the natural isomorphism
$\pi \colon H_N \to \mathbb{R}^N$ with $\pi(v) = (\langle v,e_i
\rangle)_{1 \leq i \leq N}$ for $v \in H_N$ and, analogously, we
identify $U_K$ by $\mathbb{R}^K$. Let $y, v \in H_N$, $u \in U_K$,
$L(H,E)_{N} = \{ T\arrowvert_{H_N} : \ T \in L(H,E)\}$ for some
vector space $E$ and let
$L_{HS}(U,H)_{K,N} = \{ P_N T\arrowvert_{U_K} : \ T \in
L_{HS}(U,H)\}$. Then, we obtain the following computational costs:
\begin{enumerate}[i)]
\item
One evaluation of the mapping $P_N \circ F \colon H \to H_N$ with 
\[
    P_N F(y) = \sum_{i \in \mathcal{I}_N} \langle F(y), e_i \rangle_H \, e_i
\] 
is determined by the functionals $\langle F(y), e_i \rangle_H$ for $i
\in \mathcal{I}_N$ with $\costs(P_N F(y)) = c N$.
\item
Evaluating $P_N \circ
B(\cdot)\arrowvert_{U_K} \colon H \to L_{HS}(U,H)_{K,N}$
with 
\[
    P_N B(y)u = \sum_{i \in \mathcal{I}_N} \sum_{j \in
    \mathcal{J}_K} \langle B(y) \tilde{e}_j, e_i \rangle_H \, \langle u,
    \tilde{e}_j \rangle_U \, e_i
\] 
needs the
evaluation of the functionals $\langle B(y) \tilde{e}_j, e_i
\rangle_H$ for $i \in \mathcal{I}_N$ and $j \in \mathcal{J}_K$
with $\costs(P_N \circ B(y)\arrowvert_{U_K}) = c N K$.
\item
Finally, observe that for $P_N \circ
B'(\cdot)(\cdot,\cdot)\arrowvert_{H_N,U_K} \colon H \to
L(H,L_{HS}(U,H)_{K,N})_N$ with
\[
    P_N \big( (B'(y)v)u \big) = \sum_{k,l \in
    \mathcal{I}_N} \sum_{j \in \mathcal{J}_K} \langle (B'(y) e_k)
    \tilde{e}_j, e_l \rangle_H \, \langle v, e_k \rangle_H \, \langle u,
    \tilde{e}_j \rangle_U \, e_l
\]
it follows that $\costs(P_N \circ
B'(y)(\cdot,\cdot)\arrowvert_{H_N,U_K}) = c N^2 K$ since the
functionals $\langle (B'(y) e_k) \tilde{e}_j, e_l \rangle_H$ have to
be evaluated for all $k,l \in \mathcal{I}_N$ and $j \in
\mathcal{J}_K$.
\end{enumerate}
Provided that for $T \in L_{HS}(U,H)_{K,N}$ all functionals $\langle T
\tilde{e}_j, e_i \rangle_H$ and $\langle u, \tilde{e}_j \rangle_U$
are known for $i \in \mathcal{I}_N$ and $j \in \mathcal{J}_K$, then
$Tu = \sum_{i \in \mathcal{I}_N} \sum_{j \in \mathcal{J}_K} \langle
u, \tilde{e}_j \rangle_U \, \langle T \tilde{e}_j, e_i \rangle_H \,
e_i$ and the calculation of $\pi(Tu)_i = \langle Tu, e_i \rangle_H$
needs $K$ multiplications and $K-1$ summations for each $i \in
\mathcal{I}_N$ and thus $\costs(\pi(Tu)) = 2NK-1$. Analogously, for
$T \in L(H,L_{HS}(U,H)_{K,N})_N$, it follows that $\costs(\pi((Tv)u))
= 3 N^2 K -1$ provided that the functionals $\langle (T e_k)
\tilde{e}_j, e_l \rangle_H$, $\langle v, e_k \rangle_H$, and $\langle
u, \tilde{e}_j \rangle_U$ are known for all $k,l \in \mathcal{I}_N$
and $j \in \mathcal{J}_K$. 

%
In order to assess the usefulness and efficiency of the proposed commutative \EDFM $\ESRK$
\eqref{ESRK-scheme-orig}, we compare it to the Milstein scheme \eqref{Milstein},
denoted as $\MIL$, see
\cite{MR3320928}, the linear implicit Euler
scheme considered in, e.g., \cite{MR1825100,MR2136207} and denoted as
$\LIE$, and the exponential Euler scheme, denoted as $\EES$, see
\cite{MR2856611,MR3047942}, for example. Here, we want to
mention that the Runge-Kutta type scheme
proposed in \cite{MR3011387} is not taken into account because it
cannot be applied to the general class of SPDEs under
consideration. 

%
The computational costs of the Milstein scheme $\MIL$ for each time step are
determined by one evaluation of $P_N \circ F$, $P_N \circ
B(\cdot)\arrowvert_{U_K}$, and one evaluation of $P_N \circ
B'(\cdot)\arrowvert_{H_N,U_K}$. In addition, the following linear and
bilinear operators have to be applied: One application of $P_N \circ
B(\YY_m) \arrowvert_{U_K} \in L(U,H)_{K,N}$ (here, calculating $P_N
B(\YY_m) \tilde{e}_j$ for a basis element $\tilde{e}_j \in U_K$ is
free because it is the $j$th column of the matrix
representation $P_N B(\YY_m)\arrowvert_{U_K} = \big(b_{i,j}(\YY_m)
\big)_{i \in \mathcal{I}_N, j \in \mathcal{J}_K}$ with
$b_{i,j}(\YY_m)=\langle B(\YY_m) \tilde{e}_j, e_i \rangle_H$ which
is already determined), one application of the bilinear operator
$P_N \circ B'(\YY_m)\arrowvert_{H_N,U_K} \in
L(H,L_{HS}(U,H)_{K,N})_N$, one application of an operator of type
$P_N \circ B'(\YY_m)(v)\arrowvert_{U_K} \in L_{HS}(U,H)_{K,N}$ (here
again the application of the operator to a basis element $\tilde{e}_j
\in U_K$ is free), and one application of $P_N \circ e^{Ah}\arrowvert_{H_N}
\colon H_N \to H_N$. In addition, $K$ independent realizations of
$N(0,1)$-distributed random variables have to be simulated. Summing
up, the computational cost for the approximation of one realization
of the solution $X_T$ by the Milstein scheme 
is $\costs(\MIL(N,K,M)) = \Oo(N^2 K M)$.

%
%
The introduced \EDFM $\ESRK$
needs for each time step the evaluation of $P_N \circ F$, two times
$P_N \circ B(\cdot)\arrowvert_{U_K}$, and the evaluation of
\begin{equation} \label{ESRK-comput-costs-eqn01}
    \sum_{\substack{j \in \mathcal{J}_K \\ \eta_j \neq 0}}
    P_N B \Big(\YY_m - \frac{h}{2} \sqrt{\eta_j} \, P_N B(\YY_m)
    \tilde{e}_j \Big) \sqrt{\eta_j} \, \tilde{e}_j .
\end{equation}
Observe that for each $j \in \mathcal{J}_K$ the calculation of $P_N
B \big(\YY_m -\frac{h}{2} \sqrt{\eta_j} \, P_N B(\YY_m) \tilde{e}_j \big)
\sqrt{\eta_j} \, \tilde{e}_j$ results in the computation of the functionals
$\phi_i^j=\langle B \big(\YY_m - \frac{h}{2} \sqrt{\eta_j} \, P_N B(\YY_m)
\tilde{e}_j \big) \sqrt{\eta_j} \, \tilde{e}_j, e_i \rangle_H$ for $i \in
\mathcal{I}_N$ with $\costs(\phi_1^j, \ldots, \phi_N^j)=cN$.
Therefore, the evaluation of \eqref{ESRK-comput-costs-eqn01} can be
done with cost $c N K$.
In addition, the linear operators $P_N \circ e^{Ah}\arrowvert_{H_N}
\colon H_N \to H_N$, $P_N \circ B(\YY_m)\arrowvert_{U_K} \in
L(U,H)_{K,N}$ (note again that calculating $P_N B(\YY_m) \tilde{e}_j$ for a
basis $\tilde{e}_j \in U_K$ is free), and $P_N \circ B(\YY_m +
\frac{1}{2} \sqrt{h} \, P_N B(\YY_m) \Delta \WW_m)\arrowvert_{U_K}
\in L(U,H)_{K,N}$ have to be applied. Finally, $K$ independent
realizations of $N(0,1)$-distributed random variables have to be
simulated in each step. Thus, the total computational cost for $M$ time
steps of the enhanced \EDFM for the approximation of one realization of $X_T$
is $\costs(\ESRK(N,K,M))=\Oo(N K M)$. 

%
Although both schemes $\MIL$ and $\ESRK$ have the same order of
convergence with respect to the dimensions $N$, $K$, and $M$ of the
finite-dimensional subspaces, see Theorem~\ref{MainTh}, their 
computational costs depend on
these parameters with different powers, see Table~\ref{funcEval-1}. 
In contrast to the setting
of finite dimensional SDEs with fixed dimensions, for SPDEs on
infinite dimensional spaces, the dimensions of the finite dimensional
projection subspaces have to increase for the accuracy of the
approximation to increase. Thus, the computational costs
depend not only on $M$ but also on the variable dimensions $N$ and $K$.
In particular, the reduction of the power of $N$ in the computational
cost results in an improvement of the order of convergence if one
considers errors versus computational cost. Here, we want to point
out that computational cost of order $\Oo(N K M)$ is in some sense
optimal within the class of one-step approximation methods because
in general one evaluation of the nonlinear operator
$P_N \circ B(\cdot)\arrowvert_{U_K}$ already produces a computational cost of
order $\Oo(N K)$ for each time step. Further, the linear implicit
Euler scheme $\LIE$ as well as the exponential Euler scheme
$\EES$ have computational cost $\costs(\LIE(N,K,M)) =
\costs(\EES(N,K,M)) = \Oo(N K M)$, which is of the same order as that for
the introduced \EDFM $\ESRK$. However, compared to the scheme 
$\ESRK$ the schemes $\LIE$ and $\EES$ attain, in general, significantly 
lower orders of convergence if the corresponding errors are 
considered. 
\begin{table}[tbp] 
\begin{small}
\begin{center}
  \renewcommand{\arraystretch}{1.4}
  \begin{tabular}{|l|c|c|c|c|}
    \hline
    & \multicolumn{3}{c|}{Computational cost for evaluation of} & 
    \\
    Scheme & $\quad \ P_N F(\cdot)\arrowvert_{H_N} \ \quad$ &
    $\quad \ P_N B(\cdot)\arrowvert_{U_K} \ \quad$ &
    $ \ P_N B'(\cdot)\arrowvert_{H_N,U_K} \ $ & \# of $N(0,1)$  r.\ v.\  
    \\
    \hline \hline
    $\MIL$ & $N$ & $KN$ & $KN^2$ & $K$ 
    \\ \hline
    $\LIE$ & $N$ & $KN$ & $-$ & $K$ 
    \\ \hline
    $\EES$ & $N$ & $KN$ & $-$ & $K$ 
    \\ \hline
    $\ESRK$ & $N$ & $3KN$ & $-$ & $K$ 
    \\ \hline
  \end{tabular}
  \renewcommand{\arraystretch}{1.0}
\end{center}
\end{small}
\caption{Number of real-valued nonlinear function evaluations and
independent $N(0,1)$-distributed random variables for each time step.}
\label{funcEval-1}
\end{table}
\subsection{Effective order of convergence}
\label{SubSec:Effecive-Order-conv}
Next, the effective order of convergence is determined for the
schemes under consideration. 
First, one has to solve an optimization
problem for the optimal choice of the parameters $N$, $K$, and $M$
such that the error is minimized under the constraint that the
computational cost is arbitrarily fixed. Here, one needs to know
about the relationship between $\inf_{i \in \mathcal{I} \setminus
\mathcal{I}_N} \lambda_i$ and $\dim(H_N)$ as well as between
$\sup_{j \in \mathcal{J} \setminus \mathcal{J}_K} \eta_j$ and
$\dim(U_K)$ for any $N, K \in \mathbb{N}$. Therefore, as an example,
we assume that $\inf_{i \in \mathcal{I} \setminus \mathcal{I}_N}
\lambda_i = \Oo(N^{\rho_A})$ and $\sup_{j \in \mathcal{J} \setminus
\mathcal{J}_K} \eta_j = \Oo(K^{-\rho_Q})$ for some $\rho_A, \rho_Q
>0$. 
Moreover, similar results can be obtained
under different assumptions as well. Then, for some $q > 0$ depending 
on the scheme under consideration, we investigate the error
\begin{equation} \label{Error-general-ord1}
    \err(\SCHEME(N,K,M)) := \Big( \Erw \Big[ \big\| X_T - \YY_M
    \big\|_H^2 \Big] \Big)^{\frac{1}{2}}
    = \Oo \big( N^{-\gamma \rho_A} + K^{-\alpha \rho_Q} + M^{-q}
    \big)
\end{equation}
and minimize $\err(\SCHEME(N,K,M))$ under the constraint that
for the computational cost it holds that
$\costs(\SCHEME(N,K,M))=\ccosts$ for some arbitrary constant
$\ccosts>0$. 

%
For the Milstein scheme $\MIL$ with $q = \min( 2(\gamma - \beta),
\gamma)$ and $\costs(\MIL(N,K,M)) = \Oo(N^2 K M)$, we obtain as an
optimal choice
\begin{align*}
    N &= \Oo \Big( \ccosts^{\frac{\alpha \rho_Q q}{
    (2 \alpha \rho_Q + \gamma \rho_A) q
    + \alpha \gamma \rho_A \rho_Q}} \Big), \quad \quad
    K = \Oo \Big( \ccosts^{\frac{\gamma \rho_A q}{
    (2 \alpha \rho_Q + \gamma \rho_A) q
    + \alpha \gamma \rho_A \rho_Q}} \Big), \quad \quad
    M = \Oo \Big( \ccosts^{\frac{\alpha \gamma \rho_A \rho_Q}{
    (2 \alpha \rho_Q + \gamma \rho_A) q
    + \alpha \gamma \rho_A \rho_Q}} \Big),
\end{align*}
which balances the three summands on the right-hand side of
\eqref{Error-general-ord1}. As a result, the effective order
of convergence for error versus computational cost of the
Milstein scheme is
\begin{equation} \label{effective-order-MIL}
    \err(\MIL(N,K,M)) = \Oo \Big( \ccosts^{-\frac{\alpha \gamma \rho_A \rho_Q
    \min(2(\gamma-\beta),\gamma)
    }{ (2 \alpha \rho_Q + \gamma \rho_A)
    \min(2(\gamma-\beta),\gamma) 
    + \alpha \gamma \rho_A \rho_Q}} \Big),
\end{equation}
which is optimal for the Milstein scheme \eqref{Milstein}. 

%
Solving the corresponding optimization problem for the \EDFM
$\ESRK$ with $q = \min( 2(\gamma -\beta), \gamma)$ 
and reduced computational cost given as
$\costs(\ESRK(N,K,M)) = \Oo(N K M)$ results in the optimal choice
\begin{align*}
    N = \Oo \Big( \ccosts^{\frac{\alpha \rho_Q q}{
    (\alpha \rho_Q + \gamma \rho_A) q
    + \alpha \gamma \rho_A \rho_Q}} \Big), \quad \quad
    K = \Oo \Big( \ccosts^{\frac{\gamma \rho_A q}{
    (\alpha \rho_Q + \gamma \rho_A) q
    + \alpha \gamma \rho_A \rho_Q}} \Big), \quad \quad
    M = \Oo \Big( \ccosts^{\frac{\alpha \gamma \rho_A \rho_Q}{
    (\alpha \rho_Q + \gamma \rho_A) q
    + \alpha \gamma \rho_A \rho_Q}} \Big) .
\end{align*}
Then, the effective order of convergence is given by
\begin{equation} \label{effective-order-ESRK}
    \err(\ESRK(N,K,M)) = \Oo \Big( \ccosts^{-\frac{\alpha \gamma \rho_A \rho_Q
    \min(2(\gamma-\beta),\gamma)
    }{ (\alpha \rho_Q + \gamma \rho_A)
    \min(2(\gamma-\beta),\gamma) 
    + \alpha \gamma \rho_A \rho_Q}} \Big),
\end{equation}
which is optimal for the \EDFM \eqref{ESRK-scheme-orig}.

%
It is obvious that the order of the enhanced \EDFM
$\ESRK$ is higher than the order of the Milstein
scheme $\MIL$ given in \eqref{effective-order-MIL}. That means that
for some arbitrarily prescribed amount of computational cost (or
computing time) $\ccosts$, the minimal possible error
$\err(\ESRK(N,K,M))$ of the \EDFM
$\ESRK$ decreases with some higher order than the minimal
possible error $\err(\MIL(N,K,M))$ of the Milstein scheme as
$\ccosts \to \infty$. 

%
For the linear implicit Euler scheme $\LIE$ and the exponential Euler scheme 
$\EES$, we obtain the same optimal expressions for $N$, $K$, and $M$ as for the \EDFM
$\ESRK$ with $q=\min( 2(\gamma -\beta), \gamma,\frac{1}{2})$, however. 
The effective orders of convergence of these schemes are
\begin{equation*} 
     \err(\EES(N,K,M)) = \err(\LIE(N,K,M)) = \Oo \Big( \ccosts^{-\frac{\alpha \gamma \rho_A \rho_Q
    \min( 2(\gamma -\beta), \gamma,\frac{1}{2})}{ (\alpha \rho_Q + \gamma \rho_A)
    \min( 2(\gamma -\beta), \gamma,\frac{1}{2}) + \alpha \gamma \rho_A \rho_Q}} \Big).
\end{equation*}
For the schemes $\LIE$ and $\EES$, the parameter $q > 0$ in \eqref{Error-general-ord1} is 
in general smaller than for the \EDFM $\ESRK$; i.e., here it holds that 
$q \leq  \min( 2(\gamma -\beta), \gamma)$,
which results in a lower effective order of convergence for the linear implicit Euler 
scheme $\LIE$ as well as the exponential Euler scheme $\EES$.
\subsection{The special case of pointwise multiplicative operators}
\label{SubSec:Effective-Order-Multiplicative-Operators}
For the special case of, for example, $H=U=L^2((0,1)^d, \mathbb{R})$ and
Nemytskij operators, where $F \colon H_{\beta} \to H$ is given by
$(F(v))(x)=f(x,v(x))$ and $B \colon H_{\beta} \to L_{HS}(U_0,H)$ is
given by $(B(v)u)(x) = b(x,v(x)) \cdot u(x)$ for some functions $f,
b \colon (0,1)^d \times \mathbb{R} \to \mathbb{R}$, $x \in (0,1)^d$,
$v \in H_{\beta}$, $\beta\in[0,1)$, $u \in U_0$, and some $d \in \mathbb{N}$, which is
the setting also treated in \cite{MR3320928} and exclusively in
\cite{MR3011387}, the Milstein scheme \eqref{Milstein} simplifies such that
the number of evaluations of the derivative is significantly
reduced.
Although, the scheme $\ESRK$ \eqref{ESRK-scheme-orig} 
combined with the choice of $\bar{B}$ in \eqref{ESRK-scheme-orig-bar-B}
is applicable in this special setting, we do not recommend using it.
In this case, the computational cost can be reduced by an alternative choice of 
$\bar{B}$ adapted to pointwise multiplicative operators. Therefore, we 
define the derivative-free multiplicative Milstein scheme ($\SESRK$) by
$\YY_0 = P_N \xi$ and
\begin{equation} \label{ESRK-scheme-Special}
    \begin{split}
    \YY_{m+1} &= P_N \Big( e^{Ah} \Big( \YY_m + h f( \cdot, \YY_m) +
    b(\cdot,\YY_m) \cdot \Delta \WW_m \\
    &\quad + \frac{1}{\sqrt{h}} \Big( b\Big(\cdot, \YY_m + \frac{1}{2}
    \sqrt{h} \, P_Nb(\cdot, \YY_m) \cdot \Delta \WW_m
    \Big) - b(\cdot, \YY_m) \Big) \cdot \Delta \WW_m
    \\
    &\quad + \sum_{\substack{j \in \mathcal{J}_K \\ \eta_j \neq 0}}
    \bar{B}(\YY_m,h,j)
    \Big) \Big)
    \end{split}
\end{equation}
with $\bar{B}$ now given by
\begin{equation} \label{ESRK-scheme-Special-bar-B}
  \bar{B}(\YY_m,h,j) = \Big( b\Big( \cdot, \YY_m - \frac{h}{2}
  P_Nb(\cdot, \YY_m) \Big) - b(\cdot, \YY_m) \Big) \eta_j \tilde{e}_j^2
\end{equation}
for all $m\in\{0, \ldots, M-1\}$, $j\in\mathcal{J}_K$. 
We want to emphasize that
the first part \eqref{ESRK-scheme-Special} of the 
scheme $\SESRK$ coincides with \eqref{ESRK-scheme-orig}
in this special setting whereas $\bar{B}$ is chosen differently.
\begin{cor} \label{Corollary-MainTh-pointwise}
  Let the setting of Section~\ref{SubSec:Effective-Order-Multiplicative-Operators} 
  be given and let Assumptions (A1)--(A4) be fulfilled. Then, Theorem~\ref{MainTh} 
  remains valid for the derivative-free multiplicative
  Milstein scheme $(\SESRK)$ in \eqref{ESRK-scheme-Special}--\eqref{ESRK-scheme-Special-bar-B}.
\end{cor}
For the proof of Corollary~\ref{Corollary-MainTh-pointwise}, we refer the reader to 
the proof of Theorem~\ref{MainTh} in Section~\ref{Proof} with corresponding comments.

For the implementation of scheme \eqref{ESRK-scheme-Special}, one has
to compute expressions of the form
\begin{align*}
  P_N \big( f(\cdot, \YY_m( \cdot ) ) \big) &= \sum_{i \in \mathcal{I}_N}
  \langle f(\cdot, \YY_m(\cdot)), e_i \rangle_H \, e_i 
  = \sum_{i \in \mathcal{I}_N} \Big( \int_{(0,1)^d} f(x, \YY_m(x)) \, e_i(x)
  \, \mathrm{d}x \Big) \, e_i,
\end{align*}
where each integral can be approximated by, e.g.,\ a standard
quadrature formula based on a spatial discretization of $(0,1)^d$.
However, the spatial discretization is not in our focus as we
restrict our considerations to the time discretization with a
general projector $P_N$ independent of the spatial discretization.
Then, the computational costs are determined by the calculation of
the functionals $\langle f(\cdot, \YY_m(\cdot)), e_i \rangle_H$ for
$i \in \mathcal{I}_N$. Thus, it holds that $\costs \big( P_N \big( f(\cdot, \YY_m(
\cdot ) ) \big) \big) = cN$. The same applies to the calculation of
$P_N \big( b(\cdot, \YY_m( \cdot ) ) \big)$, $P_N \big( b\big(\cdot,
\YY_m + \frac{1}{2} \sqrt{h} \, P_Nb(\cdot, \YY_m) \big) \big)$, and
$P_N \big( b\big( \cdot, \YY_m - \frac{h}{2} P_Nb(\cdot, \YY_m) \big)
\big)$. Further, the scheme $\SESRK$ makes use of $K$ independent
$N(0,1)$-distributed random variables. To sum up, the computational
cost for the calculation of one approximation of a realization of
$X_T$ with the multiplicative version of the \EDFM 
\eqref{ESRK-scheme-Special} in this
special setting is $\costs(\SESRK(N,K,M))= \Oo(NM+KM)$. 

%
In this setting, the effective order of convergence for the 
$\SESRK$ scheme can be determined
by minimizing the error $\err(\SESRK(N,K,M))$ under the constraint
that 
$\costs(\SESRK(N,K,M))=\ccosts$ is arbitrarily
fixed. Let $q=\min(2(\gamma-\beta),\gamma)$, then, a reasonable choice
is given by
\begin{align*}
    N = \Oo \Big( \ccosts^{\frac{\min(\gamma \rho_A, \alpha \rho_Q) q}{
    \gamma \rho_A (\min(\gamma \rho_A, \alpha \rho_Q) + q)}} \Big), \quad \quad
    K = \Oo \Big( \ccosts^{\frac{\min(\gamma \rho_A, \alpha \rho_Q) q}{
    \alpha \rho_Q (\min(\gamma \rho_A, \alpha \rho_Q) + q)}} \Big), \quad \quad
    M = \Oo \Big( \ccosts^{\frac{\min(\gamma \rho_A, \alpha \rho_Q)}{
    \min(\gamma \rho_A, \alpha \rho_Q) + q}} \Big),
\end{align*}
and the effective order of convergence for error versus
computational cost is 
\begin{equation} \label{effective-order-SESRK}
    \err(\SESRK(N,K,M)) = \Oo \Big( \ccosts^{-\frac{
    \min(\gamma \rho_A, \alpha \rho_Q) q}{
    \min(\gamma \rho_A, \alpha \rho_Q) + q}} \Big)
\end{equation}
for the \EDFM \eqref{ESRK-scheme-Special}.
This is the same order as for the Milstein scheme proposed in
\cite{MR3320928} and for the Runge-Kutta type scheme proposed in
\cite{MR3011387}. However, like for the Runge-Kutta type scheme in
\cite{MR3011387}, the advantage compared to the Milstein scheme is
that no derivative of $b$ has to be calculated. For the schemes $\EES$ and
$\LIE$, we obtain the same expressions for $N$, $K$, $M$, and the 
effective order of convergence as for the scheme $\SESRK$ --
however, with $q=\min(2(\gamma-\beta),\gamma,\frac{1}{2})$.
In the following, we do not restrict our analysis to this special
case of, e.g.,\ Nemytskij operators but allow for a broader class of
SPDEs.
\subsection{The special case of finite dimensional noise}\label{EffOrderFinite}
Consider the case of a $Q$-Wiener process $(W_t)_{t \in [0,T]}$ and an 
operator $Q \in L(U)$ with eigenvalues $\eta_j$ for $j \in \mathcal{J}$
such that $K:=|\{ \eta_j : \eta_j \neq 0, \ j \in \mathcal{J}\}| < \infty$.
Then, one can choose $\mathcal{J}_K=\{j \in \mathcal{J} : \eta_j \neq 0\}$
and there is no projection error if $(W_t)_{t \in [0,T]}$ is replaced
by $(W_t^K)_{t \in [0,T]}$. 
Assume that $\inf_{i \in \mathcal{I} \setminus \mathcal{I}_N}
\lambda_i = \Oo(N^{\rho_A})$ for some $\rho_A >0$. 
Then, for fixed $K$ and some $q \geq 0$, we investigate the error
\begin{equation} \label{Error-general-ord1-finite-K}
    \err(\SCHEME(N,M)) := \Big( \Erw \Big[ \big\| X_T - \YY_M
    \big\|_H^2 \Big] \Big)^{\frac{1}{2}}
    = \Oo \big( N^{-\gamma \rho_A} + M^{-q} \big)
\end{equation}
and minimize $\err(\SCHEME(N,M))$ under the constraint that
for the computational cost it holds that
$\costs(\SCHEME(N,M))=\ccosts$ for some arbitrary constant
$\ccosts>0$. 

%
Analogous considerations as in Section~\ref{SubSec:Effecive-Order-conv}
for the Milstein scheme $\MIL$ with $q = \min( 2(\gamma - \beta),
\gamma)$ and $\costs(\MIL(N,M)) = \Oo(N^2 M)$ yield as an
optimal choice
$N = \Oo \big( \ccosts^{\frac{q}{\gamma \rho_A + 2q}} \big)$
and 
$M = \Oo \big( \ccosts^{\frac{\gamma \rho_A}{ \gamma \rho_A + 2q}} \big)$
in order to balance the two summands in \eqref{Error-general-ord1-finite-K}. 
Then, the effective order of convergence for the Milstein scheme is
\begin{equation} \label{effective-order-MIL-finite-K}
    \err(\MIL(N,M)) = \Oo \Big( \ccosts^{-\frac{\gamma \rho_A
    q}{ \gamma \rho_A + 2
    q }} \Big)
\end{equation}
which is optimal for the Milstein scheme \eqref{Milstein} in this special case. 

%
For the enhanced \EDFM $\ESRK$ with $q = \min( 2(\gamma -
\beta), \gamma)$ and reduced computational cost
$\costs(\ESRK(N,M)) = \Oo(N M)$, the optimal choice is
$N = \Oo \big( \ccosts^{\frac{q}{\gamma \rho_A + q}} \big)$
and 
$M = \Oo \big( \ccosts^{\frac{\gamma \rho_A}{ \gamma \rho_A + q}} \big)$.
As a result of this, the effective order of convergence is
\begin{equation} \label{effective-order-ESRK-finite-K}
    \err(\ESRK(N,M)) = \Oo \Big( \ccosts^{-\frac{\gamma \rho_A
    q}{ \gamma \rho_A
    + q }} \Big),
\end{equation}
which is optimal for the \EDFM \eqref{ESRK-scheme-orig} and which 
is a higher order than for the Milstein scheme $\MIL$.

%
If in addition the operators are pointwise multiplicative as in 
Section~\ref{SubSec:Effective-Order-Multiplicative-Operators}, then 
the simplified \EDFM $\SESRK$ can be applied.
The computational cost for the \EDFM in this
special setting with some fixed $K$ is $\costs(\SESRK(N,M))= \Oo(NM+KM)$.
For $q=\min(2(\gamma-\beta),\gamma)$, a reasonable choice
is 
$N = \Oo \big( \ccosts^{\frac{q}{\gamma \rho_A + q}} \big)$
and
$M = \Oo \big( \ccosts^{\frac{\gamma \rho_A}{
\gamma \rho_A + q}} \big)$.
Then, the effective order of convergence results in
\begin{equation} \label{effective-order-SESRK-finite-K}
    \err(\SESRK(N,M)) = \Oo \Big( \ccosts^{-\frac{
    \gamma \rho_A q}{\gamma \rho_A 
    + q}} \Big)
\end{equation}
for the multiplicative \EDFM \eqref{ESRK-scheme-Special}.
As in Section~\ref{SubSec:Effective-Order-Multiplicative-Operators}, 
this is the same order as for the Milstein scheme in
\cite{MR3320928} and for the Runge-Kutta type scheme in
\cite{MR3011387}. 
Again, as for the Runge-Kutta type scheme in
\cite{MR3011387}, the advantage compared to the Milstein scheme is
that no derivative of $b$ has to be calculated for the \EDFM
\eqref{ESRK-scheme-Special}.

%
Independent of the operators being pointwise multiplicative,
for the schemes $\EES$ and $\LIE$ we get the same formulas for $N$, $M$, and the
effective order of convergence is the same as that for the $\ESRK$ scheme with 
$q=\min(2(\gamma-\beta),\gamma,\frac{1}{2})$, however.  
Here, we want to point out that in this case of finite dimensional 
noise one can apply the \EDFM $\ESRK$ \eqref{ESRK-scheme-orig} instead 
of the scheme $\SESRK$ \eqref{ESRK-scheme-Special} since both schemes 
achieve exactly the same effective order of convergence.
\section{Numerical tests}
In order to illustrate the benefits of the enhanced \EDFM, it is compared to the Milstein scheme
proposed in \cite{MR3320928}, the linear implicit Euler scheme, the exponential Euler
scheme, and the Runge-Kutta type scheme in \cite{MR3011387}.
First, we show that the analytical solution of an SPDE with a pointwise
multiplicative operator is 
approximated with the expected order. Then, we pick up an example from 
\cite{MR3320928} and \cite{MR3011387} to show that the \EDFM converges with the same order
as the Milstein scheme in this special case.
In the main part of this section, we illustrate the superiority of the introduced \EDFM compared to the 
other schemes in the more general setting where we are not restricted to the case of 
pointwise multiplicative 
operators. We set $\mathcal{I} =\mathcal{J} = \mathbb{N}$, $\mathcal{I}_N = \{1,\ldots,N\}$, 
and $\mathcal{J}_K = \{1,\ldots,K\}$ in all the 
examples analyzed in the following sections, if not stated otherwise.
\subsection{Test example with exact solution}
First, we consider an SPDE with a pointwise multiplicative
operator and finite dimensional noise on the spaces $H = L^2((0,1),\mathbb{R})$ and
$U = \mathbb{R}$. The SPDE is given by 
\begin{alignat}{5} \label{Num-SPDE-01}
    \mathrm{d}X_t  &= (\Delta X_t) \, \mathrm{d}t + X_t \, \mathrm{d}\beta_t,& \qquad & t>0, \nonumber\\
    X_0(x) &= \sqrt{2}\sum_{n\in\mathbb{N}} n^{-2} \sin(n\pi x),&\quad &x \in (0,1), \\
    X_t(0) &= X_t(1) = 0,& \quad&t\geq 0 \nonumber
\end{alignat}
with a scalar Brownian motion $(\beta_t)_{t \geq 0}$. The exact 
solution can be calculated as
\begin{equation} \label{SolMulti}
    X_t(x) = \sqrt{2} \sum_{n\in\mathbb{N}} n^{-2} e^{-(n^2\pi^2+\frac{1}{2})t 
    + \beta_t} \sin(n\pi x)
\end{equation}
for all $x\in(0,1)$, $t\geq 0$, which is a strong solution of \eqref{Num-SPDE-01}. 
Since SPDE~\eqref{Num-SPDE-01} belongs to the special case of pointwise multiplicative 
operators with respect to the $Q$-Wiener process, the customized schemes $\MIL$, $\SESRK$,
and the Runge-Kutta type scheme in \cite{MR3011387} ($\RKS$) can be applied.
Further, there is no truncation error from the approximation 
of the $Q$-Wiener process for $K=1$. 

%
We determine the parameters introduced in (A1)--(A4) and Section \ref{EffOrderFinite}. 
For $A = \Delta$, we get $\rho_A=2$ and obtain $\delta \in(0,\frac{1}{2})$ 
by the arguments in \cite{MR2852200}. We choose $\delta$ to be maximal, $\beta =0$,
and obtain $\gamma \in [\frac{1}{2},1)$ by Theorem 
\ref{MainTh}. For the schemes $\MIL$, $\SESRK$, and $\RKS$, we choose $q = \gamma = 1-\varepsilon$ 
for any $\varepsilon >0$. On the other hand, it holds that $q = \frac{1}{2}$ for $\EES$
and $\LIE$. The parameters $\rho_Q$ and $\alpha$
do not influence the order of convergence as in this setting there is no error from the
approximation of the $Q$-Wiener process, see \eqref{effective-order-SESRK-finite-K}.
Therefore, we expect the numerical approximations to converge with 
the effective order $\err(\MIL(N,K,M)) = \err(\RKS(N,K,M)) = 
\err(\SESRK(N,K,M)) = \Oo \big( \ccosts^{-\frac{2}{3}+\varepsilon}\big)$ 
and the effective order $\err(\LIE(N,K,M)) = \err(\EES(N,K,M)) 
= \Oo \big( \ccosts^{-\frac{2}{5}+\varepsilon}\big)$ 
in case of the linear implicit or exponential Euler scheme if we compare 
error versus computational cost. 

%
\begin{figure}[tbp]
\begin{center}
  {\includegraphics[height = 6cm, width = 0.5\textwidth]{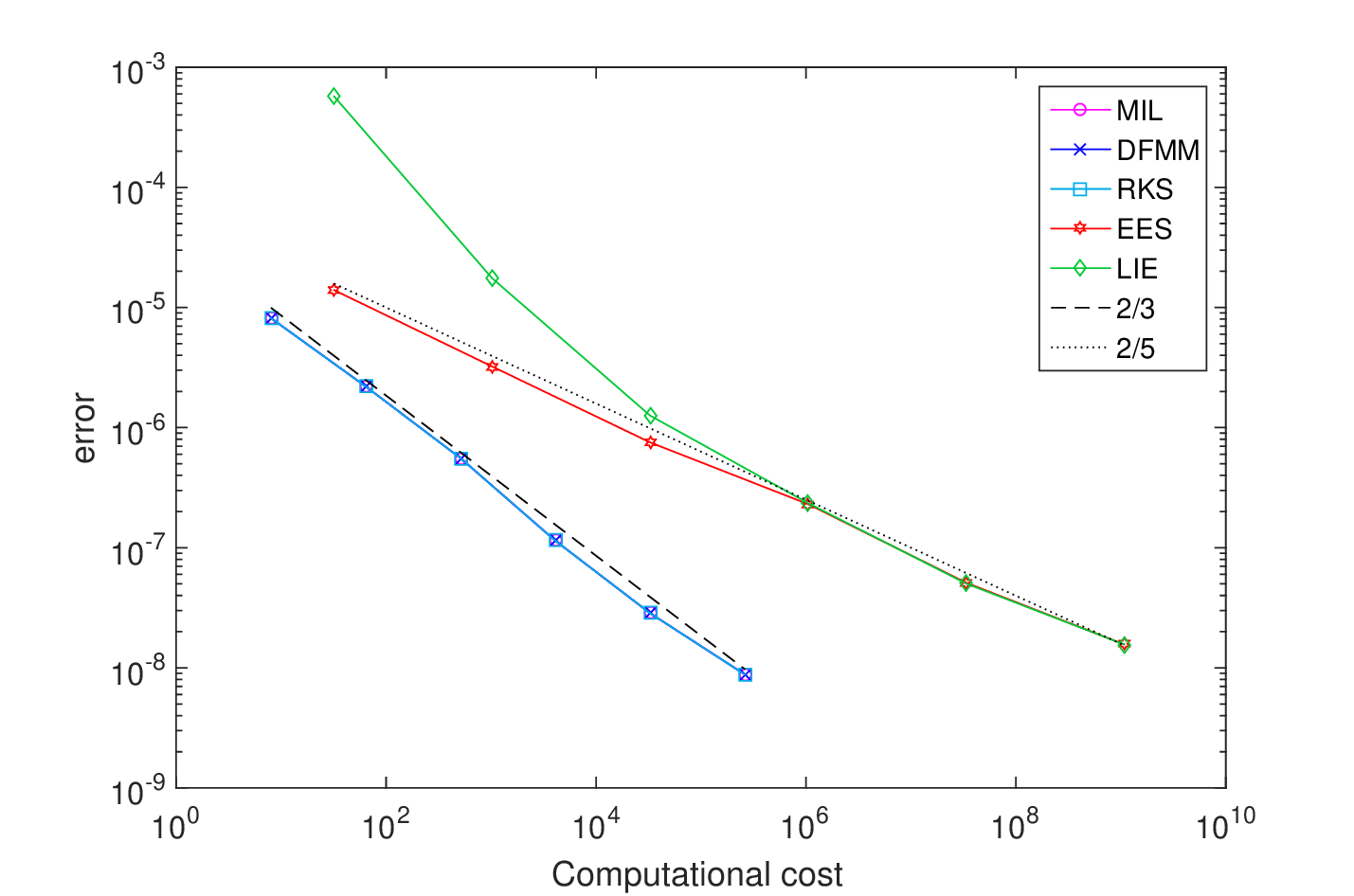}}
  \caption{Error versus computational cost for $N \in \{ 2, 4, 8, 16, 32, 64 \}$ and 300 paths 
  for the pointwise multiplicative SPDE \eqref{Num-SPDE-01}
  based on the exact solution in log-log scale.}
\label{Anal}
\end{center}
\end{figure}
%
For the numerical simulations, 500 paths are calculated to determine the 
error \eqref{Error-general-ord1-finite-K} at time $T=1$ for $N \in \{ 2,2^2,\ldots, 2^6 \}$, respectively. 
For the $\EES$ and the $\LIE$ schemes, we employ the parameter constellation $M = N^4$ 
with computational cost $\ccosts = \Oo \big(N^{5}\big)$, 
whereas for the $\MIL$, $\SESRK$, and $\RKS$ schemes we set $M=N^2$ which results in
$\ccosts = \Oo \big(N^{3}\big)$.
The results are presented in Figure~\ref{Anal}, where the dashed line 
represents the theoretical effective order of convergence derived for the schemes $\MIL$ and 
$\SESRK$ while the dotted line shows the expected 
order of convergence for the schemes $\EES$ and $\LIE$. 
In this example, the relation of the operator $B$ to the $Q$-Wiener process 
is pointwise multiplicative; therefore, we do not expect a lower computational 
cost for the $\SESRK$ compared to the Milstein scheme.
\subsection{Stochastic reaction-diffusion equation}
We show an example with pointwise multiplicative operators which has been 
analyzed in \cite{MR3320928}. Here, the $\SESRK$ converges with the same order as the Milstein scheme and 
the scheme in \cite{MR3011387}.
We fix $H = U = L^2((0,1),\mathbb{R})$  and choose $Av = \frac{1}{100}\Delta v$, $v\in D(A)$, 
with $\lambda_i = \frac{1}{100}\pi^2i^2$, $e_i(x) = \sqrt{2}\sin(i\pi x)$ for $x\in(0,1)$,
$i\in\mathbb{N}$, and $\eta_j = j^{-2}$, $\tilde{e}_j = e_j$ for all $j\in\mathbb{N}$.
We consider
\begin{equation}\label{Num-SPDE-02}
  \mathrm{d}X_t = \left(\frac{1}{100} \Delta X_t +1-X_t\right) \mathrm{d}t 
  +\frac{1-X_t}{1+X_t^2}\,\mathrm{d}W_t
\end{equation}
with $X_0(x) = 0$ and $X_t(0) = X_t(1) = 0$ for $t\in[0,1]$, $x\in(0,1)$. 
For more details, we refer the reader to \cite{MR3320928}, where there is a proof 
that Assumptions (A1)--(A4) are fulfilled in this setting with $\beta = \frac{1}{5}$, 
$\alpha \in (0,\frac{3}{4})$, $\gamma \in (\frac{1}{2},\frac{3}{4})$, and we choose $q = \gamma = \frac{3}{4}
-\varepsilon$ for any $\varepsilon>0$.
The theoretical effective order of convergence is $\err(\SESRK(N,K,M)) = \err(\MIL(N,K,M))= \err(\RKS(N,K,M)) 
= \Oo\big(\ccosts^{-\frac{1}{2}+\varepsilon}\big)$, whereas $\err(\LIE(N,K,M)) = \err(\EES(N,K,M))= 
\Oo\big(\ccosts^{-\frac{3}{8}+\varepsilon}\big)$ as described in 
Section~\ref{SubSec:Effective-Order-Multiplicative-Operators}. 

%
%
\begin{figure}[tbp]
\begin{center}
  \includegraphics[height = 6cm, width = 0.5\textwidth]{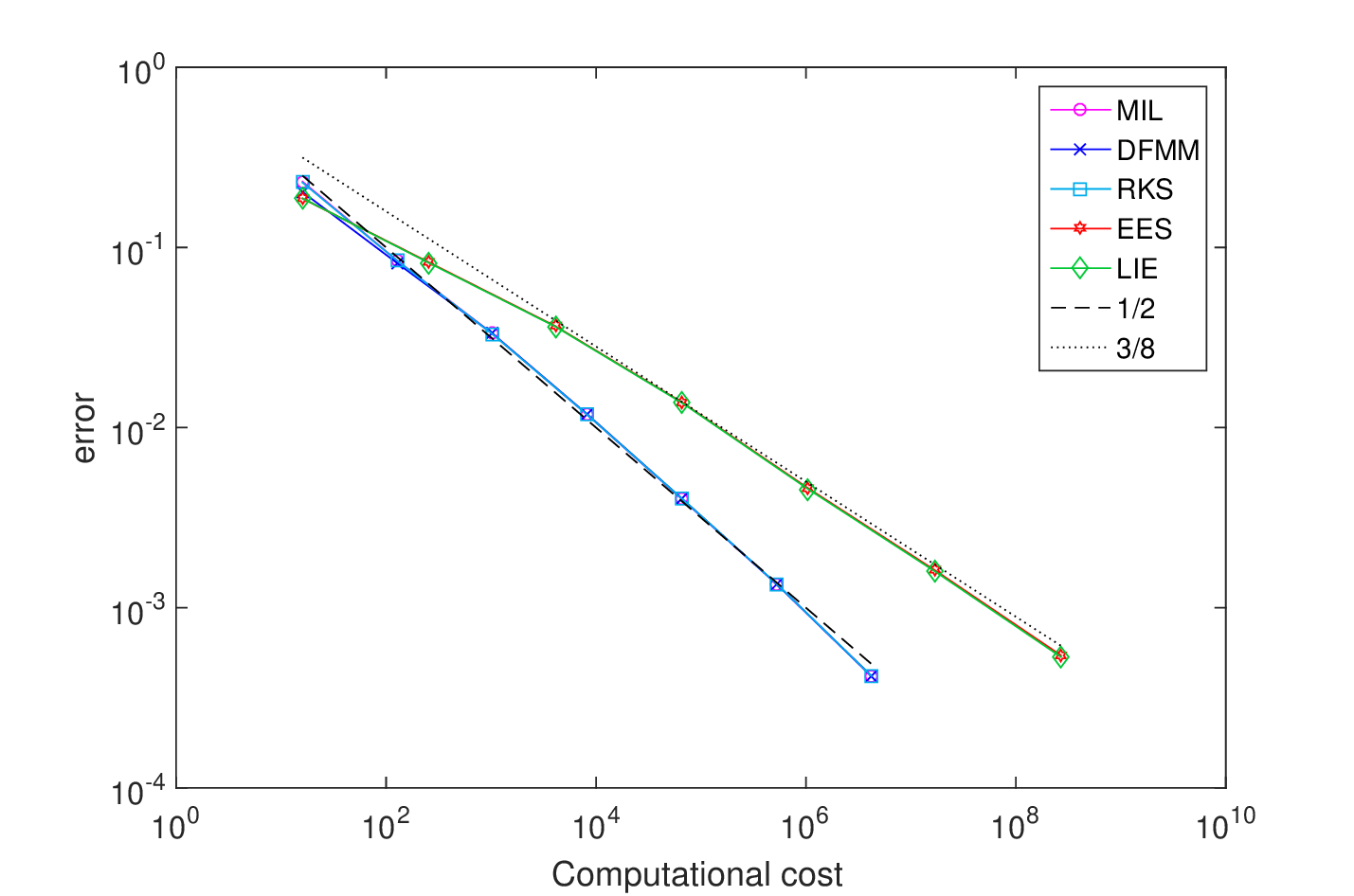}
  \caption{Error versus computational cost for SPDE \eqref{Num-SPDE-02} with
  $N \in \{2,4,8,16,32,64,128\}$ and 200 paths in log-log scale.}
\label{Rep}
\end{center}
\end{figure}
As in \cite{MR3320928}, we compare the approximations to a numerical reference solution computed with 
a linear implicit 
version of the Milstein scheme with $N=K=2^8$ and $M=2^{21}$, see \cite{2010arXiv1009.3526D}.
The approximations at $T=1$ are calculated with $M=N^2$, $K = N$, and $\ccosts = \Oo\big(N^3\big)$ for the schemes
$\MIL$, $\RKS$, $\SESRK$, and $M = N^3$, $K=N$ with $\ccosts = \Oo\big(N^4\big)$ for $\EES$ and $\LIE$.
The results are presented in Figure~\ref{Rep}, where it is obvious that the 
$\SESRK$ converges with the same order as the schemes $\MIL$
and $\RKS$ and clearly outperforms the Euler schemes $\LIE$ and $\EES$.
\subsection{Investigation of the effective order of convergence in the general case}
In the following, we consider equations which do not contain
pointwise multiplicative operators. 
Instead, we allow the operator $B$ to act on the $Q$-Wiener process in a more general manner. 
For these equations, the \EDFM $\ESRK$ is superior in terms of the effective order  of convergence compared 
to well-known schemes. 

%
In the following, let $\mu_{ij}\colon H_{\beta}\rightarrow \mathbb{R}$ and 
$\phi_{ij}^k \colon H_{\beta}\rightarrow \mathbb{R}$ 
be arbitrary functions for $i,k \in\mathcal{I}, j\in \mathcal{J}$, and consider the operators
\begin{equation} \label{OperatorSum}
\begin{split}
  B(y)u &= \sum_{i\in\mathcal{I}}\sum_{j\in\mathcal{J}} \mu_{ij}(y) \langle u,\tilde{e}_j\rangle_{U} e_i, \\
  B'(y)\left(v,u\right) &= \sum_{i\in\mathcal{I}}\sum_{j\in\mathcal{J}} D\mu_{ij}(y)(v) \langle u,\tilde{e}_j\rangle_{U} e_i \\ 
  &= \sum_{i,k\in\mathcal{I}}\sum_{j\in\mathcal{J}}
  \phi_{ij}^k(y) \langle v,e_k\rangle _H \langle u,\tilde{e}_j\rangle_{U} e_i
\end{split}
\end{equation}
for $y\in H_{\beta}$, $v\in H$, $u \in U$, where $D\mu_{ij}\colon H_{\beta}\rightarrow L(H,\mathbb{R})$
denotes the Fr\'{e}chet derivative of $\mu_{ij}$ for all $i\in\mathcal{I}$, $j\in\mathcal{J}$, 
i.e., the functional $\phi_{ij}^k$ denotes the derivative of $\mu_{ij}$ in direction $e_k$ 
for $i,k\in\mathcal{I}$, $j \in\mathcal{J}$.
The functionals $\mu_{ij}$, $\phi_{ij}^k$, $i,k\in\mathcal{I}$, $j \in\mathcal{J}$ have to be
chosen such that $B(y)u\in H$ and $B'(y)\left(v,u\right)\in H$ 
for all $y\in H_{\beta}$, $v\in H$, $u\in U$. 

%
In order to investigate Assumption (A3), we transfer the conditions to our setting such that 
they depend on $\mu_{ij}$ and $\phi_{ij}^k,$  $i,k\in\mathcal{I}$, $j\in \mathcal{J}$.
We assume that $\delta\in(0,\frac{1}{2})$ and $\mu_{ij}$, $i\in\mathcal{I}$, $j\in \mathcal{J}$ are chosen 
such that $B(H_{\delta})\subset L(U,H_{\delta})$. 
First, we rewrite $\|B(v)\|_{L(U,H_ {\delta})}$ for all $v\in H_{\delta}$ as
\begin{equation} \label{Bdelta}
  \begin{split}
  \|B(v)\|_{L(U,H_ {\delta})} &= \sup_{\substack{u\in U \\ \|u\|_U = 1}}\|B(v)u\|_{H_ {\delta}} \\
  &=\sup_{\substack{u\in U \\ \|u\|_U = 1}}  
  \Big\|(-A)^{\delta}\sum_{i\in\mathcal{I}} \sum_{j\in\mathcal{J}} \mu_{ij}(v) \langle u,\tilde{e}_j \rangle_{U}e_i  \Big\|_{H} \\
  &= \sup_{\substack{u\in U \\ \|u\|_U = 1}}  
  \Big\|\sum_{k\in\mathcal{I}} \lambda_k^{\delta} \sum_{j\in\mathcal{J}} \mu_{kj}(v) \langle u,\tilde{e}_j \rangle_{U}e_k  \Big\|_{H} \\
  &\leq \sum_{k\in\mathcal{I}} \sum_{j\in\mathcal{J}} \lambda_k^{\delta}  \vert\mu_{kj}(v) \vert .
  \end{split}
\end{equation}
We need $\|B(v)\|_{L(U,H_ {\delta})} \leq C(1+\|v\|_{H_{\delta}})$ for some $C>0$ and all $v\in  H_{\delta}$ 
which is examined for the different examples in the corresponding sections. 
Further, we calculate for $v,w \in H_{\gamma}$
\begin{equation}\label{BBLip}
  \begin{split}
  &\|B'(v)B(v)-B'(w)B(w)\|^2_{L_{HS}^{(2)}(U_0,H)} \\
  &\quad = \sum_{k,l\in\mathcal{J}} \left\|\sqrt{\eta_k} \sqrt{\eta_l} 
  \left(B'(v)(B(v)\tilde{e}_k,\tilde{e}_l) - 
  B'(w)(B(w)\tilde{e}_k,\tilde{e}_l) \right) \right\|_H^2 \\
  &\quad =  \sum_{k,l\in\mathcal{J}} \eta_k\eta_l \sum_{i,r_1,r_2\in\mathcal{I}} \left(\phi_{il}^{r_1}(v) \mu_{r_1k}(v)
  -\phi_{il}^{r_1}(w) \mu_{r_1k}(w)\right)\left(\phi_{il}^{r_2}(v) \mu_{r_2k}(v)
  -\phi_{il}^{r_2}(w) \mu_{r_2k}(w)\right) .
  \end{split}
\end{equation}
In order to analyze the conditions on the derivatives of $B$, we define $\hat{\phi}_{ij}^{kr}$ as the Fr\'{e}chet derivative of 
$\phi_{ij}^k$ in direction of $e_r$ for $i,k,r\in\mathcal{I}$, $j\in\mathcal{J}$
and obtain for $v\in H_{\beta}$ the estimate
\begin{align*}
  \|B'(v)\|_{L(H,L(U,H))} &= 
  \sup_{\substack{w\in H,u\in U \\ \|w\|_H =\|u\|_U = 1}} 
  \Big\|  \sum_{i,k\in\mathcal{I}} \sum_{j\in\mathcal{J}} \phi_{ij}^k(v)
  \langle w,e_k\rangle_H\langle u,\tilde{e}_j\rangle_Ue_i \Big\|_H 
  \leq \sum_{i,k\in\mathcal{I}} \sum_{j\in\mathcal{J}} \vert\phi_{ij}^k(v)\vert
\end{align*}
and for the second derivative, we get
\begin{align*}
  \|B''(v)\|_{L^{(2)}(H,L(U,H))} &= 
  \sup_{\substack{z,w\in H\\ \|z\|_H =\|w\|_H = 1}}\|B''(v)(z,w)\|_{L(U,H)} \nonumber \\
  &= \sup_{\substack{z,w\in H,u\in U \\ \|z\|_H =\|w\|_H = \|u\|_U= 1}}
  \Big\|  \sum_{i,r,l\in\mathcal{I}} \sum_{j\in\mathcal{J}} \hat{\phi}_{ij}^{rl}(v)
  \langle z,e_l\rangle_H \langle w,e_r\rangle_H\langle u,\tilde{e}_j\rangle_Ue_i \Big\|_H \nonumber \\
  &\leq \sum_{i,r,l\in\mathcal{I}} \sum_{j\in\mathcal{J}} \vert\hat{\phi}_{ij}^{rl}(v)\vert.
\end{align*}
Finally, for $v\in H_{\gamma}$, we have to investigate the term
\begin{equation}\label{Cond:Atheta}
  \begin{split}
  \|(-A)^{-\vartheta}B(v)Q^{-\alpha}\|_{L_{HS}(U_0,H)} &= \Big(\sum_{k\in\mathcal{J}} 
  \|(-A)^{-\vartheta}B(v)Q^{-\alpha+\frac{1}{2}}\tilde{e}_k\|_H^2\Big)^{\frac{1}{2}} \\
  &= \Big(\sum_{k\in\mathcal{J}} \eta_k^{1-2\alpha} \Big\|
  \sum_{i\in\mathcal{I}} \lambda_i^{-\vartheta} \mu_{ik}(v)e_i\Big\|_H^2\Big)^{\frac{1}{2}} \\
  &= \Big(\sum_{k\in\mathcal{J}} \eta_k^{1-2\alpha} \sum_{i\in\mathcal{I}} \lambda_i^{-2\vartheta} \mu_{ik}^2(v)\Big)^{\frac{1}{2}}
  \end{split}
\end{equation}
with $\vartheta \in (0,\frac{1}{2})$, $\alpha \in(0,\infty)$. 

%
We assume commutativity to rewrite \eqref{DoubleInt}. 
In this framework, the condition reads 
\begin{align*}
  \sum_{k\in\mathcal{I}} \phi_{im}^k(v)\mu_{kn}(v)
  = \sum_{k\in\mathcal{I}} \phi_{in}^k(v)\mu_{km}(v)
\end{align*}
for all $i\in\mathcal{I}$, $n,m\in\mathcal{J}_K$, $K\in\mathbb{N}$,  and $v\in H_{\beta}$. 

%
In the following, we fix some $N,K,M \in\mathbb{N}$ and $m\in \{0,\ldots,M-1\}$, and we consider an
SPDE of type~\eqref{SPDE} with operators as in~\eqref{OperatorSum}.
In this case, the Milstein scheme~\eqref{Milstein} reads
 \begin{align*}
  \YY_{m+1}  &= P_N\Big(e^{Ah}\Big(\YY_m + h F(\YY_m) 
  + \sum_{i\in\mathcal{I}}\sum_{\substack{j\in\mathcal{J}_K\\ \eta_j\neq0}} \mu_{ij}(\YY_m)
  \sqrt{\eta_j}\Delta\beta_m^j e_i \\
  &\quad + \frac{1}{2}\sum_{i,k\in\mathcal{I}}\sum_{\substack{j\in\mathcal{J}_K\\ \eta_j\neq0}}\phi_{ij}^k(\YY_m) 
  \sum_{\substack{r\in\mathcal{J}_K\\ \eta_r\neq 0}}\mu_{kr}(\YY_m) \sqrt{\eta_r}
  \Delta\beta_m^r\sqrt{\eta_j}\Delta\beta_m^j e_i \\
  &\quad -\frac{h}{2} 
  \sum_{i,k\in\mathcal{I}}\sum_{\substack{j\in\mathcal{J}_K\\ \eta_j\neq0}} \eta_j\phi_{ij}^k(\YY_m)\mu_{kj}(\YY_m) e_i \Big)\Big).
\end{align*}
Here, it is obvious that the evaluation of $\phi_{ij}^k$ for $i,k\in\mathcal{I}_N$ and $j\in\mathcal{J}_K$
results in $N^2K$ necessary evaluations of scalar nonlinear functions. 

%
For the \EDFM \eqref{ESRK-scheme-orig}, we obtain
\begin{align*}
  \YY_{m+1} &= P_N\Big(e^{Ah}\Big(\YY_m + h F(\YY_m) 
  + \sum_{i\in\mathcal{I}}\sum_{\substack{j\in\mathcal{J}_K\\ \eta_j\neq0}} \mu_{ij}(\YY_m) 
  \sqrt{\eta_j}\Delta\beta_m^j e_i \\
  &\quad + \frac{1}{\sqrt{h}}\sum_{i\in\mathcal{I}}\sum_{\substack{j\in\mathcal{J}_K\\ \eta_j\neq0}} \Big(\mu_{ij}
  \Big(\YY_m+\frac{\sqrt{h}}{2}  P_N\Big( \sum_{k\in\mathcal{I}}
  \sum_{\substack{ l\in\mathcal{J}_K \\ \eta_l\neq 0}}
  \mu_{kl}(\YY_m)\sqrt{\eta_l}\Delta\beta_m^l e_k\Big)\Big) \\
  &\quad -\mu_{ij}(\YY_m)\Big)\sqrt{\eta_j}\Delta\beta_m^j e_i \\
  &\quad +\sum_{i\in\mathcal{I}}\sum_{\substack{j\in\mathcal{J}_K\\ \eta_j\neq0}} 
  \Big(\mu_{ij}\Big(\YY_m-\frac{h}{2} P_N\Big(\sqrt{\eta_j}\sum_{k\in\mathcal{I}} 
  \mu_{kj}(\YY_m) e_k\Big)\Big) \\
  &\quad -\mu_{ij}(\YY_m)\Big)\sqrt{\eta_j}e_i \Big)\Big).
\end{align*}
The enhanced \EDFM needs 3 evaluations of each $\mu_{ij}$ for $i\in\mathcal{I}_N$,
$j\in\mathcal{J}_K$ which results in only $3NK$ necessary evaluations of scalar functions. 

%
The linear implicit Euler scheme takes the form
\begin{equation*}
  \YY_{m+1} = P_N\Big(\Big(I-hA\Big)^{-1} \Big(\YY_m + h F(\YY_m)
  + \sum_{i\in\mathcal{I}}\sum_{\substack{j\in\mathcal{J}_K \\ 
  \eta_j\neq0}} \mu_{ij}(\YY_m) \sqrt{\eta_j}\Delta\beta_m^j e_i\Big)\Big)
\end{equation*}
and the exponential Euler scheme reads
\begin{equation*}
  \YY_{m+1} = P_N\Big(e^{Ah} \YY_m + A^{-1}\Big(e^{Ah}-I\Big) F(\YY_m)
  + e^{Ah} \sum_{i\in\mathcal{I}}\sum_{\substack{j\in\mathcal{J}_K \\ 
  \eta_j\neq0}} \mu_{ij}(\YY_m) \sqrt{\eta_j}\Delta\beta_m^j e_i\Big).
\end{equation*}
Both Euler schemes require one evaluation $\mu_{ij}$ for $i\in\mathcal{I}_N$,
$j\in\mathcal{J}_K$ and thus $NK$ evaluations of scalar functions.
The Runge-Kutta type scheme in~\cite{MR3011387} is not applicable in this setting. 

%
Stochastic partial differential equations as in~\eqref{SPDE} where the operator $B$
is of type~\eqref{OperatorSum} are extensively treated in~\cite{MR3236753} and
applications as well as models based on such SPDEs can be found, e.g., within the following references: 
Stochastic reaction-diffusion equations, which describe phenomena from chemistry, biology,
and physics, are considered in~\cite{MR3305472}.
Stochastic regulator problems as well as optimal stationary control problems are discussed in \cite{MR516862}.
In the field of computational neuroscience, stochastic Hopfield neural networks with 
distributed parameters are analyzed in~\cite{MR2127642}.
Further examples are stochastic distributed parameter systems and optimal control problems in~\cite[Chap.~5.2]{MR2165651}
and the stochastic modeling of flame propagation in~\cite{MR2174871}.
All of the mentioned applications contain 
settings such that the proposed \EDFM $\ESRK$ can be applied to 
the equations involved and where it attains a higher effective order of convergence compared to the schemes $\MIL$, $\EES$, and $\LIE$.

In the following examples in this section, we consider a stochastic 
reaction-diffusion equation from~\cite{MR3305472} which reads as
\begin{equation} \label{SPDE-Example_reac-diff-applications}
 \mathrm{d} X_t(x) =\big( a \, \Delta X_t(x) + F(X_t(x))\big) \, \mathrm{d}t + \sum_{k=1}^{\infty} g_k(x,X_t(x)) \, \mathrm{d}\beta^k_t, 
 \quad x\in(0,1),
\end{equation}
with some initial and boundary conditions, some functions $g_k \colon (0,1) \times \mathbb{R} \to \mathbb{R}$, 
$k\in\mathbb{N}$, and $a\in\mathbb{R}$, see~\cite{MR3305472} for details. Here, we set
for $k\in\mathcal{J}$ and $x\in(0,1)$
\begin{equation*}
  g_k(x,X_t(x))= \big( B(X_t) \sqrt{\eta_k} \, \tilde{e}_k \big) (x) =
  \sum_{i\in\mathcal{I}}\mu_{ik}(X_t) \, \sqrt{\eta_k} \, e_i(x)
\end{equation*}
in order to align this notation with the operators introduced in \eqref{OperatorSum}. In the examples
below, we consider different possible choices for $\mu_{ij}$, $i\in\mathcal{I}$,
$j\in\mathcal{J}$ such that the assumptions in \cite{MR3305472} hold.

%
To be specific, we choose $H=U=L^2((0,1),\mathbb{R})$, $T=1$ and consider the equation
\begin{equation}\label{SPDE:Setting}
  \mathrm{d} X_t = \Big(\frac{1}{100}\Delta X_t + 1-X_t\Big)\, \mathrm{d} t + B(X_t) \, \mathrm{d}W_t
\end{equation}
with Dirichlet boundary
conditions $X_t(0)=X_t(1) =0$ for all $t\in[0,T]$ and assume $X_0(x) = 0$ 
for all $x\in(0,1)$. We select $e_i = \sqrt{2}\sin(i x \pi)$, 
$i\in\mathcal{I}$, as the orthonormal basis of $H$ 
with $\lambda_i = \frac{1}{100}\pi^2i^2$ for all $i\in\mathcal{I}$.  $(W_t)_{t\in[0,T]}$
is a $Q$-Wiener process in $U$ and we choose the eigenvalues $\eta_j = j^{-3}$ 
of $Q$ with eigenfunctions 
$\tilde{e}_j =  \sqrt{2}\sin(j x \pi)$ for all $j\in\mathcal{J}$, if not stated otherwise.
Thus, SPDE \eqref{SPDE:Setting} is of type \eqref{SPDE-Example_reac-diff-applications}.
In this setting, Assumptions (A1), (A2), and (A4) obviously hold, see also \cite{MR3320928}.
Below, we have a look at some specific examples
in order to illustrate the effective order of convergence for the schemes under consideration
and show that Assumption (A3) is fulfilled.
\subsubsection{The case of a linear operator}\label{Ex1}
In our first example, we consider SPDE \eqref{SPDE:Setting},
define the operator $B$ as in \eqref{OperatorSum} by the linear mappings 
$\mu_{ij}(y) = \frac{\langle y,e_i\rangle_{H} }{i^4+j^4}$, and obtain
for the derivative in direction $e_k$ the function
\begin{equation*}
  \phi_{ij}^k(y) = \begin{cases} 0, & k \neq i \\ \frac{1}{i^4+j^4}, & k = i \end{cases}
\end{equation*}
for all $i,k \in\mathcal{I}$, $j\in \mathcal{J}$, $y\in H_{\beta}$. 

%
First, we prove that Assumption (A3) holds. 
By \eqref{Bdelta}, we obtain for all $v\in H_{\delta}$
\begin{align*}
  \|B(v)\|_{L(U,H_{\delta})} &\leq 
  \sum_{k\in\mathcal{I}} \sum_{j\in\mathcal{J}} \Big(\frac{1}{100}\pi^2k^2\Big)^{\delta} 
  \frac{\vert \langle v, e_k\rangle_H\vert}{k^4+j^4} 
  \leq C  \sum_{k\in\mathcal{I}} \sum_{j\in\mathcal{J}} \frac{1}{k^{2-2\delta}}
  \frac{1}{j^2} \|(-A)^{-\delta}\|_{L(H)} \|v\|_{H_{\delta}} .
\end{align*}
Thus, we obtain $\|B(v)\|_{L(U,H_{\delta})} \leq  C(1+\|v\|_{H_{\delta}})$
for $v\in H_{\delta}$ and $\delta \in (0,\frac{1}{2})$ due to Assumption (A1).
Considering \eqref{BBLip}, we compute
\begin{align*}
  \|B'(v)B(v)-B'(w)B(w)\|^2_{L_{HS}^{(2)}(U_0,H)} 
  &= \sum_{k,l\in\mathcal{J}} \frac{1}{k^3}\frac{1}{l^3} \sum_{i\in\mathcal{I}} \frac{1}{(i^4+l^4)^2} \Big(\frac{\langle v-w,e_i \rangle_H}{i^4+k^4}\Big)^2 \\
  &\leq C \sum_{k,l\in\mathcal{J}} \frac{1}{k^7}\frac{1}{l^7} \sum_{i\in\mathcal{I}} \frac{1}{i^8} \|v-w\|_H^2
\end{align*}
for $v,w\in H_{\gamma}$.
Moreover, for the derivative of $B$, we obtain
\begin{align*}
  \|B'(v)\|_{L(H,L(U,H))} &
  \leq  
  \sum_{i\in\mathcal{I}}\sum_{j\in\mathcal{J}} \frac{1}{i^4+j^4} 
  \leq  \sum_{i\in\mathcal{I}}\sum_{j\in\mathcal{J}} \frac{1}{i^2}\frac{1}{j^2},
\end{align*}
that is, $\|B'(v)\|_{L(H,L(U,H))}  <\infty$
for all $v\in H_{\beta}$.
The condition on the second derivative is obviously fulfilled as $\hat{\phi}_{ij}^{kr}(y) = 0$ 
for all $i,j,r\in\mathcal{I}$, $j\in\mathcal{J}$, and $y\in H_{\beta}$. 
Finally, with \eqref{Cond:Atheta} we determine $\alpha$ by
\begin{align*}
  \|(-A)^{-\vartheta}B(v)Q^{-\alpha}\|_{L_{HS}(U_0,H)} 
  &= \Big(\sum_{k\in\mathcal{J}} \frac{1}{k^{3(1-2\alpha)}} \sum_{i\in\mathcal{I}} 
  \lambda_i^{-2\vartheta} \frac{\langle v,e_i\rangle_H^2}{(i^4+k^4)^2}\Big)^{\frac{1}{2}} \\
  &\leq C \Big(\sum_{k\in\mathcal{J}} \frac{1}{k^{3(1-2\alpha)+4}} 
  \sum_{i\in\mathcal{I}} \frac{1}{i^{4+4\vartheta}} \Big)^{\frac{1}{2}} \|v\|_{H_{\gamma}}
\end{align*}
and obtain
$\|(-A)^{-\vartheta}B(v)Q^{-\alpha}\|_{L_{HS}(U_0,H)}  \leq C(1+\|v\|_{H_{\gamma}})$
for all $\alpha \in (0,1)$, $\vartheta \in (0,\frac{1}{2})$, and $v\in H_{\gamma}$. 

%
Summarizing, the parameters can take the values 
$ \delta,\vartheta \in (0,\frac{1}{2})$, $\alpha \in (0,1)$, $\beta\in[0,1)$,
and  $\gamma\in[\frac{1}{2},1)$.
Here and in the examples below, we select the maximal value for $\delta$ 
and choose $\beta =0$.

%
Finally, the commutativity condition is fulfilled due to
\begin{align*}
  \sum_{k\in\mathcal{I}} \phi_{im}^k(v)\mu_{kn}(v) 
  &= \frac{1}{i^4+m^4}\frac{\langle v,e_i\rangle_H}{i^4+n^4}  
  = \sum_{k\in\mathcal{I}} \phi_{in}^k(v)\mu_{km}(v)
\end{align*}
for all $i\in\mathcal{I}$, $n,m\in\mathcal{J}_K$, $K\in\mathbb{N}$, $v\in H_{\beta}$. 

%
For this example, we have $\rho_Q = 3$, $\rho_A =2$ and choose $q=\gamma =\alpha = 1-\varepsilon$ 
for some arbitrary
$\varepsilon >0$ which yields $K = N^{\frac{2}{3}}$ and 
$M = N^2$ with computational cost $\ccosts = \Oo\big(N^{\frac{11}{3}}\big)$ for the scheme $\ESRK$
and $\ccosts = \Oo\big(N^{\frac{14}{3}}\big)$ for $\MIL$. Further, 
we select $\gamma =\alpha = 1-\varepsilon$  
which results in $q=\frac{1}{2}$ and choose $K = N^{\frac{2}{3}}$, $M = N^4$ 
with $\ccosts = \Oo\big(N^{\frac{17}{3}}\big)$ for the schemes $\EES$ and $\LIE$. 
For the effective order of convergence, we obtain $\err(\MIL(N,K,M)) =  
\Oo \big( \ccosts^{-\frac{3}{7}+\varepsilon} \big), \,
\err(\LIE(N,K,M)) = \err(\EES(N,K,M)) =\Oo \big( \ccosts^{-\frac{6}{17}+\varepsilon} \big)$
whereas for the $\ESRK$, we get 
$\err(\ESRK(N,K,M)) =  \Oo \big( \ccosts^{-\frac{6}{11}+\varepsilon }\big)$. 

%
\setlength\extrarowheight{3pt}
\begin{table}[tbp]
\begin{small}
\begin{center}
\begin{tabular}{|p{0.7cm}|p{0.7cm}|p{0.7cm}||p{1.3cm}|p{2cm}|p{1.5cm}||p{1.3cm}|p{2cm}|p{1.5cm}|}\hline
\multicolumn{3}{|c||}{} & \multicolumn{3}{|c||}{Milstein}      &    \multicolumn{3}{|c|}{CDFM}       \\ \hline
 $N$   & $M$       &  $K$                  & $\ccosts$                &Error                  & Std                   & $\ccosts$                  & Error                & Std     \\ \hline
 2   & $2^2$   &  $2^{\frac{2}{3}}$  &$\Oo(2^{\frac{14}{3}})$	& $3.0\cdot 10^{-2}$	&  $1.5\cdot 10^{-3}$	&$\Oo(2^{\frac{11}{3}})$     & $3.0\cdot 10^{-2}$	&  $1.5\cdot 10^{-3}$   \\ \hline
 4   & $2^4$   &  $2^{\frac{4}{3}}$  &$\Oo(2^{\frac{28}{3}})$	& $2.5\cdot 10^{-2}$ 	&  $3.0\cdot 10^{-4}$	&$\Oo(2^{\frac{22}{3}})$     & $2.5\cdot 10^{-2}$ 	&  $3.0\cdot 10^{-4}$  \\ \hline		
 8   & $2^6$   & $2^{2}$	     &	$\Oo(2^{14})$	        & $1.7\cdot 10^{-2}$	&  $6.0\cdot 10^{-5}$   &$\Oo(2^{11})$               & $1.7\cdot 10^{-2}$	&  $6.0\cdot 10^{-5}$   \\ \hline
 16  & $2^8$   &$2^{\frac{8}{3}}$    & $\Oo(2^{\frac{56}{3}})$  &$6.3\cdot 10^{-3}$ 	&  $1.1\cdot 10^{-5}$ 	&$\Oo(2^{\frac{44}{3}})$     & $6.3\cdot 10^{-3}$ 	&  $1.1\cdot 10^{-5}$   \\ \hline
 32  & $2^{10}$&$2^{\frac{10}{3}}$   & $\Oo(2^{\frac{70}{3}})$  &$1.6\cdot 10^{-3}$	&  $2.0\cdot 10^{-6}$	&$\Oo(2^{\frac{55}{3}})$     & $1.6\cdot 10^{-3}$	&  $2.0\cdot 10^{-6}$   \\ \hline
\end{tabular}
\quad \\[0.2cm]
\begin{tabular}{|p{0.7cm}|p{0.7cm}|p{0.7cm}||p{1.3cm}|p{2cm}|p{1.5cm}||p{1.3cm}|p{2cm}|p{1.5cm}|}\hline
\multicolumn{3}{|c||}{} &\multicolumn{3}{|c||}{Linear Implicit Euler}  & \multicolumn{3}{|c|}{Exponential Euler}   \\ \hline
 $N$    &$M$       &  $K$                  & $\ccosts$               &Error                     & Std                 & $\ccosts$                    & Error             & Std   \\ \hline
2    &$2^4$   &  $2^{\frac{2}{3}}$  &$\Oo(2^{\frac{17}{3}})$	& $2.2\cdot 10^{-2}$ 	   &  $4.0\cdot 10^{-3}$ &$\Oo(2^{\frac{17}{3}})$	& $2.3\cdot 10^{-2}$ 	&  $4.0\cdot 10^{-3}$  \\ \hline
4    &$2^8$   &  $2^{\frac{4}{3}}$  &$\Oo(2^{\frac{34}{3}})$	& $2.7\cdot 10^{-2}$ 	   &  $6.5\cdot 10^{-4}$ &$\Oo(2^{\frac{34}{3}})$	& $2.7\cdot 10^{-2}$ 	&  $6.5\cdot 10^{-4}$   \\ \hline		
8    &$2^{12}$& $2^{2}$             &$\Oo(2^{17})$	        & $1.7\cdot 10^{-2}$ 	   &  $1.2\cdot 10^{-4}$ &$\Oo(2^{17})$		        & $1.7\cdot 10^{-2}$ 	&  $1.1\cdot 10^{-4}$   \\ \hline
16   &$2^{16}$&$2^{\frac{8}{3}}$    &$\Oo(2^{\frac{68}{3}})$	& $6.1\cdot 10^{-3}$       &  $2.3\cdot 10^{-5}$ &$\Oo(2^{\frac{68}{3}})$	& $6.1\cdot 10^{-3}$    &  $2.3\cdot 10^{-5}$   \\ \hline
32   &$2^{20}$&$2^{\frac{10}{3}}$   &$\Oo(2^{\frac{85}{3}})$	& $1.5\cdot 10^{-3}$	   &  $3.9\cdot 10^{-6}$ &$\Oo(2^{\frac{85}{3}})$	& $1.5\cdot 10^{-3}$    &  $3.9\cdot 10^{-6}$  \\ \hline
\end{tabular}
\end{center}
\end{small}
\caption[Error and standard deviation for Example 1]{Error and standard deviation 
for Example \ref{Ex1}  -- computed for 700 paths with batches of size 50 (\cite[p.312]{MR1214374}).} 
\label{Tab:Ex3-1}
\end{table}
\begin{figure}[tbp] 
\begin{center}
\includegraphics[height = 6cm, width = 0.5\textwidth]{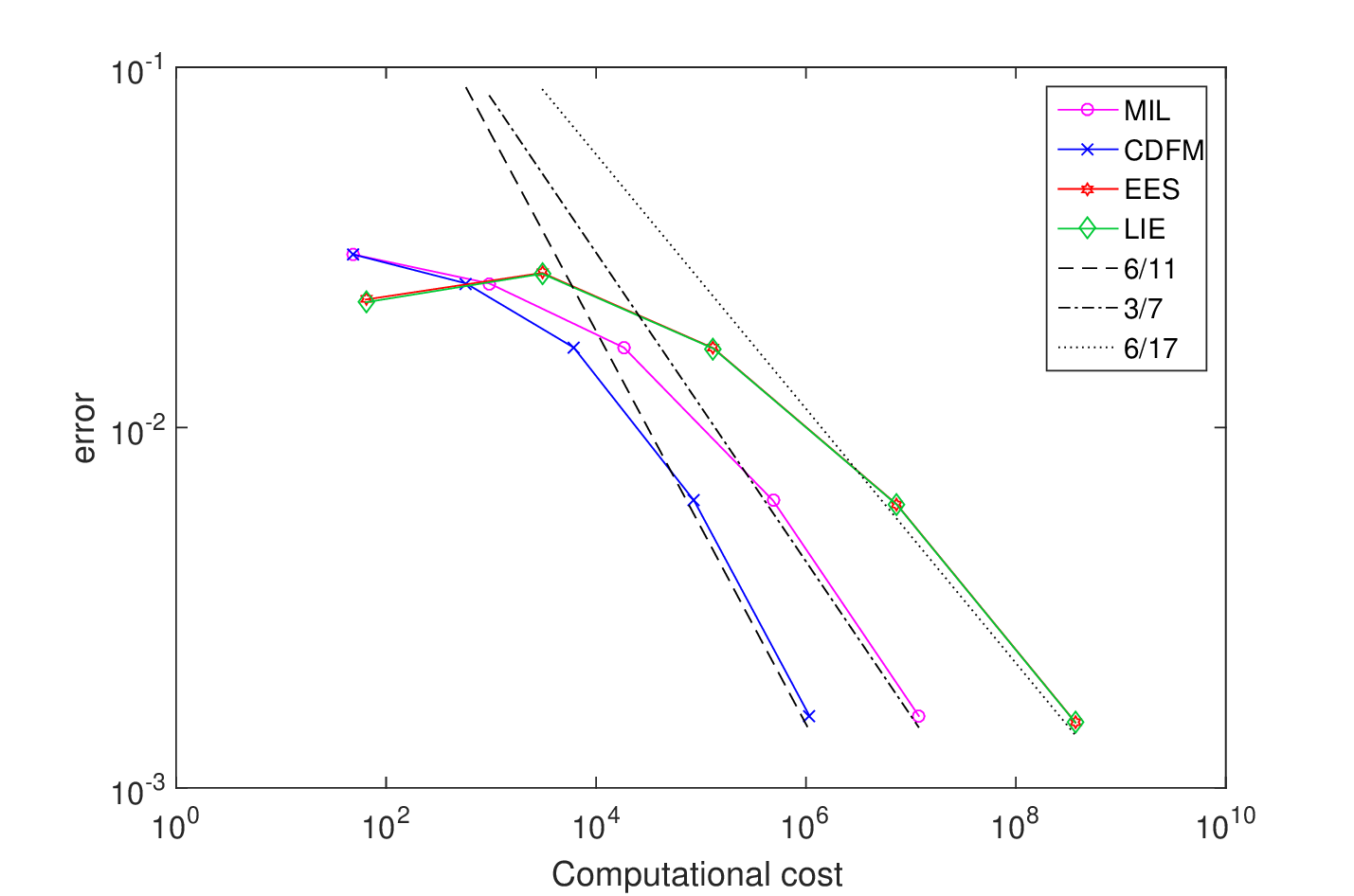}
\caption{Error against computational cost for Example \ref{Ex1} 
for $N \in \{2,4,8,16,32 \}$ and 700 paths in log-log scale.}
\label{Ex431}
\end{center}
\end{figure}
The following logarithmic plot of the error for $N \in \{ 2,4,8,16,32 \}$ confirms the theoretical results. 
As a substitute for an exact solution, we choose the linear implicit Euler scheme with 
$N = 2^6$, $K = 2^4$, and $M= 2^{20}$.
Compared to the other schemes, the effective order of convergence is significantly higher for the 
enhanced \EDFM. 
In Figure~\ref{Ex431}, the dashed line represents the 
effective order of convergence derived for the \EDFM theoretically and the dotted or dashed-dotted 
line shows the expected order of
convergence for the reference schemes, see also Table~\ref{Tab:Ex3-1}.
The orders which are suggested by the computations in Section~\ref{CC} are numerically confirmed.
\subsubsection{An SPDE with different bases for $H$ and $U$}\label{Ex2}
Now, we analyze an example for SPDE \eqref{SPDE:Setting} 
where the basis functions of the spaces $H$ and $U$ are not the same.
Therefore, we choose $\tilde{e}_j(x) = \sqrt{2}\cos(j\pi x)$ for all $j\in\mathcal{J}$, $x\in(0,1)$ as
a basis in $U$. Further, we set
$\mu_{ij}(y) = \frac{1}{j^2}\sum_{p\in\mathcal{I}}\frac{\langle y,e_p\rangle_{H} }{i^3+p^4}$, 
$ i\in\mathcal{I}$, $j\in\mathcal{J}$, $y \in H_{\beta}$ to define the operator $B$.
In this case, we get
$\phi_{ij}^k(y) = \frac{1}{j^2}\frac{1}{i^3+k^4}$ for all $i,k \in\mathcal{I}$, $j\in\mathcal{J}$, 
$y \in H_{\beta}$. 

%
Here, we only check the commutativity condition in Assumption (A3) and observe
\begin{align*}
  \sum_{k\in\mathcal{I}} \phi_{im}^k(v)\mu_{kn}(v)
  &= \sum_{k\in\mathcal{I}}\frac{1}{m^2}\frac{1}{i^3+k^4}\frac{1}{n^2} 
  \sum_{p\in\mathcal{I}}\frac{\langle v,e_p\rangle_H}{k^3+p^4} 
  = \sum_{k\in\mathcal{I}} \phi_{in}^k(v)\mu_{km}(v)
\end{align*}
for all $i\in\mathcal{I}$, $n,m\in\mathcal{J}_K$, $K\in\mathbb{N}$, and $v\in H_{\beta}$. 

%
\setlength\extrarowheight{3pt}
\begin{table}[tbp]
\begin{small}
\begin{center}
\begin{tabular}{|p{0.7cm}|p{0.7cm}|p{0.7cm}||p{1.3cm}|p{2cm}|p{1.5cm}||p{1.3cm}|p{2cm}|p{1.5cm}|}\hline
\multicolumn{3}{|c||}{} & \multicolumn{3}{|c||}{Milstein}      &    \multicolumn{3}{|c|}{CDFM}       \\ \hline
 $N$   &$M$        &  $K$                & $\ccosts$                 &Error                   & Std                   & $\ccosts$                  & Error                    & Std   \\ \hline
 2   & 4       &$2^{\frac{1}{2}}$  &$\Oo(2^{\frac{9}{2}})$	&$3.2\cdot 10^{-2}$	&  $3.0\cdot 10^{-3}$	&$\Oo(2^{\frac{7}{2}})$      & $3.2\cdot 10^{-2}$	&  $3.0\cdot 10^{-3}$   \\ \hline
 4   & $2^4$   &$2$                &$\Oo(2^{9})$		& $2.5\cdot 10^{-2}$ 	&  $5.0\cdot 10^{-4}$	&$\Oo(2^{7})$                & $2.5\cdot 10^{-2}$ 	&  $5.0\cdot 10^{-4}$  \\ \hline		
 8   & $2^6$   &$2^{\frac{3}{2}}$  &$\Oo(2^{\frac{27}{2}})$	& $1.7\cdot 10^{-2}$	&  $6.2\cdot 10^{-5}$   &$\Oo(2^{\frac{21}{2}})$     & $1.7\cdot 10^{-2}$	&  $6.2\cdot 10^{-5}$   \\ \hline
 16  & $2^8$   &$2^{2}$            &$\Oo(2^{18})$		& $6.6\cdot 10^{-3}$ 	&  $2.0\cdot 10^{-5}$ 	&$\Oo(2^{14})$               & $6.6\cdot 10^{-3}$ 	&  $2.0\cdot 10^{-5}$   \\ \hline
 32  & $2^{10}$&$2^{\frac{5}{2}}$  &$\Oo(2^{\frac{45}{2}})$	& $2.0\cdot 10^{-3}$	&  $7.0\cdot 10^{-6}$	&$\Oo(2^{\frac{35}{2}})$     & $2.0\cdot 10^{-3}$	&  $7.0\cdot 10^{-6}$   \\ \hline
 64  & $2^{12}$&$2^{3}$            &$\Oo(2^{27})$		&$4.6\cdot 10^{-4}$     &  $5.4\cdot 10^{-6}$   &$\Oo(2^{21})$               & $4.6\cdot 10^{-4}$       &  $5.4\cdot 10^{-6}$   \\ \hline
\end{tabular}
\quad \\[0.2cm]
 \begin{tabular}{|p{0.7cm}|p{0.7cm}|p{0.7cm}||p{1.3cm}|p{2cm}|p{1.5cm}||p{1.3cm}|p{2cm}|p{1.5cm}|}\hline
\multicolumn{3}{|c||}{} &\multicolumn{3}{|c||}{Linear Implicit Euler}  & \multicolumn{3}{|c|}{Exponential Euler}   \\ \hline
$N$    &$M$       &  $K$                & $\ccosts$                   &Error                     & Std                 & $\ccosts$                    & Error             & Std   \\ \hline
2    &$2^3$   &$2^{\frac{1}{2}}$  &$\Oo(2^{\frac{9}{2}})$	& $2.4\cdot 10^{-2}$ 	   &  $4.6\cdot 10^{-3}$ &$\Oo(2^{\frac{9}{2}})$	& $2.4\cdot 10^{-2}$ 	&  $4.7\cdot 10^{-3}$  \\ \hline
4    &$2^6$   &$2$                &$\Oo(2^{9})$			& $2.7\cdot 10^{-2}$ 	   &  $6.2\cdot 10^{-4}$ &$\Oo(2^{9})$			&$2.7\cdot 10^{-2}$ 	&  $6.7\cdot 10^{-4}$   \\ \hline		
8    &$2^{9}$ &$2^{\frac{3}{2}}$  &$\Oo(2^{\frac{27}{2}})$	& $1.7\cdot 10^{-2}$ 	   &  $1.6\cdot 10^{-4}$ &$\Oo(2^{\frac{27}{2}})$	&$1.7\cdot 10^{-2}$ 	&  $1.8\cdot 10^{-4}$   \\ \hline
16   &$2^{12}$&$2^{2}$            &$\Oo(2^{18})$		& $6.5\cdot 10^{-3}$       &  $5.3\cdot 10^{-5}$ &$\Oo(2^{18})$		        & $6.7\cdot 10^{-3}$    &  $5.9\cdot 10^{-5}$   \\ \hline
32   &$2^{15}$&$2^{\frac{5}{2}}$  &$\Oo(2^{\frac{45}{2}})$	& $1.9\cdot 10^{-3}$	   &  $7.9\cdot 10^{-6}$ &$\Oo(2^{\frac{45}{2}})$	& $2.0\cdot 10^{-3}$    &  $9.6\cdot 10^{-6}$  \\ \hline
64   &$2^{18}$&$2^{3}$            &$\Oo(2^{27})$		& $4.3\cdot 10^{-4}$       &  $4.5\cdot 10^{-6}$ &$\Oo(2^{27})$		        & $4.8\cdot 10^{-4}$    & $5.8\cdot 10^{-6}$  \\ \hline
\end{tabular}
 \end{center}
\end{small}
\caption[Error and standard deviation for Example 2]
{Error and standard deviation computed for Example \ref{Ex2} -- computed for 500 paths with batches of size 50.}
\label{Std2}
\end{table}
\begin{figure}[tbp]
\begin{center}
\includegraphics[height = 6cm, width = 0.5\textwidth]{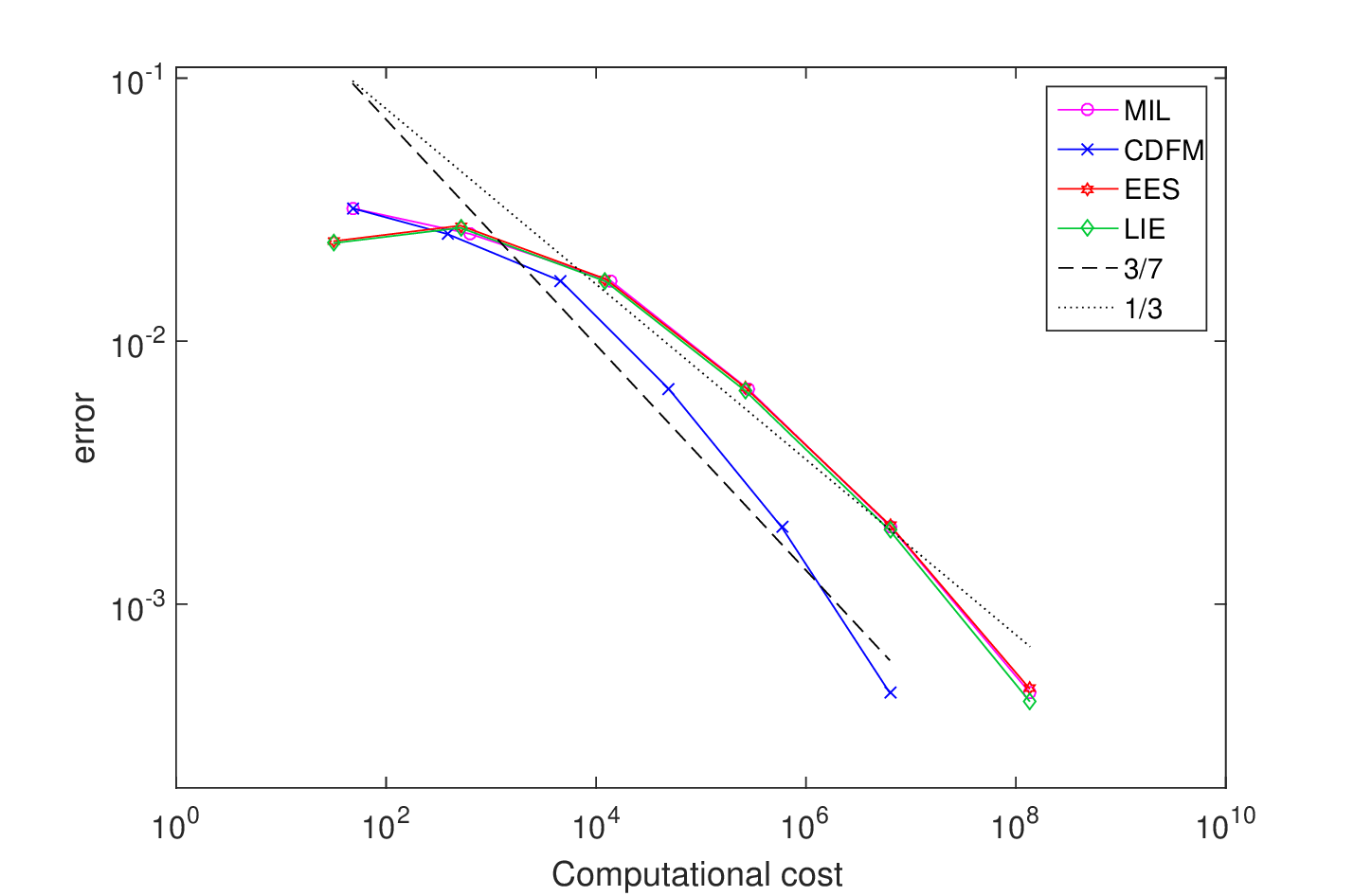}
\caption{Error against computational 
cost for Example \ref{Ex2} with $N \in \{2,4,8,16,32,64 \}$ for 500 paths in log-log scale.}
\label{Bsp2}
\end{center}
\end{figure}
The validation of Assumption (A3) follows as in the previous example and is not detailed here. 
The parameter values are
$ \delta\in (0,\frac{1}{4})$,  $\vartheta \in (0,\frac{1}{2})$, $\alpha \in (0,1)$, $\beta\in[0,1)$, and
with the choice $\beta = 0$ and $\delta$ maximal, we obtain $\gamma \in [\frac{1}{4},\frac{3}{4})$.
Here, the optimal choice is $\alpha =1-\varepsilon$, $\gamma = \frac{3}{4}-\varepsilon$ 
and we get $q = \frac{3}{4}-\varepsilon$ for the schemes
$\MIL$ and $\ESRK$, whereas we get $q= \frac{1}{2}$ for the Euler schemes.
For this parameter setting, we choose $K = \sqrt{N}$, $M = N^2$ 
for the Milstein scheme and the \EDFM and
$K = \sqrt{N}$, $M = N^3$ for the linear implicit and the exponential Euler scheme from Section \ref{CC}. 
The effective order of convergence equals $\err(\ESRK(N,K,M)) =  
\Oo \big( \ccosts^{-\frac{3}{7}+\varepsilon }\big)$ 
with cost $\ccosts = \Oo\big(N^{\frac{7}{2}}\big)$
whereas for the other schemes, we get
$\err(\MIL(N,K,M))$ $=  \err(\LIE(N,K,M))$ $= \err(\EES(N,K,M)) =
\Oo \big( \ccosts^{-\frac{1}{3}+\varepsilon} \big)$
with $\ccosts = \Oo\big(N^{\frac{9}{2}}\big)$.
As a substitute for the exact solution, we choose an approximation obtained with the linear implicit 
Euler scheme with $N= 2^7$, $K = 2^{7/2}$, and $M = 2^{18}$.
Again, the theoretical results are nicely confirmed by Table~\ref{Std2} and Figure~\ref{Bsp2}.
\subsubsection{The case of a nonlinear operator}\label{Ex3}
Here, we consider SPDE \eqref{SPDE:Setting} with a nonlinear operator $B$. Therefore,
we define the operator $B$ as in \eqref{OperatorSum} 
with the nonlinear functions $\mu_{ij} \colon H_{\beta}\rightarrow \mathbb{R}$ for  
$i \in\mathcal{I}$, $j\in \mathcal{J}$ in this example.
Precisely, we choose
\begin{align*}
  \mu_{ij}(y) &= \sum_{p\in \mathcal{I}} \frac{e^{-\langle y,e_p\rangle_H^2}}{i^{\frac{3}{2}}j^2}
  \Big( \frac{1}{i+j+p^2} \mathds{1}_{(j-1)^2+1\leq i,p\leq j^2} +\sum_{r=0}^{j-2}\frac{1}{(r+1)}
  \frac{1}{i+(r+1)+p^2}\mathds{1}_{r^2+1\leq i,p\leq(r+1)^2} \Big).
\end{align*}
For the derivative of $\mu_{ij}$, we get the nonlinear function
\begin{align*}
  \phi_{ij}^k(y) &=  \frac{-2\langle y,e_k\rangle_He^{-\langle y,e_k\rangle_H^2}}{i^{\frac{3}{2}}j^2} 
  \Big( \frac{1}{i+j+k^2} \mathds{1}_{(j-1)^2+1\leq i,k\leq j^2} \\
  &\quad +\sum_{r=0}^{j-2} \frac{1}{(r+1)}
  \frac{1}{i+(r+1)+k^2}\mathds{1}_{r^2+1\leq i,k\leq(r+1)^2} \Big)
\end{align*}
for $i,k\in\mathcal{I},\; j\in\mathcal{J},$ and  $y\in H_{\beta}$. 
Here, we only prove the commutativity 
{\allowdisplaybreaks
\begin{align*}
  \sum_{k\in\mathcal{I}} \phi_{im}^k(y)\mu_{kn}(y)
  &= \sum_{k\in\mathcal{I}} \Big(\frac{1}{i^{\frac{3}{2}}m^2} \frac{1}{i+m+k^2} 
  \mathds{1}_{(m-1)^2+1\leq i,k\leq m^2} \\
  &\quad +\frac{1}{m^2}\sum_{r=0}^{m-2} \frac{1}{i^{\frac{3}{2}}(r+1)}\frac{1}{i+(r+1)+k^2}\mathds{1}_{r^2+1\leq i,k\leq(r+1)^2} \Big)e^{-\langle y,e_k\rangle_H^2}(-2\langle y,e_k \rangle_H)\\
  &\quad \quad \times \bigg(\sum_{p\in \mathcal{I}}\frac{1}{k^{\frac{3}{2}}n^2} \frac{1}{k+n+p^2} \mathds{1}_{(n-1)^2+1\leq k,p\leq n^2}\\
  &\quad + \sum_{p\in \mathcal{I}}\frac{1}{n^2}\sum_{r=0}^{n-2} \frac{1}{k^{\frac{3}{2}}(r+1)}\frac{1}{k+(r+1)+p^2}\mathds{1}_{r^2+1\leq k,p\leq(r+1)^2} \bigg) e^{-\langle y,e_p\rangle_H^2} \\
  &= \sum_{k,p=(m-1)^2+1}^{m^2} \frac{1}{i^{\frac{3}{2}}k^{\frac{3}{2}}m^4} \frac{1}{i+m+k^2}  \frac{1}{k+m+p^2} \\
  &\quad \quad \times e^{-\langle y,e_k\rangle_H^2}e^{-\langle y,e_p\rangle_H^2}(-2\langle y,e_k\rangle_H) \mathds{1}_{(m-1)^2+1\leq i\leq m^2} \mathds{1}_{m=n}\\
  &\quad +\sum_{k,p=(m-1)^2+1}^{m^2}  \frac{1}{n^2m^2}\frac{1}{i^{\frac{3}{2}}} \frac{1}{i+m+k^2}\frac{1}{k^{\frac{3}{2}}m}\frac{1}{k+m+p^2} \\
  &\quad \quad \times e^{-\langle y,e_k\rangle_H^2}e^{-\langle y,e_p\rangle_H^2}(-2\langle y,e_k\rangle_H) \mathds{1}_{(m-1)^2+1\leq i\leq m^2}\mathds{1}_{m < n}\\
  &\quad +  \sum_{k,p=(n-1)^2+1}^{n^2} \frac{1}{m^2} \frac{1}{i^{\frac{3}{2}}n} \frac{1}{i+n+k^2} \frac{1}{k^{\frac{3}{2}}n^2}\frac{1}{k+n+p^2}\\
  &\quad \quad \times   e^{-\langle y,e_k\rangle_H^2}e^{-\langle y,e_p\rangle_H^2}(-2\langle y,e_k\rangle_H)\mathds{1}_{(n-1)^2+1\leq i\leq n^2}\mathds{1}_{n <m}\\
  &\quad + \sum_{r=0}^{\min(n-2,m-2)} \sum_{k,p=r^2+1}^{(r+1)^2} \frac{1}{m^2n^2} \frac{1}{i^{\frac{3}{2}}(r+1)}\frac{1}{i+(r+1)+k^2}\frac{1}{k^{\frac{3}{2}}(r+1)}\frac{1}{k+(r+1)+p^2}   \\
  &\quad \quad \times \mathds{1}_{r^2+1\leq i\leq (r+1)^2} e^{-\langle y,e_k\rangle_H^2-\langle y,e_p\rangle_H^2}(-2\langle y,e_k\rangle_H)  \\
  &= \sum_{k\in\mathcal{I}} \phi_{in}^k(y)\mu_{km}(y)
\end{align*}
for all $i\in\mathcal{I}$, $m,n\in\mathcal{J}_K$, $K\in\mathbb{N}$, and $y\in H_{\beta}$.} 

%
The condition (A3) is fulfilled for this example which we obtain similarly as before with parameters
$\delta \in(0,\frac{1}{4})$, $\vartheta \in(0,\frac{1}{2})$, $\beta\in[0,1)$, and $\alpha\in(0,1)$. 
With the choice $\beta = 0$, we obtain $\gamma\in[\frac{1}{4},\frac{3}{4})$. 
\setlength\extrarowheight{3pt}
\begin{table}[tbp]
\begin{small}
\begin{center}
\begin{tabular}{|p{0.7cm}|p{0.7cm}|p{0.7cm}||p{1.3cm}|p{2cm}|p{1.5cm}||p{1.3cm}|p{2cm}|p{1.5cm}|}\hline
\multicolumn{3}{|c||}{} & \multicolumn{3}{|c||}{Milstein}      &    \multicolumn{3}{|c|}{CDFM}       \\ \hline
 $N$   &$M$        &  $K$                & $\ccosts$                 &Error                   & Std                   & $\ccosts$                  & Error                    & Std   \\ \hline
 2   & 4       &$2^{\frac{1}{2}}$  &$\Oo(2^{\frac{9}{2}})$	&$2.9\cdot 10^{-2}$	&  $2.3\cdot 10^{-3}$	& $\Oo(2^{\frac{7}{2}})$      &$2.8\cdot 10^{-2}$	&  $2.1\cdot 10^{-3}$   \\ \hline
 4   & $2^4$   &$2$                &$\Oo(2^{9})$		& $2.5\cdot 10^{-2}$ 	&  $3.8\cdot 10^{-4}$	& $\Oo(2^{7})$                &$2.5\cdot 10^{-2}$ 	&  $3.9\cdot 10^{-4}$  \\ \hline		
 8   & $2^6$   &$2^{\frac{3}{2}}$  &$\Oo(2^{\frac{27}{2}})$	& $1.7\cdot 10^{-2}$	&  $6.3\cdot 10^{-5}$   & $\Oo(2^{\frac{21}{2}})$     &$1.7\cdot 10^{-2}$	&  $6.4\cdot 10^{-5}$   \\ \hline
 16  & $2^8$   &$2^{2}$            &$\Oo(2^{18})$		& $6.6\cdot 10^{-3}$ 	&  $1.2\cdot 10^{-5}$ 	& $\Oo(2^{14})$               &$6.6\cdot 10^{-3}$ 	&  $1.2\cdot 10^{-5}$   \\ \hline
 32  & $2^{10}$&$2^{\frac{5}{2}}$  &$\Oo(2^{\frac{45}{2}})$	& $1.9\cdot 10^{-3}$	&  $3.5\cdot 10^{-6}$	& $\Oo(2^{\frac{35}{2}})$     &$1.9\cdot 10^{-3}$	&  $3.5\cdot 10^{-6}$   \\ \hline
 64  & $2^{12}$&$2^{3}$            &$\Oo(2^{27})$		&$4.4\cdot 10^{-4}$     &  $1.2\cdot 10^{-6}$   & $\Oo(2^{21})$               & $4.4\cdot 10^{-4}$     &  $1.2\cdot 10^{-6}$   \\ \hline
\end{tabular}
\quad \\[0.2cm]
  \begin{tabular}{|p{0.7cm}|p{0.7cm}|p{0.7cm}||p{1.3cm}|p{2cm}|p{1.5cm}||p{1.3cm}|p{2cm}|p{1.5cm}|}\hline
\multicolumn{3}{|c||}{} &\multicolumn{3}{|c||}{Linear Implicit Euler}  & \multicolumn{3}{|c|}{Exponential Euler}   \\ \hline
$N$    &$M$       &  $K$                & $\ccosts$                   &Error                  & Std                  & $\ccosts$              & Error                 & Std   \\ \hline
2    &$2^3$   &$2^{\frac{1}{2}}$  &$\Oo(2^{\frac{9}{2}})$	&$1.8\cdot 10^{-2}$ 	&  $1.7\cdot 10^{-3}$  &$\Oo(2^{\frac{9}{2}})$	&$1.9\cdot 10^{-2}$ 	&  $1.7\cdot 10^{-3}$  \\ \hline
4    &$2^6$   &$2$                &$\Oo(2^{9})$			&$2.6\cdot 10^{-2}$ 	&  $3.8\cdot 10^{-4}$  &$\Oo(2^{9})$		&$2.6\cdot 10^{-2}$ 	&  $4.5\cdot 10^{-4}$   \\ \hline		
8    &$2^{9}$ &$2^{\frac{3}{2}}$  &$\Oo(2^{\frac{27}{2}})$	& $1.7\cdot 10^{-2}$ 	&$1.7\cdot 10^{-2}$    &$\Oo(2^{\frac{27}{2}})$	&$1.7\cdot 10^{-2}$ 	 &  $1.0\cdot 10^{-4}$   \\ \hline
16   &$2^{12}$&$2^{2}$            &$\Oo(2^{18})$		&$6.4\cdot 10^{-3}$     &  $1.2\cdot 10^{-5}$  &$\Oo(2^{18})$		&$6.6\cdot 10^{-3}$     &  $1.8\cdot 10^{-5}$   \\ \hline
32   &$2^{15}$&$2^{\frac{5}{2}}$  &$\Oo(2^{\frac{45}{2}})$	&$1.9\cdot 10^{-3}$	&  $4.9\cdot 10^{-6}$  &$\Oo(2^{\frac{45}{2}})$	&$2.0\cdot 10^{-3}$     &  $5.1\cdot 10^{-6}$  \\ \hline
64   &$2^{18}$&$2^{3}$            &$\Oo(2^{27})$		&$4.2\cdot 10^{-4}$     &  $2.6\cdot 10^{-7}$  &$\Oo(2^{27})$		&$4.9\cdot 10^{-4}$     & $2.2\cdot 10^{-6}$  \\ \hline
\end{tabular}
 \end{center}
\end{small}
\caption[Error and standard deviation for Example 3]
{Error and standard deviation computed for Example \ref{Ex3} -- computed for 500 paths with batches of size 50.}
\label{Tab:Ex3.3}
\end{table}
\begin{figure}[tbp]
\begin{center}
\includegraphics[height = 6cm, width = 0.5\textwidth]{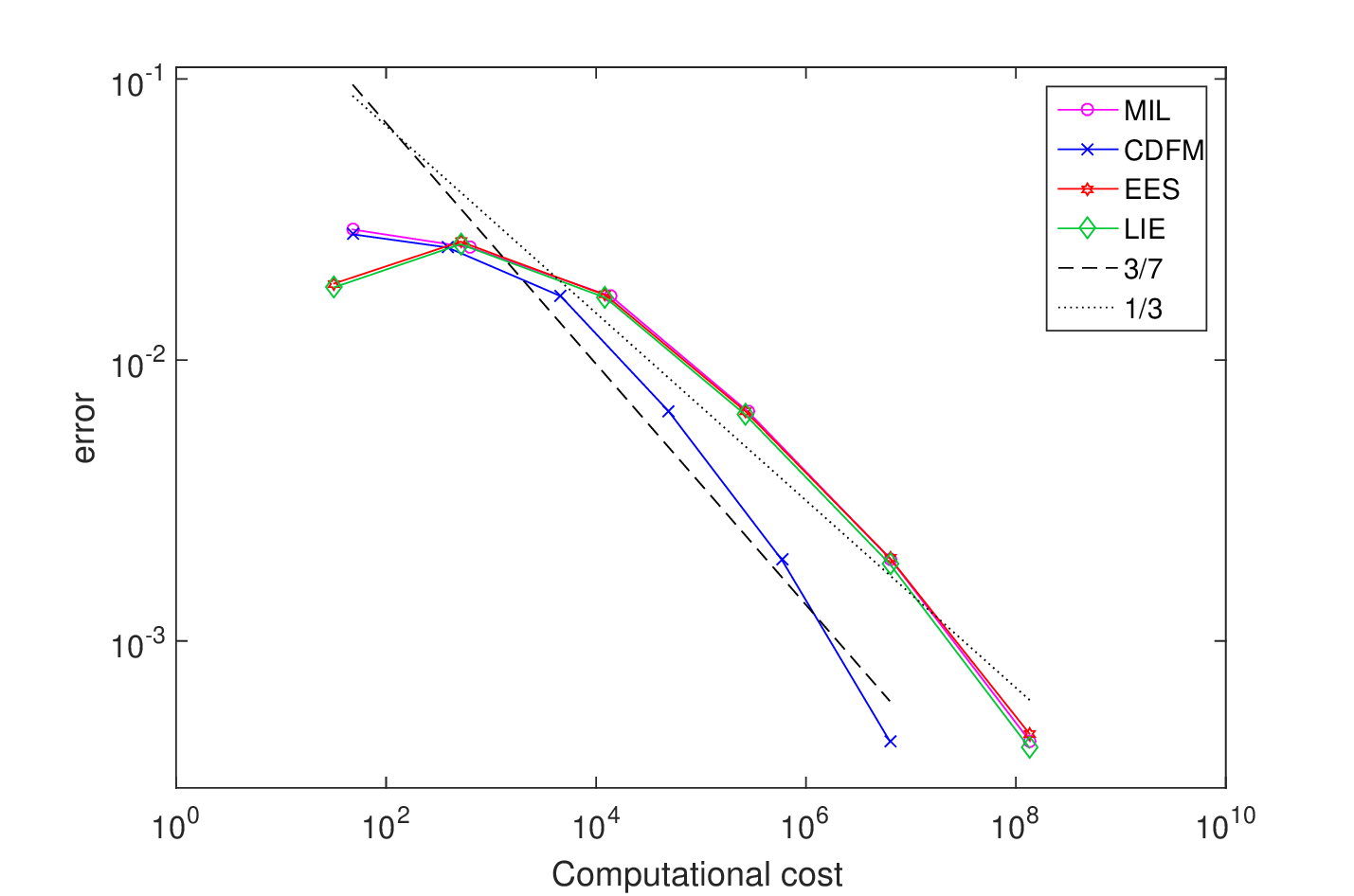}
\caption{Error against computational cost for Example \ref{Ex3} with $N \in \{2,4,8,16,32,64 \}$ for 500 paths
 in log-log scale.}
 \label{Fig:Ex3.3}
\end{center}
\end{figure} 
Again, we choose $K = \sqrt{N}$, $M = N^2$ for the Milstein and the \EDFM, and
$K = \sqrt{N}$, $M = N^3$ for the linear implicit and the exponential Euler schemes. 
For the Milstein scheme, we expect an effective order of convergence of 
$\err(\MIL(N,K,M)) =  \Oo \big( \ccosts^{-\frac{1}{3}+\varepsilon }\big)$ with cost $\ccosts = \Oo\big(N^{\frac{9}{2}}\big)$, 
for the linear implicit Euler and the exponential  Euler schemes, we expect the same order 
$\err(\LIE(N,K,M)) =\err(\EES(N,K,M)) =  \Oo \big( \ccosts^{-\frac{1}{3}+\varepsilon }\big)$ with
$\ccosts = \Oo\big(N^{\frac{9}{2}}\big)$,
and for the \EDFM, we have $\err(\ESRK(N,K,M)) =  \Oo \big( \ccosts^{-\frac{3}{7}+\varepsilon }\big)$
with $\ccosts = \Oo\big(N^{\frac{7}{2}}\big)$.
In order to compute the mean-square error, we replace the exact solution with an approximation obtained 
with the linear implicit Euler scheme for $N = 2^7$, $K = 2^{7/2}$, $M = 2^{18}$. The simulation 
results are displayed in Figure~\ref{Fig:Ex3.3} and Table~\ref{Tab:Ex3.3}.
\allowdisplaybreaks
\section{Proofs}\label{Proof}
Before we give the proof of Theorem~\ref{MainTh} and Corollary~\ref{Corollary-MainTh-pointwise}, we recall 
some elementary facts on the analytic semigroup $e^{At}$, $t \geq 0$ that are frequently used below.
\begin{lma}[{\cite[Lemma 6.13]{MR710486}}]\label{SGest}
  Let Assumption (A1) be fulfilled. Then, it holds that
  $\|(-A)^{-\theta}(e^{At}-I)\|_{L(H)} \leq C_{\theta}t^{\theta}$ and $\|(-A)^{\theta} e^{At}\|_{L(H)} \leq C_{\theta}t^{-\theta}$
  for $t>0$ and $\theta\in[0,1]$.
\end{lma}
Further, we also need the following lemma giving a uniform bound for the numerical approximation
to prove Theorem~\ref{MainTh} and Corollary~\ref{Corollary-MainTh-pointwise}. 
Note that a generic constant $C>0$ which may change from line to line is used in the following proofs.
\begin{lma} \label{Proof:Lemma-Moment}
Let Assumptions (A1)--(A4) be fulfilled. Then, 
for all $p \in [2,\infty)$, $N,K,M \in \mathbb{N}$, and some constant
$C_{p,T,Q} > 0$ it holds that
\begin{equation*}
  \sup_{m\in\{0,\ldots,M\}}\mathrm{E}\left[\|\YY_m\|_{H_{\delta}}^p\right]^{\frac{1}{p}}
  \leq C_{p,T,Q} \left(1+\mathrm{E}\left[\|\xi\|_{H_{\delta}}^p\right]^{\frac{1}{p}}\right). 
\end{equation*}
\end{lma}
\begin{proof}[Proof of Lemma~\ref{Proof:Lemma-Moment}]
The assertion is proved by induction. Let $N,K,M \in \mathbb{N}$, $p\in[2,\infty)$
and set $Y_m := \YY_m$ as defined in \eqref{ESRK-scheme-orig}--\eqref{ESRK-scheme-orig-bar-B}
and \eqref{ESRK-scheme-Special}--\eqref{ESRK-scheme-Special-bar-B}, respectively,
as well as $\Delta W_m^K := \Delta \WW_m$, for better legibility. 
For $m=0$, the estimate obviously holds. Therefore, let $m \in \{1,\ldots,M\}$ and assume 
that the estimate holds for all $l \in \{0,\ldots,m-1\}$.
Then, we get by the triangle inequality
\begin{align*}
  \mathrm{E}\Big[\| Y_{m}\|_{H_{\delta}}^p\Big]^{\frac{2}{p}}
  &\leq C \, \Bigg( \mathrm{E}\left[\left\|X_0\right\|_{H_{\delta}}^p\right]^{\frac{1}{p}} 
  + \su \mathrm{E}\left[\left\|\I \e{t_m-t_l}F(Y_l)\, \mathrm{d}s\right\|_{H_{\delta}}^p\right]^{\frac{1}{p}} \\
  &\quad + \mathrm{E}\left[\left\| \int_{t_0}^{t_m} \su \e{t_m-t_l}B(Y_l) \, \mathds{1}_{[t_l,t_{l+1})}(s)\, \mathrm{d}W^K_s\right\|_{H_{\delta}}^p\right]^{\frac{1}{p}}\\
  &\quad + \su \mathrm{E}\Bigg[\Bigg\|e^{A(t_m-t_l)} \frac{1}{\sqrt{h}} \bigg(B\bigg(Y_l+ \frac{\sqrt{h}}{2}P_NB(Y_l)\Delta W^K_l\bigg)-B(Y_l)\bigg)\Delta W^K_l  \Bigg\|_{H_{\delta}}^p\Bigg]^{\frac{1}{p}} \\
  &\quad + \su \sum_{\substack{j \in \mathcal{J}_K \\ \eta_j \neq 0}} \mathrm{E}\Bigg[\Bigg\| e^{A(t_m-t_l)} \bar{B}(Y_l,h,j)\Bigg\|_{H_{\delta}}^p\Bigg]^{\frac{1}{p}} \Bigg)^2.
\end{align*}
With a Burkholder-Davis-Gundy type inequality \cite[Theorem~4.37]{MR3236753} applied to the third
summand and with the definition of $H_{\delta}$, we obtain
\begin{align*}
  \mathrm{E}\Big[\| Y_{m}\|_{H_{\delta}}^p\Big]^{\frac{2}{p}} 
  &\leq C \Bigg( \mathrm{E}\left[\left\|X_0\right\|_{H_{\delta}}^p\right]^{\frac{2}{p}} 
  + \left(\su \left( \mathrm{E}\left[\left\|(-A)^{\delta}\e{t_m-t_l}F(Y_l)\right\|_{H}^p\right] \right)^{\frac{1}{p}} h \right)^2 \\
  &\quad + \int_{t_0}^{t_m} \mathrm{E}\Bigg[\left\|  \su \e{t_m-t_l}B(Y_l) \, \mathds{1}_{[t_l,t_{l+1})}(s)\right\|_{L_{HS}(U_0,H_{\delta})}^p\Bigg]^{\frac{2}{p}} \, \mathrm{d}s \\
  &\quad + \Bigg(\su \|(-A)^{\delta}e^{A(t_m-t_l)}\|_{L(H)} \\
  &\quad \quad \times \mathrm{E}\bigg[\bigg\|\frac{1}{\sqrt{h}} 
  \bigg(B\bigg(Y_l+ \frac{\sqrt{h}}{2}P_NB(Y_l)\Delta W^K_l\bigg)-B(Y_l)\bigg)\Delta W^K_l  \bigg\|_{H}^p\bigg]^{\frac{1}{p}} \Bigg)^2\\
  &\quad  + \Bigg(\su \sum_{\substack{j \in \mathcal{J}_K \\ \eta_j \neq 0}} \mathrm{E}\left[\left\|(-A)^{\delta}e^{A(t_m-t_l)}  
  \bar{B}(Y_l,h,j) \right\|_{H}^p\right]^{\frac{1}{p}} \Bigg)^2 \Bigg).
\end{align*}
First, we consider the $\ESRK$ scheme where
\begin{equation*}
 \bar{B}(Y_l,h,j) = \bigg(B\left(Y_l-\frac{h}{2}P_NB(Y_l)\sqrt{\eta_j}\tilde{e}_j\right)-B(Y_l)\bigg)\sqrt{\eta_j}\tilde{e}_j
\end{equation*}
and use the following Taylor expansions of the difference approximations for all $l\in\{0,\ldots,m-1\}$, $j\in\mathcal{J}_K$:
\begin{equation} \label{Taylor}
  \begin{split}
  & B\bigg(Y_l+\frac{\sqrt{h}}{2} P_NB(Y_l)\Delta W^K_l\bigg)\Delta W^K_l
  = B(Y_l)\Delta W^K_l + \int_0^1 B'(\xi_1(Y_l,u)) \bigg(\frac{\sqrt{h}}{2} P_NB(Y_l)\Delta W^K_l,\Delta W^K_l\bigg)\, \mathrm{d}u , \\
  & B\bigg(Y_l-\frac{h}{2}P_NB(Y_l)\sqrt{\eta_j}\tilde{e}_j\bigg)\sqrt{\eta_j}\tilde{e}_j
  = B(Y_l)\sqrt{\eta_j}\tilde{e}_j + \int_0^1 B'(\xi_2(Y_l,j,u)) \bigg(-\frac{h}{2}P_NB(Y_l)\sqrt{\eta_j}\tilde{e}_j,\sqrt{\eta_j}\tilde{e}_j\bigg) \, \mathrm{d}u ,
  \end{split}
\end{equation}
where
\begin{align*}
  \xi_1(Y_l,u) = Y_l + u \frac{\sqrt{h}}{2} P_N B(Y_l)\Delta W^K_l
\end{align*}
and
\begin{align*}
  \xi_2(Y_l,j,u) = Y_l - u \frac{h}{2}P_N B(Y_l)\sqrt{\eta_j}\tilde{e}_j
\end{align*}
for some $u \in [0,1]$. 
Note that $\xi_1(Y_l,u), \xi_2(Y_l,j,u) \in H_N$
and therefore, it holds $\xi_1(Y_l,u), \xi_2(Y_l,j,u) \in H_{\beta}$ for arbitrary 
$l\in\{0,\ldots,m-1\}$, $j\in\mathcal{J}_K$, $u \in [0,1]$. Inserting the Taylor expansions and applying 
(A1)--(A3) together with Lemma~\ref{SGest} yields
\begin{align*}
  &\mathrm{E}\Big[\| Y_{m}\|_{H_{\delta}}^p\Big]^{\frac{2}{p}} \\
  &\leq  C\mathrm{E}\left[\left\|X_0\right\|_{H_{\delta}}^p\right]^{\frac{2}{p}} 
  + C h^{2} M \su (t_m-t_l)^{-2 \delta} \mathrm{E}\Big[\left\|F(Y_l)\right\|_{H}^p\Big]^{\frac{2}{p}} \\
  &\quad + C \su \int_{t_l}^{t_{l+1}} \mathrm{E}\bigg[\bigg\|\sum_{k=0}^{m-1} \e{t_m-t_k}B(Y_k) \, 
  \mathds{1}_{[t_k,t_{k+1})}(s)\bigg\|_{L_{HS}(U_0,H_{\delta})}^p\bigg]^{\frac{2}{p}} \, \mathrm{d}s \\      
  &\quad + C \frac{M}{h} \su (t_m-t_l)^{-2\delta}  \mathrm{E}\bigg[\bigg\| \int_0^1 B'( \xi_1(Y_l,u) ) 
  \frac{\sqrt{h}}{2}P_NB(Y_l)\Delta W^K_l \,\mathrm{d}u \bigg\|_{L(U,H)}^p \left\|\Delta W^K_l\right\|_U^p\bigg]^{\frac{2}{p}}\\
  &\quad + C M \su  \bigg( \sum_{\substack{j \in \mathcal{J}_K \\ \eta_j \neq 0}} (t_m-t_l)^{-\delta} 
  \mathrm{E}\bigg[\bigg\| -\int_0^1 B'( \xi_2(Y_l,j,u) ) \frac{h}{2} P_N B(Y_l)\sqrt{\eta_j}\tilde{e}_j \, \mathrm{d}u \bigg\|_{L(U,H)}^p 
  \|\sqrt{\eta_j}\tilde{e}_j\|_U^p \bigg]^{\frac{1}{p}} \bigg)^2 \\
  &\leq C \mathrm{E}\left[\left\|X_0\right\|_{H_{\delta}}^p\right]^{\frac{2}{p}} + C_{p,T} h \su (t_m-t_l)^{-2 \delta}
  \Big(1+\mathrm{E}\Big[\|Y_l\|_{H_{\delta}}^p\Big]^{\frac{2}{p}}\Big) \\
  &\quad + C \su \mathrm{E}\left[ \left\|B(Y_l)\right\|_{L_{HS}(U_0,H_{\delta})}^p\right]^{\frac{2}{p}} \int_{t_l}^{t_{l+1}}  
  \left\|(-A)^{-\delta}\right\|_{L(H)}^2  \left\|(-A)^{\delta}\e{t_m-t_l}\right\|_{L(H)}^2 \, \mathrm{d}s \\      
  &\quad + C M \su (t_m-t_l)^{-2\delta}
  \mathrm{E}\left[ \bigg( \int_0^1 \| B'( \xi_1(Y_l,u) )\|_{L(H,L(U,H))}
  \|B(Y_l)\Delta W^K_l\|_{H_{\delta}} \, \mathrm{d}u\bigg)^p \left\|\Delta W^K_l\right\|_U^p\right]^{\frac{2}{p}}\\
  &\quad + C M \su \left(t_m-t_l\right)^{-2\delta}
  \bigg( \sum_{\substack{j \in \mathcal{J}_K \\ \eta_j \neq 0}} \sqrt{\eta_j} h \\
  &\quad \quad \times \mathrm{E}\left[ \bigg( \int_0^1 \| B'( \xi_2(Y_l,j,u) )\|_{L(H,L(U,H))}
  \left\| B(Y_l) \sqrt{\eta_j} \tilde{e}_j \right\|_{H_{\delta}} \, \mathrm{d}u\bigg)^p \right]^{\frac{1}{p}} \bigg)^2 \\
  &\leq C\mathrm{E}\left[\left\|X_0\right\|_{H_{\delta}}^p\right]^{\frac{2}{p}} 
  + h^{1-2\delta} C_{p,T} \su (m-l)^{-2\delta} \Big(1+\mathrm{E}\Big[\|Y_l\|_{H_{\delta}}^p\Big]^{\frac{2}{p}}\Big) \\
  &\quad + C_p \su h\left(t_m-t_l\right)^{-2\delta}  \Big(1+\mathrm{E}\Big[\|Y_l\|_{H_{\delta}}^p\Big]^{\frac{2}{p}}\Big) \\
  &\quad + C_{p} M h^{-2\delta} \su (m-l)^{-2\delta} \mathrm{E} \left[ (1+ \|Y_l\|_{H_{\delta}}^p ) \left\|\Delta W^K_l\right\|_U^{2p}\right]^{\frac{2}{p}} \\
  &\quad + C_{p,T} h^{1-2\delta} \su (m-l)^{-2\delta} \bigg( \sum_{\substack{j \in \mathcal{J}_K \\ \eta_j \neq 0}} \sqrt{\eta_j}  
  \, \mathrm{E}\left[ (1+ \|Y_l\|_{H_{\delta}}^p ) \left\|\sqrt{\eta_j}\tilde{e}_j\right\|_U^p\right]^{\frac{1}{p}}\bigg)^2 .
\end{align*}
Looking at the sum as a lower Darboux sum, 
we obtain for $\delta \in (0,\frac{1}{2})$ and all $m\in\{1,\ldots,M\}$
\begin{equation}\label{EstSum}
  \su (m-l)^{-2\delta} = \sum_{l=1}^m \frac{1}{l^{2\delta}} \leq 1+\int_1^M \frac{1}{r^{2\delta}}\, \mathrm{d}r 
  = 1+\frac{M^{1-2\delta}-1}{1-2\delta} \leq \frac{M^{1-2\delta}}{1-2\delta}.
\end{equation}
This results in
\begin{align*}
  \mathrm{E}\Big[\| Y_{m}\|_{H_{\delta}}^p\Big]^{\frac{2}{p}} 
  &\leq C \mathrm{E}\left[\left\|X_0\right\|_{H_{\delta}}^p\right]^{\frac{2}{p}} 
  + h^{1-2\delta} C_{T,p,Q} \su (m-l)^{-2\delta} \Big(1+\mathrm{E}\Big[\|Y_l\|_{H_{\delta}}^p\Big]^{\frac{2}{p}}\Big)\\
  &\leq C \mathrm{E}\left[\left\|X_0\right\|_{H_{\delta}}^p\right]^{\frac{2}{p}} + C_{T,p,Q} 
  + h^{1-2\delta} C_{T,p,Q} \su (m-l)^{-2\delta} \mathrm{E}\Big[\|Y_l\|_{H_{\delta}}^p\Big]^{\frac{2}{p}} .
\end{align*}
Finally, we obtain by the discrete Gronwall lemma
\begin{align*}
  \mathrm{E}\Big[\| Y_{m}\|_{H_{\delta}}^p\Big]^{\frac{2}{p}} 
  &\leq \left(C_p\mathrm{E}\left[\left\|X_0\right\|_{H_{\delta}}^p\right]^{\frac{2}{p}}
  + C_{T,p,Q}\right)e^{C_{T,p,Q} h^{1-2\delta} \su (m-l)^{-2\delta}} \\
  &\leq C_{T,p,Q}\left(1+\mathrm{E}\left[\left\|X_0\right\|_{H_{\delta}}^p\right]^{\frac{2}{p}}\right) .
\end{align*}
%
The result for the $\SESRK$ scheme where $\bar{B}(Y_l,h,j)$ 
is defined by \eqref{ESRK-scheme-Special-bar-B} follows analogously.
\end{proof}
Next, we give the proof of Theorem \ref{MainTh} and Corollary~\ref{Corollary-MainTh-pointwise} 
that builds on the proof of convergence in \cite{MR3320928} 
-- however with an additional new part which accounts for the approximation of the derivative. 
We do not incorporate the analysis of the error which possibly results from the approximation of the coefficients
in the spectral projection $P_NX_t = \sum_{n\in\mathcal{I}_N}\langle X_t,e_n\rangle_He_n$ here. 
%
\begin{proof}[Proof of Theorem~\ref{MainTh} and Corollary~\ref{Corollary-MainTh-pointwise}]
We use the representation
\begin{align*}
  X_{t_m} &= e^{At_m}X_0 + \sI\e{t_m-s}F(X_s)\, \mathrm{d}s + \sI\e{t_m-s}B(X_s)\, \mathrm{d}W_s,
\end{align*}
set $Y_m := \YY_m$ as defined in \eqref{ESRK-scheme-orig}--\eqref{ESRK-scheme-orig-bar-B}
and \eqref{ESRK-scheme-Special}--\eqref{ESRK-scheme-Special-bar-B} for $m\in\{0,\ldots,M\}$, respectively,
and set $\Delta W_m^K := \Delta \WW_m$ for $m\in\{0,\ldots,M-1\}$,  with $N,K,M \in\mathbb{N}$, for improved legibility.
Further, we define some auxiliary processes for $m\in\{0,\ldots,M\}$, $M\in\mathbb{N}$:
\begin{align*}
  \bar{X}_{t_m} &:= P_N\left( e^{At_m}X_0 + \sI \e{t_m-t_l}F(X_{t_l})\, \mathrm{d}s + \sI \e{t_m-t_l}B(X_{t_l})\, \mathrm{d}W^K_s \right. \\  
  &\quad \left.+\sI \e{t_m-t_l} B'(X_{t_l}) \left(\int_{t_l}^s P_N B(X_{t_l}) \, \mathrm{d}W_r^K\right)\, \mathrm{d}W_s^K\right), \\
  \bar{Y}_{t_m} &:= P_N \left( e^{At_m}X_0 + \sI \e{t_m-t_l}F(Y_l)\, \mathrm{d}s + \sI \e{t_m-t_l}B(Y_l)\, \mathrm{d}W^K_s \right. \\  
  &\quad \left. +\sI \e{t_m-t_l} B'(Y_l) \left(\int_{t_l}^s P_N B(Y_l)\, \mathrm{d}W_r^K\right)\, \mathrm{d}W_s^K\right) \\
  &= P_N \Bigg( e^{At_m}X_0 + \sI \e{t_m-t_l}F(Y_l)\, \mathrm{d}s + \sI \e{t_m-t_l}B(Y_l)\, \mathrm{d}W^K_s  \\
  &\quad +\su \e{t_m-t_l} \Bigg(\frac{1}{2} B'(Y_l) \left( P_N B(Y_l)\Delta W_l^K,\Delta W_l^K\right) -\frac{h}{2}\sum_{\substack{j \in \mathcal{J}_K \\ 
  \eta_j \neq 0}}\eta_j B'(Y_l) \left( P_N B(Y_l)\tilde{e}_j,\tilde{e}_j\right)\Bigg)\Bigg).
\end{align*}
We estimate
\begin{align*}
  \E{X_{t_m} -\Y_m} =  \E{X_{t_m}-P_NX_{t_m}+P_NX_{t_m}-\bar{X}_{t_m}+\bar{X}_{t_m} -\bar{Y}_{t_m} +\bar{Y}_{t_m}-\Y_m}
\end{align*}
for all $m\in \{0,\ldots,M\}$, $N,M \in \mathbb{N}$, in several parts:
\begin{equation} \label{ZZ}
  \begin{split}
  \E{X_{t_m}-\Y_m} &\leq 4 \Big(\E{X_{t_m}-P_NX_{t_m}} +\E{P_NX_{t_m}-\bar{X}_{t_m}} \\
  &\quad + \E{\bar{X}_{t_m} -\bar{Y}_{t_m}}  + \E{\bar{Y}_{t_m}-\Y_m}\Big). 
  \end{split}
\end{equation}
The first part is the error that results from the projection of $H$ to a finite dimensional subspace $H_N$, $N\in\mathbb{N}$. 
The second and third terms arise due to the approximation of the solution process with the Milstein scheme 
and the last one is the error that we obtain by approximating the derivative.
After estimating these terms separately, we obtain
\begin{align*}
  \E{X_{t_m}-\Y_m} &\leq C_T\Big(\inf_{i\in \mathcal{I}\setminus\mathcal{I}_N}\lambda_i\Big)^{-2\gamma}
  + C_T\Big(\Big(\sup_{j\in \mathcal{I}\setminus\mathcal{I}_K}\eta_j\Big)^{2\alpha}
  + M^{-2\min(2(\gamma-\beta),\gamma)}\Big) \\
  &\quad + \frac{C_T}{M} \sum_{l=0}^{m-1} \E{X_{t_l}-\Y_l} + C_{T}M^{-2}(\tr Q)^4 \\
  &\leq C_{T,Q}\Big(\Big(\inf_{i\in \mathcal{I}\setminus\mathcal{I}_N}\lambda_i\Big)^{-2\gamma}
  + \Big(\sup_{j\in \mathcal{I}\setminus\mathcal{I}_K}\eta_j\Big)^{2\alpha}+M^{-2\min(2(\gamma-\beta),\gamma)}\Big)
\end{align*}
for all $m\in \{1,\ldots,M\}$, $N,K,M \in \mathbb{N}$, by a discrete version of Gronwall's lemma.

%
The estimates of the first three terms are not specific 
to our scheme and the ideas originate from \cite{MR3320928}. 
However, there are some modifications necessary in order to handle the projection operator
$P_N$ that we introduced.
The main idea, however, remains the same. For completeness, we state the whole proof. 
\subsection{Spectral Galerkin projection}\label{Sec:Projection}
The error resulting from the spectral Galerkin projection is
estimated for all $m \in\{0,\ldots,M\}$, $M,N\in\mathbb{N}$ as
\begin{align*}
  \mathrm{E}\big[\|X_{t_m}-P_NX_{t_m}\|_H^2\big]  
            &= \mathrm{E}\big[\|(I-P_N)X_{t_m}\|_H^2\big] \\
            &\leq \mathrm{E}\big[\|(I-P_N)(-A)^{-\gamma}\|_{L(H)}^2\big\|X_{t_m}\big\|_{H_{\gamma}}^2 \big] \\
            &= \sup_{\substack{y\in H \\ \|y\|_H=1}}  
            \|(I-P_N)(-A)^{-\gamma}y\|_{H}^2 \mathrm{E}\big[\|X_{t_m}\|_{H_{\gamma}}^2 \big] \\
            &= \sup_{\substack{y\in H \\ \|y\|_H=1}}
            \Big\|(I-P_N)\sum_{k\in\mathcal{I}}\lambda_k^{-\gamma}\langle y,e_k\rangle_H e_k\Big\|_{H}^2 
            \mathrm{E}\big[\|X_{t_m}\|_{H_{\gamma}}^2 \big] \\
            &= \sup_{\substack{y\in H \\ \|y\|_H=1}}  
            \Big\|\sum_{n\in\mathcal{I}\setminus\mathcal{I}_N} 
            \langle \sum_{k\in\mathcal{I}}\lambda_k^{-\gamma}\langle y,e_k\rangle_H e_k,e_n\rangle_He_n\Big\|_{H}^2 
            \mathrm{E}\big[\|X_{t_m}\|_{H_{\gamma}}^2 \big] .
\end{align*}
Due to (A1)--(A4) and Proposition~\ref{ExSol}, 
we further obtain
\begin{align*}
  \mathrm{E}\big[\|X_{t_m}-P_NX_{t_m}\|_H^2\big]  
  &= \sup_{\substack{y\in H \\ \|y\|_H=1}}  \Big\|\sum_{n\in\mathcal{I}\setminus\mathcal{I}_N}  
  \lambda_n^{-\gamma}\langle y,e_n\rangle_He_n\Big\|_{H}^2 \mathrm{E}\big[\|X_{t_m}\big\|_{H_{\gamma}}^2 ]\\
  &\leq C\sup_{\substack{y\in H \\ \|y\|_H=1}}  \sum_{n\in\mathcal{I}\setminus\mathcal{I}_N} 
  \lambda_n^{-2\gamma}\langle y,e_n\rangle_H^2\\
  &\leq C  \Big(\inf_{i\in\mathcal{I}\setminus\mathcal{I}_N}  \lambda_i\Big)^{-2\gamma}
  \sup_{\substack{y\in H \\ \|y\|_H=1}}  \sum_{n\in\mathcal{I}\setminus\mathcal{I}_N} \langle y,e_n\rangle_H^2\\
  &\leq C  \Big(\inf_{i\in\mathcal{I}\setminus\mathcal{I}_N}  \lambda_i\Big)^{-2\gamma} 
  \sup_{\substack{y\in H \\ \|y\|_H=1}}  \|y\|_H\\
  &= C \Big(\inf_{i\in\mathcal{I}\setminus\mathcal{I}_N}  \lambda_i\Big)^{-2\gamma}
\end{align*}
for all $m \in\{0,\ldots,M\}$, $M,N\in\mathbb{N}$. This proves the first part.\\ \\
In the following we use
\begin{align*}
  \|P_Nx\|^2_H &= \|\sum_{n\in\mathcal{I}_N} \langle x,e_n\rangle_He_n\|^2_H
  = \sum_{n\in\mathcal{I}_N} \langle x,e_n\rangle_H^2
  \leq  \sum_{n\in\mathcal{I}} |\langle x,e_n\rangle_H|^2 =\|x\|_H^2
\end{align*}
several times. 

%
In order to estimate the second term in~\eqref{ZZ}, we write
\begin{align*}
    \left(\E{P_NX_{t_m}-\bar{X}_{t_m}}\right)^{\frac{1}{2}} 
    &\leq \mathrm{E}\Bigg[\bigg\|\sI \left(\e{t_m-s}F(X_s)-\e{t_m-t_l}F(X_{t_l}) \right)
    \, \mathrm{d}s \bigg\|_H^2\Bigg]^{\frac{1}{2}}  \\
    &\quad \left.   +  \mathrm{E}\Bigg[\Bigg\|\sI \left(\e{t_m-s}B(X_s)-\e{t_m-t_l}B(X_{t_l}) \right)
    \, \mathrm{d}W^K_s \right. \\
    &\quad - \sI \e{t_m-t_l}B'(X_{t_l})\left(\int_{t_{l}}^{s}P_NB(X_{t_l})
    \, \mathrm{d}W_r^K\right)\, \mathrm{d}W^K_s\Bigg\|_H^2\Bigg]^{\frac{1}{2}} \\
    &\quad +  \mathrm{E}\Bigg[\bigg\|\sI \e{t_m-s}B(X_s)\left(\, \mathrm{d}W_s-\, \mathrm{d}W_s^K\right)
    \bigg\|_H^2\Bigg]^{\frac{1}{2}}
\end{align*}
for $m\in\{1,\ldots,M\}$, $M\in\mathbb{N}$.
\subsection{Temporal discretization - the nonlinearity F}
Next, we prove the error resulting from the temporal 
discretization of the Bochner integral by partitioning 
the error into three components which we again estimate separately.
Let $m\in\{1,\ldots,M\}$, $M\in\mathbb{N}$. We show
\begin{align*}
    &\bigg(\mathrm{E}\bigg[\Big\|\sI \left(e^{A(t_m-s)}F(X_s)-e^{A(t_m-t_l)}F(X_{t_l}) \right)
    \, \mathrm{d}s \Big\|_H^2\bigg]\bigg)^{\frac{1}{2}} \\
    &\leq \bigg(\mathrm{E}\bigg[\Big\|\sI e^{A(t_m-s)}\left(F(X_s)-F(X_{t_l}) \right)\, 
    \mathrm{d}s \Big\|_H^2\bigg]\bigg)^{\frac{1}{2}}\\
    &\quad +\bigg( \mathrm{E}\bigg[\Big\|\sII \left(e^{A(t_m-s)}-e^{A(t_m-t_l)}\right)F(X_{t_l}) 
    \, \mathrm{d}s \Big\|_H^2\bigg]\bigg)^{\frac{1}{2}} \\
    &\quad + \bigg(\mathrm{E}\bigg[\Big\|\int_{t_{m-1}}^{t_m}\left(e^{A(t_m-s)}-e^{A(t_m-t_{m-1})}\right)
    F(X_{t_{m-1}}) \, \mathrm{d}s
    \Big\|_H^2\bigg]\bigg)^{\frac{1}{2}}\\
    &\leq C_TM^{-\min(2(\gamma-\beta),\gamma)}.
\end{align*}
We define $\tilde{X}_{s,l}:=X_s-X_{t_l}$ for all $s\in[0,T]$, $l\in\{0,\ldots,M-1\}$, 
$M\in\mathbb{N}$, for legibility. For the first term, 
we obtain by the triangle inequality and the representation of the mild solution 
$(X_t)_{t\in[0,T]}$  
\begin{align}
  &\bigg(\mathrm{E}\bigg[\Big\|\sI e^{A(t_m-s)}\left(F(X_s)-F(X_{t_l}) \right)
  \, \mathrm{d}s \Big\|_H^2\bigg]\bigg)^{\frac{1}{2}} \nonumber \\
  &\leq \bigg(\mathrm{E}\bigg[\Big\|\sI e^{A(t_m-s)}F'(X_{t_l})(X_s-X_{t_l})
  \, \mathrm{d}s \Big\|_H^2\bigg]\bigg)^{\frac{1}{2}}  \nonumber\\
  &\quad + \bigg(\mathrm{E}\bigg[\Big\|\sI e^{A(t_m-s)}\Big( 
  \int_0^1 \int_0^r \frac{1}{2}  F''(X_{t_l}+u\tilde{X}_{s,l})(\tilde{X}_{s,l},\tilde{X}_{s,l})
  \, \mathrm{d}u \,\mathrm{d}r\Big)\, \mathrm{d}s \Big\|_H^2\bigg]\bigg)^{\frac{1}{2}}  \nonumber\\
  &\leq   \su \bigg( \mathrm{E}\bigg[\Big\|\I e^{A(t_m-s)}F'(X_{t_l})\left(e^{A(s-t_l)}-I\right)X_{t_l}
  \, \mathrm{d}s\Big\|_H^2\bigg] \bigg)^{\frac{1}{2}}  \nonumber\\
  &\quad + \su \bigg( \mathrm{E}\bigg[\Big\|\I e^{A(t_m-s)}F'(X_{t_l})\Big(\int_{t_l}^se^{A(s-u)}F(X_u)
  \, \mathrm{d}u\Big)\, \mathrm{d}s\Big\|_H^2\bigg] \bigg)^{\frac{1}{2}} \nonumber\\
  &\quad + \bigg(\su\mathrm{E}\bigg[\Big\|\I e^{A(t_m-s)}F'(X_{t_l})\Big(\int_{t_l}^se^{A(s-u)}B(X_u)
  \, \mathrm{d}W_u\Big)\, \mathrm{d}s \Big\|_H^2\bigg]\bigg)^{\frac{1}{2}} \nonumber \\
  &\quad + \su \bigg(\mathrm{E}\bigg[\Big\|\I e^{A(t_m-s)}\Big( \int_0^1 \int_0^r\frac{1}{2} 
  F''(X_{t_l}+u\tilde{X}_{s,l})(\tilde{X}_{s,l},\tilde{X}_{s,l})\, \mathrm{d}u
  \,\mathrm{d}r\Big)\, \mathrm{d}s \Big\|_H^2\bigg]\bigg)^{\frac{1}{2}}.  \nonumber
\end{align}
Then, H\"older's inequality implies 
\begin{align*}
      &\bigg(\mathrm{E}\bigg[\Big\|\sI e^{A(t_m-s)}\left(F(X_s)-F(X_{t_l}) \right)
      \, \mathrm{d}s \Big\|_H^2\bigg]\bigg)^{\frac{1}{2}} \\
      &\leq   \su \bigg( \mathrm{E}\bigg[h \I \Big\|e^{A(t_m-s)}F'(X_{t_l})
      \left(e^{A(s-t_l)}-I\right)X_{t_l}\Big\|_H^2\, \mathrm{d}s\bigg] \bigg)^{\frac{1}{2}}  \\
      &\quad +  \su \bigg( \mathrm{E}\bigg[h\I \Big\|e^{A(t_m-s)}F'(X_{t_l})
      \Big(\int_{t_l}^se^{A(s-u)}F(X_u)\, \mathrm{d}u\Big)\Big\|_H^2\, \mathrm{d}s\bigg] \bigg)^{\frac{1}{2}} \\
      &\quad + \bigg(\su \mathrm{E}\bigg[h\I \Big\|e^{A(t_m-s)}F'(X_{t_l})
      \Big(\int_{t_l}^se^{A(s-u)}B(X_u)\, \mathrm{d}W_u\Big)\Big\|_H^2\, \mathrm{d}s\bigg]\bigg)^{\frac{1}{2}}\\
      &\quad + \su \bigg(\mathrm{E}\bigg[h\I \Big\|e^{A(t_m-s)}
      \Big( \int_0^1 \int_0^r  F''(X_{t_l}+u\tilde{X}_{s,l})(X_s-X_{t_l},X_s-X_{t_l}) 
      \, \mathrm{d}u \,\mathrm{d}r\Big)\, \mathrm{d}s \Big\|_H^2\bigg]\bigg)^{\frac{1}{2}}
\end{align*}
and by (A2), Theorem~\ref{SGest}, and Proposition~\ref{ExSol}, we get
\begin{align*}
    &\bigg(\mathrm{E}\bigg[\Big\|\sI e^{A(t_m-s)}\left(F(X_s)-F(X_{t_l}) \right)
    \, \mathrm{d}s \Big\|_H^2\bigg]\bigg)^{\frac{1}{2}} \\
    &\leq   C \su \bigg( \mathrm{E}\bigg[h \I \big\|F'(X_{t_l})\big\|_{L(H)}^2
    \big\|(-A)^{-\gamma}\big(e^{A(s-t_l)}-I\big)\big\|_{L(H)}^2\|X_{t_l} \|_{H_{\gamma}}^2\, \mathrm{d}s\bigg] \bigg)^{\frac{1}{2}}  \\
    &\quad  +  C\su \bigg( \mathrm{E}\bigg[h\I \big\|F'(X_{t_l})
    \big\|_{L(H)}^2\Big\|\int_{t_l}^se^{A(s-u)}F(X_u)\, \mathrm{d}u\
    \Big\|_H^2\, \mathrm{d}s\bigg] \bigg)^{\frac{1}{2}} \\
    &\quad + C\bigg(\su h \, \mathrm{E}\bigg[\I 
    \big\|F'(X_{t_l})\big\|_{L(H)}^2\Big\|\int_{t_l}^se^{A(s-u)}B(X_u)\, \mathrm{d}W_u\Big\|_H^2\,
    \mathrm{d}s\bigg]\bigg)^{\frac{1}{2}}  \\
    &\quad +C \su \bigg(\mathrm{E}\bigg[h\I 
    \int_0^1\int_0^r\| F''(X_{t_l}+u\tilde{X}_{s,l})\|_{L^{(2)}(H_{\beta},H)} 
    \,\mathrm{d}u \,\mathrm{d}r\|X_s-X_{t_l}\|_{H_{\beta}}^4 \, \mathrm{d}s \bigg]\bigg)^{\frac{1}{2}}\\
    &\leq   C \su \bigg( h \I(s-t_l)^{2\gamma} \, \mathrm{E}\big[\|X_{t_l} \|_{H_{\gamma}}^2\big]
    \, \mathrm{d}s \bigg)^{\frac{1}{2}}  \\
    &\quad  + C \su \bigg( \mathrm{E}\bigg[h\I \Big\|\int_{t_l}^se^{A(s-u)}F(X_u)\, \mathrm{d}u\Big\|_H^2
    \, \mathrm{d}s\bigg] \bigg)^{\frac{1}{2}} \\
    &\quad + C \bigg(\su h \, \mathrm{E}\bigg[\I \Big\|\int_{t_l}^se^{A(s-u)}B(X_u)\, \mathrm{d}W_u\Big\|_H^2
    \, \mathrm{d}s\bigg]\bigg)^{\frac{1}{2}}\\
    &\quad + C\su \bigg(h\I (s-t_l)^{4\min(\gamma-\beta,\frac{1}{2})} \, \mathrm{d}s\bigg)^{\frac{1}{2}}.
\end{align*}
Then, (A1)--(A4) and It\^{o}'s isometry imply
\begin{align*}
      &\bigg(\mathrm{E}\bigg[\Big\|\sI e^{A(t_m-s)}\left(F(X_s)-F(X_{t_l}) \right)
      \, \mathrm{d}s \Big\|_H^2\bigg]\bigg)^{\frac{1}{2}} \\
      &\leq C Mh^{1+\gamma} +  C \su \bigg(h\I (s-t_l)^2 \, \mathrm{d}s\bigg)^{\frac{1}{2}}  \\
      &\quad + \bigg(C \su h \I \int_{t_l}^s\mathrm{E}\big[\|(-A)^{-\delta}\|^2_{L(H)}
      \left\|B(X_u)\right\|_{L_{HS}(U_0,H_{\delta})}^2\big]\, \mathrm{d}u\, \mathrm{d}s\bigg)^{\frac{1}{2}}
      + C \su \Big( h^{4\min(\gamma-\beta,\frac{1}{2})+2} \Big)^{\frac{1}{2}}\\
      &\leq  C_Th^{\gamma}+ CMh^2 + \bigg( C\su h \I (s-t_l)\, \mathrm{d}s\bigg)^{\frac{1}{2}}
      + C h^{\min(2(\gamma-\beta),1) } \\
      &\leq  C_Th^{\gamma}+ C_T h + C\big(Mh^3\big)^{\frac{1}{2}}
      \leq  C_Th^{\min(2(\gamma-\beta),\gamma)}
\end{align*}
for all $m\in\{1,\ldots,M\}$, $M\in\mathbb{N}$.

%
The estimates of the second and third part follow easily by 
the triangle inequality,  H\"older's inequality, (A1)--(A4), and Theorem \ref{SGest} as well.
For all $m\in\{2,\ldots,M\}$, $M\in\mathbb{N}$, we get
\begin{align*}
         & \bigg( \mathrm{E}\bigg[\Big\|\sII \Big(e^{A(t_m-s)}-e^{A(t_m-t_l)}\Big)F(X_{t_l}) 
         \, \mathrm{d}s \Big\|_H^2\bigg]\bigg)^{\frac{1}{2}}\\
         &\leq \s \bigg( \mathrm{E}\bigg[\Big\| \I\Big(e^{A(t_m-s)}-e^{A(t_m-t_l)}\Big)F(X_{t_l}) 
         \, \mathrm{d}s\Big\|_H^2\bigg] \bigg)^{\frac{1}{2}}\\
        &\leq C \s  \bigg( h\I \big\|(-A)e^{A(t_m-s)}\big\|_{L(H)}^2
        \big\|(-A)^{-1}\big(I-e^{A(s-t_l)}\big)\big\|_{L(H)}^2 \, \mathrm{d}s \bigg)^{\frac{1}{2}}\\
           &\leq C \s\bigg(h \I\Big(\frac{s-t_l}{t_m-s}\Big)^2\, \mathrm{d}s\bigg)^{\frac{1}{2}} \\
           &\leq C \s \bigg(h\I \Big(\frac{s-t_l}{(m-l-1)h}\Big)^2\, \mathrm{d}s\bigg)^{\frac{1}{2}} \\
           &= C \s \Big(\frac{h^4}{(m-l-1)^2h^2}\Big)^{\frac{1}{2}} 
           = Ch\sum_{l=0}^{m-2}\frac{1}{m-l-1} = Ch\sum_{l=1}^{m-1}\frac{1}{l} \\
           &\leq C\frac{1+\ln(M)}{M}
           \leq C\frac{M^{1-\gamma}}{M(1-\gamma)} = C h^{\gamma}.
\end{align*}
In the last step, we employed some basic computations for $m\in\{1,\ldots,M\}$, $M\in\mathbb{N}$
\begin{equation*}
  \sum_{l=1}^{m-1} \frac{1}{l} = 1+\sum_{l=2}^{m-1} \frac{1}{l}\leq 1+ \sum_{l=2}^{M} \frac{1}{l}
  \leq 1+ \int_1^M \frac{1}{s} \, \mathrm{d}s = 1+\ln(M)
\end{equation*}
and for all $r\in[0,1)$ and $x \geq 1$, we get
\begin{equation*}
  1+\ln(x) = 1+\int_1^x s^{-1}\, \mathrm{d}s
  \leq 1+\int_1^x \frac{1}{s^{1-r}}\, \mathrm{d}s 
  = 1+\frac{x^r-1}{r} =\frac{x^r}{r}-\frac{(1-r)}{r}\leq \frac{x^r}{r},
\end{equation*}
see \cite{MR3320928}.
Further, we obtain
\begin{align*}
        &\bigg(\mathrm{E}\bigg[\Big\|\int_{t_{m-1}}^{t_m}
        \Big(e^{A(t_m-s)}-e^{A(t_m-t_{m-1})}\Big)F(X_{t_{m-1}}) \, \mathrm{d}s \Big\|_H^2\bigg]\bigg)^{\frac{1}{2}} \\
        &\leq\sqrt{h}\bigg(\int_{t_{m-1}}^{t_m} 
        \mathrm{E}\Big[\big\|\big(e^{A(t_m-s)}-e^{A(t_m-t_{m-1})}\big)F(X_{t_{m-1}}) \big\|_H^2\Big]
        \, \mathrm{d}s\bigg)^{\frac{1}{2}}\\
        &\leq  \sqrt{h} \Big(\int_{t_{m-1}}^{t_m}C\, \mathrm{d}s \Big)^{\frac{1}{2}} \leq C_T h
\end{align*}
for all $m\in\{1,\ldots,M\}$, $M\in\mathbb{N}$.
%
%
\subsection{Temporal discretization with Milstein scheme - the diffusion B}
For the estimation of the error resulting from the discretization 
of the stochastic integrals, we compute for all $m\in\{1,\ldots,M\}$, $M,K\in\mathbb{N}$
\begin{align}\label{BSplit}
  & \mathrm{E}\bigg[\Big\|\sI \Big(e^{A(t_m-s)}B(X_s)-e^{A(t_m-t_l)}B(X_{t_l}) \Big)\, \mathrm{d}W^K_s  \nonumber \\
  &\quad - \sI e^{A(t_m-t_l)} B'(X_{t_l})\Big(\int_{t_{l}}^{s}P_NB(X_{t_l})
  \, \mathrm{d}W_r^K\Big)\, \mathrm{d}W^K_s\Big\|_H^2\bigg] \nonumber \\
  &\leq \sum_{l=0}^{m-1}  \mathrm{E}\bigg[ \Big\|\I e^{A(t_m-t_l)}
  \left(B(X_{s})-B(X_{t_l})\right) \, \mathrm{d}W^K_s \nonumber\nonumber \\
  &\quad -  \I e^{A(t_m-t_l)}B'(X_{t_l})\Big(\int_{t_{l}}^{s}P_NB(X_{t_l})\, \mathrm{d}W_r^K\Big)
  \, \mathrm{d}W^K_s\Big\|_H^2\bigg] \nonumber\\
  &\quad +  \mathrm{E}\bigg[\Big\| \sII \Big(e^{A(t_m-s)}-e^{A(t_m-t_l)}\Big)B(X_{s}) 
  \, \mathrm{d}W^K_s\Big\|_H^2\bigg]  \nonumber \\
  &\quad + \mathrm{E}\bigg[\Big\|\int_{t_{m-1}}^{t_m} \Big(e^{A(t_m-s)}-e^{A(t_m-t_{m-1})}\Big)B(X_{s})
  \, \mathrm{d}W_s^K\Big\|_H^2\bigg]\nonumber\\
  &\leq  C_T \Big(M^{-2\gamma} + \Big(\sup_{j\in \mathcal{J}\setminus\mathcal{J}_K}\eta_j\Big)^{2\alpha}
          +\Big(\inf_{i\in\mathcal{I}\setminus\mathcal{I}_N}\lambda_i\Big)^2\Big),
\end{align}
where
\begin{align*}
         & \sum_{l=0}^{m-1} \mathrm{E}\bigg[\Big\|\I e^{A(t_m-t_l)}
         \left(B(X_{s})-B(X_{t_l})\right) \, \mathrm{d}W^K_s \\
         &\quad - \I e^{A(t_m-t_l)} B'(X_{t_l})\Big(\int_{t_{l}}^{s}P_NB(X_{t_l})
         \, \mathrm{d}W_r^K\Big)\, \mathrm{d}W^K_s\Big\|_H^2\bigg]\\
         &= \sum_{l=0}^{m-1} \mathrm{E}\bigg[\Big\|\I e^{A(t_m-t_l)}\Big(B'(X_{t_l})(X_s-X_{t_l})\\
         &\quad  + \int_0^1\Big( \int_0^r
         B''(X_{t_l}+u(X_s-X_{t_l}))\big(X_s-X_{t_l},X_s-X_{t_l}\big)
         \, \mathrm{d}u \Big)\, \mathrm{d}r\Big)\, \mathrm{d}W^K_s \\
         &\quad - \I e^{A(t_m-t_l)} B'(X_{t_l})
         \Big(\int_{t_{l}}^{s}P_NB(X_{t_l})\, \mathrm{d}W_r^K\Big)\, \mathrm{d}W^K_s\Big\|_H^2\bigg]\\
         &\leq \sum_{l=0}^{m-1} \I \mathrm{E}\bigg[\Big\|e^{A(t_m-t_l)}B'(X_{t_l})\Big((X_s-X_{t_l})
         -\int_{t_{l}}^{s}P_NB(X_{t_l})\, \mathrm{d}W_r^K\Big) \\
         &\quad + e^{A(t_m-t_l)} \int_0^1\Big( 
         \int_0^r B''(X_{t_l}+u(X_s-X_{t_l}))\big(X_s-X_{t_l},X_s-X_{t_l}\big) 
         \, \mathrm{d}u \Big) \, \mathrm{d}r \Big\|_{L_{HS}(U_0,H)}^2\bigg]\, \mathrm{d}s
\end{align*}
due to It\^{o}'s isometry.

%
%
With Lemma~\ref{SGest} and Proposition~\ref{ExSol}, we obtain 
\begin{align*}
    &\sum_{l=0}^{m-1} \mathrm{E}\bigg[\Big\|\I e^{A(t_m-t_l)} \left(B(X_{s})
    -B(X_{t_l})\right) \, \mathrm{d}W^K_s \\
    &\quad - \I e^{A(t_m-t_l)} B'(X_{t_l})\Big(\int_{t_{l}}^{s}P_NB(X_{t_l})
    \, \mathrm{d}W_r^K\Big)\, \mathrm{d}W^K_s\Big\|_H^2\bigg]\\
    &\leq 2 \sum_{l=0}^{m-1} \I \mathrm{E}\bigg[\Big\|e^{A(t_m-t_l)}
    B'(X_{t_l})\Big((X_s-X_{t_l})-\Big(\int_{t_{l}}^{s}P_NB(X_{t_l})
    \, \mathrm{d}W_r^K\Big)\Big) \Big\|_{L_{HS}(U_0,H)}^2\bigg]\, \mathrm{d}s\\
    &\quad + 2 \sum_{l=0}^{m-1} \I \mathrm{E}\bigg[\big\|e^{A(t_m-t_l)}\big\|^2_{L(H)}
         \left\|X_s-X_{t_l} \right\|_{H}^4 \\
    &\quad \quad \times \int_0^1\Big( \int_0^r 
         \big\| B''(X_{t_l}+u(X_s-X_{t_l}))\big\|^2_{L^{(2)}(H,L_{HS}(U_0,H))}   
         \, \mathrm{d}u \, \Big)\,r \, \mathrm{d}r \bigg] \, \mathrm{d}s\\
    &\leq C\sum_{l=0}^{m-1} \bigg(\I \mathrm{E}\bigg[\Big\|e^{A(t_m-t_l)}
          B'(X_{t_l})\Big((X_s-X_{t_l})-\Big(\int_{t_{l}}^{s}P_NB(X_{t_l})\, \mathrm{d}W_r^K\Big)\Big) 
          \Big\|_{L_{HS}(U_0,H)}^2\bigg]\, \mathrm{d}s \\
    &\quad +  \frac{h^{1+\min(4\gamma,2)}}{1+\min(4\gamma,2)}\bigg).
\end{align*}
The following part differs from the estimate in the proof given in~\cite{MR3320928}.
We plug in the expression for the mild solution and use (A3) in order 
to obtain
\begin{align*}
    &\sum_{l=0}^{m-1} \mathrm{E}\bigg[\Big\|\I e^{A(t_m-t_l)}\left(B(X_{s})-B(X_{t_l})\right) \, \mathrm{d}W^K_s \\
    &\quad  - \I e^{A(t_m-t_l)}B'(X_{t_l})\Big(
         \int_{t_l}^{s}P_NB(X_{t_l})\, \mathrm{d}W_r^K\Big)\, \mathrm{d}W^K_s\Big\|_H^2\bigg]\\
    &\leq   C \sum_{l=0}^{m-1} \bigg(\I \mathrm{E}\bigg[\Big\|e^{A(t_m-t_l)}
        B'(X_{t_l})\Big(\big(e^{A(s-t_l)}-I\big)X_{t_l} + \int_{t_l}^s e^{A(s-u)}F(X_u)\, \mathrm{d}u\\
    &\quad +\int_{t_l}^se^{A(s-u)}B(X_u)\, \mathrm{d}(W_u-W_u^K) 
         + \int_{t_l}^se^{A(s-u)}\big(B(X_u)-P_NB(X_{t_l})\big)\, \mathrm{d}W_u^K \\
    &\quad + \int_{t_l}^s\big(e^{A(s-u)}-I\big)P_NB(X_{t_l})
          \, \mathrm{d}W_u^K\Big) \Big\|_{L_{HS}(U_0,H)}^2\bigg]\, \mathrm{d}s +  h^{1+\min(4\gamma,2)}\bigg)\\
    &\leq  C\sum_{l=0}^{m-1} \bigg(\I \mathrm{E}\Big[\big\|\big(e^{A(s-t_l)}
        -I\big)X_{t_l}\big\|_H^2\Big]\, \mathrm{d}s + \I \mathrm{E}\bigg[\Big\|\int_{t_l}^se^{A(s-u)}F(X_u)
        \, \mathrm{d}u\Big\|_H^2 \bigg]\, \mathrm{d}s \\
    &\quad + \I \mathrm{E}\bigg[\Big\|\int_{t_l}^se^{A(s-u)}B(X_u)
        \, \mathrm{d}(W_u-W_u^K) \Big\|_H^2\bigg]\, \mathrm{d}s \\
    &\quad +\I \mathrm{E}\bigg[\Big\| \int_{t_l}^se^{A(s-u)}\big(B(X_u)-P_NB(X_{t_l})\big)
        \, \mathrm{d}W_u^K \Big\|_H^2\bigg]\, \mathrm{d}s \\
    &\quad + \I  \mathrm{E}\bigg[\Big\| \int_{t_l}^s\big(e^{A(s-u)}-I\big)P_NB(X_{t_l})
          \, \mathrm{d}W_u^K \Big\|_H^2\bigg]\, \mathrm{d}s +  h^{1+\min(4\gamma,2)}\bigg).\\ 
\end{align*}
The proof of
\begin{equation*}
  \I \mathrm{E}\bigg[\Big\|\int_{t_l}^se^{A(s-u)}B(X_u) 
  \, \mathrm{d}(W_u-W_u^K) \Big\|_H^2\bigg]\, \mathrm{d}s\leq C_Th 
  \Big(\sup_{j\in \mathcal{J}\setminus\mathcal{J}_K}\eta_j\Big)^{2\alpha},
\end{equation*}
for all $l\in\{0,\ldots,M-1\}$, $M,K\in\mathbb{N}$, can be found in the 
next part in Section~\ref{QApp}. 

%
With Lemma~\ref{SGest}, (A1)--(A4), by H\"older's inequality, and It\^{o}'s isometry,
we obtain
\begin{align*}
    &\sum_{l=0}^{m-1} \mathrm{E}\bigg[\Big\|\I e^{A(t_m-t_l)}\left(B(X_{s})-B(X_{t_l})\right) \, \mathrm{d}W^K_s \\
    &\quad  - \I e^{A(t_m-t_l)}B'(X_{t_l})\Big(
         \int_{t_l}^{s}P_NB(X_{t_l})\, \mathrm{d}W_r^K\Big)\, \mathrm{d}W^K_s\Big\|_H^2\bigg]\\
    &\leq C \sum_{l=0}^{m-1} \bigg(\I \big\|(-A)^{-\gamma}
    \big(e^{A(s-t_l)}-I\big)\big\|_{L(H)}^2\mathrm{E}\big[\|(-A)^{\gamma}X_{t_l}\|_H^2\big] \, \mathrm{d}s \\
    &\quad + \I (s-t_l) \Big(\int_{t_l}^s\mathrm{E}\big[\|e^{A(s-u)}F(X_u)\|_H^2 \big]
    \, \mathrm{d}u \Big)\, \mathrm{d}s + C_Th\,
    \Big(\sup_{j\in \mathcal{J}\setminus\mathcal{J}_K}\eta_j\Big)^{2\alpha} \\   
    &\quad + \I \Big(\int_{t_l}^s\mathrm{E}\big[\| e^{A(s-u)}
    \big(I-P_N\big)B(X_u)\|_{L_{HS}(U_0,H)}^2\big]\, \mathrm{d}u\Big)\, \mathrm{d}s   \\     
    &\quad + \I \Big(\int_{t_l}^s\mathrm{E}\big[\| e^{A(s-u)}
    P_N\big(B(X_u)-B(X_{t_l})\big)\|_{L_{HS}(U_0,H)}^2\big]
    \, \mathrm{d}u\Big)\, \mathrm{d}s \\     
    &\quad +\I  \Big(\int_{t_l}^s\big\|(-A)^{-\delta}\big(e^{A(s-u)}-I\big)\big\|_{L(H)}^2
    \mathrm{E}\big[\| (-A)^{\delta}P_NB(X_{t_l})\|_{L_{HS}(U_0,H)}^2\big]\, \mathrm{d}u\Big)\, \mathrm{d}s 
    + h^{1+\min(4\gamma,2)}\bigg)\\
    &\leq C \sum_{l=0}^{m-1} \bigg(\I(s-t_l)^{2\gamma} \,\mathrm{E}\big[\|(-A)^{\gamma}X_{t_l}\|_H^2\big] 
    \, \mathrm{d}s \\
    &\quad + \I (s-t_l) \Big(\int_{t_l}^sC \, \mathrm{E}\big[\|F(X_u)\|_H^2 \big]\, \mathrm{d}u \Big)
    \, \mathrm{d}s + C_Th\Big(\sup_{j\in \mathcal{J}\setminus\mathcal{J}_K}\eta_j\Big)^{2\alpha} \\        %
    &\quad + \I \Big(\int_{t_l}^s
    \mathrm{E}\big[\| (-A)^{-\gamma}\big(I-P_N\big)\|_{L(H)}^2
    \|e^{A(s-u)}(-A)^{\gamma-\delta}\|_{L(H)}^2 
    \|(-A)^{\delta}B(X_u)\|_{L_{HS}(U_0,H)}^2\big]\, \mathrm{d}u\Big)\, \mathrm{d}s   \\     
    &\quad + \I \Big(\int_{t_l}^s\mathrm{E}\big[\| e^{A(s-u)}\|_{L(H)}^2
    \| P_N||^2_{L(H)}\|B(X_u)-B(X_{t_l})\|_{L_{HS}(U_0,H)}^2\big]
    \, \mathrm{d}u\Big)\, \mathrm{d}s \\     
    &\quad +\I  \Big(\int_{t_l}^s(s-u)^{2\delta}
    \mathrm{E}\big[\|B(X_{t_l})\|_{L_{HS}(U_0,H_{\delta})}^2\big]\, \mathrm{d}u\Big)\, \mathrm{d}s
    + h^{1+\min(4\gamma,2)}\bigg).     
\end{align*}
This expression can be simplified further by Lemma~\ref{SGest} and
Section~\ref{Sec:Projection}, which implies
\begin{align*}
    &\sum_{l=0}^{m-1} \mathrm{E}\bigg[\Big\|\I e^{A(t_m-t_l)}\left(B(X_{s})-B(X_{t_l})\right) \, \mathrm{d}W^K_s \\
    &\quad  - \I e^{A(t_m-t_l)}B'(X_{t_l})\Big(
         \int_{t_l}^{s}P_NB(X_{t_l})\, \mathrm{d}W_r^K\Big)\, \mathrm{d}W^K_s\Big\|_H^2\bigg]\\
    &\leq C_Q \sum_{l=0}^{m-1} \bigg( h^{2\gamma+1} + h^3 
    + C_T \, h\Big(\sup_{j\in \mathcal{J}\setminus\mathcal{J}_K}\eta_j\Big)^{2\alpha}   \\   
    &\quad + \Big(\inf_{i\in\mathcal{I}\setminus\mathcal{I}_N}\lambda_i\Big)^{-2 \gamma} \I \Big(\int_{t_l}^s
    (s-u)^{-2(\gamma-\delta)} \big(1+ \mathrm{E}\big[\|X_u\|_{H_{\delta}}^2\big]\big)\, \mathrm{d}u\Big)\, \mathrm{d}s   \\     
    &\quad + \I \Big(\int_{t_l}^s(u-t_l)^{\min(2\gamma,1)}\, \mathrm{d}u\Big)\, \mathrm{d}s 
    +\I \Big(\int_{t_l}^s (s-u)^{2\delta}\, \mathrm{d}u\Big)\, \mathrm{d}s 
    + h^{1+\min(4\gamma,2)} \bigg)\\ 
    &\leq  C_Q\sum_{l=0}^{m-1} \Big(h^{2\gamma+1} + h^3 
    + C_Th\Big(\sup_{j\in \mathcal{J}\setminus\mathcal{J}_K}\eta_j\Big)^{2\alpha} 
    + \Big(\inf_{i\in\mathcal{I}\setminus\mathcal{I}_N} \lambda_i \Big)^{-2\gamma} h \\
    &\quad 
    + h^{\min(2\gamma,1)+2} + h^{2\delta+2} + h^{1+\min(4\gamma,2)}\Big)\\
    &\leq  C_{T,Q}\Big(\Big(\sup_{j\in \mathcal{J}\setminus\mathcal{J}_K}\eta_j\Big)^{2\alpha} 
    +\Big(\inf_{i\in\mathcal{I}\setminus\mathcal{I}_N} \lambda_i \Big)^{-2\gamma}
    + h^{2\gamma}\Big),
\end{align*}
where we also used $\gamma-\delta\in[0,\frac{1}{2})$ and $2+\min(2\gamma,1) \geq 1+ \min(4\gamma,2)$. 

%
The second term in \eqref{BSplit} is estimated 
for all $m\in\{1,\ldots,M\}$, $M,K\in\mathbb{N}$, 
using the independence of the increments of the $Q$-Wiener process in time, 
the It\^{o} isometry, Proposition~\ref{ExSol}, and (A1)--(A4) \phantomsection\label{EstB}
\begin{align*}
        &\mathrm{E}\bigg[\Big\|\s \int_{t_l}^{t_{l+1}}
        \Big(e^{A(t_m-s)}-e^{A(t_m-t_l)}\Big)B(X_{s}) \, \mathrm{d}W^K_s\Big\|_H^2\bigg]\\
        &= \s \mathrm{E}\bigg[\Big\|\I\Big(e^{A(t_m-s)}-e^{A(t_m-t_l)}\Big)B(X_{s})
        \, \mathrm{d}W^K_s\Big\|_H^2\bigg]\\
        &\leq  \s \I \big\|(-A)^{-\delta}\big(e^{A(t_m-s)}-e^{A(_m-t_l)}\big)\big\|_{L(H)}^2 
        \mathrm{E}\big[\|(-A)^{\delta}B(X_{s})\|_{L_{HS}(U_0,H)}^2\big]\, \mathrm{d}s\\
        &\leq \s \I\big\|(-A)^{1-\delta}e^{A(t_m-s)}\big\|_{L(H)}^2 
        \big\|(-A)^{-1}\big(I-e^{A(s-t_l)}\big)\big\|^2_{L(H)} 
        \Erw \big[ \| B(X_s) \|_{L_{HS}(U_0,H_{\delta})}^2 \big] \, \mathrm{d}s\\
        &\leq C_Q h^2\s\I (t_m-s)^{2(\delta-1)}  \, \mathrm{d}s \\
        &= C_Q h^2\s \left((t_m-t_{l+1})^{2\delta-1} - (t_m-t_{l})^{2\delta-1}\right)
        = C_Q h^2 \big((t_m-t_{m-1})^{2\delta-1} - (t_m)^{2\delta-1}\big) \\
        &\leq C_{T,Q} h^{2\delta+1}\leq C_Th^{2\gamma}.
\end{align*}
Finally, we obtain by conditions (A1), (A3), Lemma~\ref{SGest}, and Proposition~\ref{ExSol} 
for all $m\in\{1,\ldots,M\}$, $M,K\in\mathbb{N}$
\begin{align*}
      & \mathrm{E}\bigg[\Big\|\int_{t_{m-1}}^{t_m} 
      \left(e^{A(t_m-s)}-e^{A(t_m-t_{m-1})}\right)B(X_{s})\, \mathrm{d}W_s^K\Big\|_H^2\bigg]\\
      &\leq C\int_{t_{m-1}}^{t_m} \|e^{A(t_m-s)}\|^2_{L(H)}
      \big\|(-A)^{-\delta}\big(I-e^{A(s-t_{m-1})}\big)\big\|_{L(H)}^2 
      \mathrm{E}\big[\|(-A)^{\delta}B(X_{s})\|_{L_{HS}(U_0,H)}^2\big] \, \mathrm{d}s \\
      &\leq Ch^{2\delta+1} \leq Ch^{2\gamma}.
\end{align*}
\subsection{Approximation of the Q-Wiener process}\label{QApp}
Next, we prove the error estimate resulting from the approximation of the $Q$-Wiener process and employ
\begin{equation*}
  \mathrm{d} (W_s-W_s^K)
  = \sum_{j\in\mathcal{J}\setminus{\mathcal{J}_K} }\sqrt{\eta_j}\tilde{e}_j\, \mathrm{d}\beta_s^j
 \end{equation*}
for all $s\in[0,T]$, $K\in\mathbb{N}$. 

%
%
For all $l\in\{0,\ldots,M-1\}$, $M,K\in\mathbb{N}$, $s\in[0,T]$, it holds
\begin{align*}
    &\mathrm{E}\bigg[\Big\|\int_{t_l}^s e^{A(s-u)}B(X_u) \,
    \mathrm{d}(W_u-W_u^K) \Big\|_H^2\bigg]^{\frac{1}{2}} \\
    &= \mathrm{E}\bigg[\Big\| \sum_{j\in\mathcal{J}\setminus{\mathcal{J}_K} }
    \int_{t_l}^s e^{A(s-u)}B(X_u)\sqrt{\eta_j} \, \mathrm{d}\beta_u^j\tilde{e}_j \Big\|_H^2\bigg]^{\frac{1}{2}}\\
    &= \bigg(\sum_{j\in\mathcal{J}\setminus{\mathcal{J}_K}}\eta_j\int_{t_l}^s 
    \mathrm{E}\Big[\big\|e^{A(s-u)}B(X_u)Q^{-\alpha}Q^{\alpha}\tilde{e}_j \big\|_H^2\Big]\,
    \mathrm{d}u\bigg)^{\frac{1}{2}}\\
    &= \bigg(\sum_{j\in\mathcal{J}\setminus{\mathcal{J}_K} }\eta_j^{2\alpha+1}\int_{t_l}^s 
    \mathrm{E}\Big[\big\|e^{A(s-u)}B(X_u)Q^{-\alpha}\tilde{e}_j \big\|_H^2\Big]\, \mathrm{d}u\bigg)^{\frac{1}{2}}\\
    &\leq \bigg(\Big(\sup_{j\in \mathcal{J}\setminus\mathcal{J}_K}\eta_j\Big)^{2\alpha}
    \int_{t_l}^s \mathrm{E}\Big[\sum_{j \in \mathcal{J}}\eta_j
    \big\|e^{A(s-u)}B(X_u)Q^{-\alpha}\tilde{e}_j \big\|_H^2\Big]\, \mathrm{d}u\bigg)^{\frac{1}{2}}\\
    &= \bigg(\Big(\sup_{j\in \mathcal{J}\setminus\mathcal{J}_K}\eta_j\Big)^{2\alpha}
    \int_{t_l}^{s} \mathrm{E}\Big[\big\|e^{A(s-u)}B(X_u)Q^{-\alpha}\big\|_{L_{HS}(U_0,H)}^2\Big]\,
    \mathrm{d}u\bigg)^{\frac{1}{2}}.
\end{align*}
By Assumptions (A1), (A3), and Lemma~\ref{SGest}, we get
\begin{align*}
    &\mathrm{E}\bigg[\Big\|\int_{t_l}^s e^{A(s-u)}B(X_u) \, \mathrm{d}(W_u-W_u^K) \Big\|_H^2\bigg]^{\frac{1}{2}}  \\ 
    &\leq \bigg(\Big(\sup_{j\in \mathcal{J}\setminus\mathcal{J}_K}\eta_j\Big)^{2\alpha}
    \int_{t_l}^s \|(-A)^{\vartheta} e^{A(s-u)}\|_{L(H)}^2
    \mathrm{E}\Big[\big\|(-A)^{-\vartheta}B(X_u)Q^{-\alpha}\big\|_{L_{HS}(U_0,H)}^2\Big]\, 
    \mathrm{d}u\bigg)^{\frac{1}{2}}\\
    &\leq \bigg(C\,\Big(\sup_{j\in \mathcal{J}\setminus\mathcal{J}_K}\eta_j\Big)^{2\alpha} 
    \int_{t_l}^s (s-u)^{-2\vartheta}\, \mathrm{d}u\bigg)^{\frac{1}{2}} \\
    &= \bigg(C\,\Big(\sup_{j\in \mathcal{J}\setminus\mathcal{J}_K}\eta_j\Big)^{2\alpha} 
    \frac{(s-t_l)^{-2\vartheta+1}}{1-2\vartheta}\bigg)^{\frac{1}{2}}
\end{align*}\phantomsection \label{QApp2}
all $s\in[0,T]$, $l\in\{0,\ldots,M-1\}$, $M,K\in\mathbb{N}$.
\subsection{The Lipschitz estimate}
Finally for $m\in\{0,\ldots,M\}$, $M\in\mathbb{N}$, we estimate
\begin{align*}
    \E{\bar{X}_{t_m}-\bar{Y}_m}
    &= \Erw \left[\left\|P_N\left(\sI \e{t_m-t_l}\left(F(X_{t_l})-F(\Y_l)\right)\, 
    \mathrm{d}s \right.\right.\right. \\
    &\quad +\sI \e{t_m-t_l}\left(B(X_{t_l})-B(\Y_l)\right)\, \mathrm{d}W_s^K \\
    &\quad + \sI \e{t_m-t_l}\left(B'(X_{t_l})\left(\int_{t_l}^sP_NB(X_{t_l})\, 
    \mathrm{d}W_r^K\right)  \right.\\
    &\quad \left. \left .\left. \left. -B'(\Y_l)\left(\int_{t_l}^sP_NB(\Y_l)\, \mathrm{d}W_r^K\right)\right)\, 
    \mathrm{d}W_s^K\right)\right\|_H^2\right]\\
    &\leq 3\left( Mh \sI  \Erw \left[  \| \e{t_m-t_l}\left(F(X_{t_l})-F(\Y_l)\right)\|_H^2\right]\, \mathrm{d}s \right. \\
    &\quad + \sI \Erw \left[\|\e{t_m-t_l}\left(B(X_{t_l})-B(\Y_l)\right)\|_{L_{HS}(U_0,H)}^2\right]\, \mathrm{d}s \\
    &\quad + \sI \Erw \left[ \| \e{t_m-t_l}\left(B'(X_{t_l})\left(\int_{t_l}^sP_NB(X_{t_l})\, \mathrm{d}W_r^K\right)\right. \right.\\
    &\quad \left. \left. \left. 
    -B'(\Y_l)\left(\int_{t_l}^s P_N B(\Y_l)\, \mathrm{d}W_r^K\right)\right)\|_{L_{HS}(U_0,H)}^2\right]\, \mathrm{d}s\right) \\
    &\leq C_Th \su \Erw \left[  \|F(X_{t_l})-F(\Y_l)\|_H^2\right]
    + C h \su \Erw \left[\|B(X_{t_l})-B(\Y_l)\|_{L_{HS}(U_0,H)}^2\right] \\
    &\quad + \sI \Erw \Bigg[ \| \e{t_m-t_l}\Bigg(B'(X_{t_l})\Bigg(\sum_{\substack{j \in \mathcal{J}_K \\ \eta_j \neq 0}}
    \int_{t_l}^s P_N B(X_{t_l})\tilde{e}_j\sqrt{\eta_j}\, \mathrm{d}\beta_r^j\Bigg)  \\
    &\quad -B'(\Y_l)\Bigg(\sum_{\substack{j \in \mathcal{J}_K \\ \eta_j \neq 0}}
    \int_{t_l}^s P_N B(\Y_l)\tilde{e}_j\sqrt{\eta_j}\, \mathrm{d}\beta_r^j\Bigg)\Bigg)
    \|_{L_{HS}(U_0,H)}^2\Bigg]\, \mathrm{d}s.
\end{align*}
By Assumptions (A2), (A3) and the properties of
the independent Brownian motions $(\beta_t^j)_{t\in[0,T]}$, $j\in\mathcal{J}$, we obtain
\begin{align*}
    \E{\bar{X}_{t_m}-\bar{Y}_m}
    &\leq C_T h \su \Erw \left[  \|X_{t_l}-\Y_l\|_H^2\right]  
    +Ch\su \Erw \left[\|X_{t_l}-\Y_l\|_{H}^2\right] \\
    &\quad + C \sI \Erw \Bigg[ \Bigg\| \Bigg(B'(X_{t_l})
    \Bigg(\sum_{\substack{j \in \mathcal{J}_K\\ \eta_j \neq 0}}P_NB(X_{t_l})
    \tilde{e}_j\sqrt{\eta_j}(\beta_s^j-\beta_{t_l}^j)\Bigg)  \\
    &\quad -B'(\Y_l)\Bigg(\sum_{\substack{j \in \mathcal{J}_K\\ \eta_j \neq 0}}
    P_N B(\Y_l)\tilde{e}_j\sqrt{\eta_j}(\beta_s^j-\beta^j_{t_l})\Bigg)\Bigg) 
    \Bigg\|_{L_{HS}(U_0,H)}^2\Bigg]\, \mathrm{d}s
\end{align*}
and
\begin{align*}
    &\E{\bar{X}_{t_m}-\bar{Y}_m} \\
    &\leq C_Th \su \Erw \left[  \|X_{t_l}-\Y_l\|_H^2\right]  \\
    &\quad+ C \sI \Erw \Bigg[ \| \sum_{\substack{j \in \mathcal{J}_K\\ \eta_j \neq 0}}
    \sqrt{\eta_j}\left(B'(X_{t_l})\left(P_N B(X_{t_l})\tilde{e}_j\right)
    -B'(\Y_l)\left(P_N B(\Y_l)\tilde{e}_j\right)\right) \\
    &\quad \quad \times (\beta_s^j-\beta_{t_l}^j) 
    \|_{L_{HS(U_0,H)}}^2\Bigg]\, \mathrm{d}s\\
    &\leq C_Th \su \Erw \left[  \|X_{t_l}-\Y_l\|_H^2\right] \\
    &\quad + C\sI \sum_{\substack{j \in \mathcal{J}\\ \eta_j \neq 0}}
    \eta_j \Erw \left[ \| \left(B'(X_{t_l})\left(P_NB(X_{t_l})\tilde{e}_j\right)
    -B'(\Y_l)\left(P_N B(\Y_l)\tilde{e}_j\right)\right)\|_{L_{HS(U_0,H)}}^2\right] 
    \\
    &\quad \quad \times 
    \Erw \left[(\beta_s^j-\beta_{t_l}^j)^2\right]\, \mathrm{d}s\\
    &\leq  C_Th \su \Erw \left[  \|X_{t_l}-\Y_l\|_H^2\right] \\
    &\quad + C\sI \Erw \left[ \| B'(X_{t_l})\left(P_NB(X_{t_l})\right)
    -B'(\Y_l)\left(P_N B(\Y_l)\right)\|_{L_{HS}^{(2)}(U_0,H)}^2\right](s-t_l)\, \mathrm{d}s \\
    &\leq C_Th \su \Erw \left[  \|X_{t_l}-\Y_l\|_H^2\right].
\end{align*}
\subsection{Approximation of the derivative}
It remains to show that the approximation of the derivative does not distort the convergence properties. 
Therefore, we prove an estimate for the last term in~\eqref{ZZ} which shows that the rate of 
convergence obtained for the Milstein scheme is not influenced by the approximation of the derivative.

For all $N,K,M \in\mathbb{N}$ and $m\in\{1,\ldots,M\}$, we consider
\begin{align*}
  \E{\bar{Y}_{t_m}-\Y_m}
  &= \mathrm{E}\Bigg[\Bigg\|P_N\Bigg( e^{At_m}X_0 + \sI \e{t_m-t_l}F(Y_l)\, \mathrm{d}s +  \sI \e{t_m-t_l}B(Y_l)\, \mathrm{d}W^K_s \\
  &\quad +\su\Bigg(\frac{1}{2} \e{t_m-t_l} B'(Y_l)\left( P_N B(Y_l)\Delta W^K_l,\Delta W^K_l\right) \\
  &\quad -\frac{h}{2}\e{t_m-t_l} \sum_{\substack{j \in \mathcal{J}_K \\ \eta_j \neq 0}} \eta_j B'(Y_l)\left( P_N B(Y_l)\tilde{e}_j,\tilde{e}_j\right)\Bigg)\Bigg) \\
  &\quad - P_N\Bigg( e^{At_m}X_0 + \sI \e{t_m-t_l}F(Y_l)\, \mathrm{d}s + \sI \e{t_m-t_l}B(Y_l)\, \mathrm{d}W^K_s \\
  &\quad + \su e^{A(t_m-t_l)}\frac{1}{\sqrt{h}}\left(B\left(\Y_l+\frac{1}{2}\sqrt{h}P_NB(\Y_l)\Delta W^K_l \right)-B(\Y_l)\right) \Delta W^K_l\\
  &\quad + \su \sum_{\substack{j \in \mathcal{J}_K \\ \eta_j \neq 0}}e^{A(t_m-t_l)} \bar{B}(Y_l,h,j) \Bigg)\Bigg\|_H^2\Bigg] .
\end{align*}
This expression simplifies and we estimate
\begin{align*}
  &\E{\bar{Y}_{t_m}-\Y_m} \\
  &= \Erw \Bigg[\Bigg\|P_N \Bigg( \su \e{t_m-t_l} \Bigg( \frac{1}{2} B'(Y_l)\left( P_N B(Y_l)\Delta W^K_l,\Delta W^K_l\right) 
  -\frac{h}{2} \sum_{\substack{j \in \mathcal{J}_K \\ \eta_j \neq 0}} \eta_j B'(Y_l) \left( P_N B(Y_l)\tilde{e}_j,\tilde{e}_j\right)\Bigg)\Bigg) \\
  &\quad -P_N \left(\su e^{A(t_m-t_l)}\frac{1}{\sqrt{h}}\left(B\left(\Y_l+\frac{1}{2}\sqrt{h} P_N B(\Y_l)\Delta W^K_l \right)-B(\Y_l)\right)\Delta W^K_l\right)\\
  &\quad -P_N \Bigg(\su \sum_{\substack{j \in \mathcal{J}_K \\ \eta_j \neq 0}}e^{A(t_m-t_l)} \bar{B}(Y_l,h,j) \Bigg)\Bigg\|_H^2\Bigg]
\end{align*}
in the following for all $m\in\{1,\ldots,M\}$. 
Now, we consider
\begin{equation*}
 \bar{B}(Y_l,h,j) = \bigg(B\left(Y_l-\frac{h}{2}P_NB(Y_l)\sqrt{\eta_j}\tilde{e}_j\right)-B(Y_l)\bigg)\sqrt{\eta_j}\tilde{e}_j
\end{equation*}
for $l\in\{0,\ldots,M-1\}$, $j\in\mathcal{J}_K$, first and use Taylor expansions similar to~\eqref{Taylor}. 
Inserting these expressions yields
\begin{align*}
  &\E{\bar{Y}_{t_m}-\Y_m} \\ 
  &\leq  \mathrm{E}\Bigg[\Bigg\| P_N\Bigg(\su \e{t_m-t_l}\Bigg(\frac{1}{2} B'(Y_l)\left( P_N B(Y_l)\Delta W^K_l,\Delta W^K_l\right) 
  -\frac{h}{2} \sum_{\substack{j \in \mathcal{J}_K \\ \eta_j \neq 0}}\eta_j B'(Y_l)\left( P_N B(Y_l)\tilde{e}_j,\tilde{e}_j\right)\Bigg)\Bigg) \\
  &\quad - P_N \left(\su e^{A(t_m-t_l)} \frac{1}{\sqrt{h}} \left( B'(Y_l)\left(\frac{\sqrt{h}}{2}P_NB(Y_l)\Delta W^K_l,\Delta W^K_l \right)\right. \right.\\
  &\quad \left.\left. + \int_0^1 B''( \xi_1(Y_l,u) )
  \left(\frac{\sqrt{h}}{2}P_NB(Y_l)\Delta W^K_l,\frac{\sqrt{h}}{2} P_NB(Y_l)\Delta W^K_l\right)\Delta W^K_l (1-u) \, \mathrm{d}u \right)\right) \\
  &\quad - P_N \Bigg(\su \sum_{\substack{j \in \mathcal{J}_K \\ \eta_j \neq 0}}e^{A(t_m-t_l)} 
  \left(B'(Y_l)\left(-\frac{h}{2}P_NB(Y_l)\sqrt{\eta_j}\tilde{e}_j,\sqrt{\eta_j}\tilde{e}_j\right) \right. \\
  &\quad \left. + \int_0^1 B''( \xi_2(Y_l,j,u) )
  \left(-\frac{h}{2}P_NB(Y_l)\sqrt{\eta_j}\tilde{e}_j,-\frac{h}{2}P_NB(Y_l)\sqrt{\eta_j}\tilde{e}_j\right)
  \sqrt{\eta_j}\tilde{e}_j (1-u) \, \mathrm{d}u \right)\Bigg)\Bigg\|_H^2\Bigg]
\end{align*}
for all $m\in\{1,\ldots,M\}$.
Further, we rewrite
\begin{align*}
  &\E{\bar{Y}_{t_m}-\Y_m} \\ 
  &\leq  \mathrm{E}\Bigg[\Bigg\| \su \frac{1}{\sqrt{h}} e^{A(t_m-t_l)}
  \int_0^1 B''( \xi_1(Y_l,u) )\bigg(\frac{\sqrt{h}}{2}P_NB(Y_l)\Delta W^K_l,\frac{\sqrt{h}}{2}P_NB(Y_l)\Delta W^K_l\bigg)\Delta W^K_l
  (1-u) \, \mathrm{d}u \\
  &\quad +\su \sum_{\substack{j \in \mathcal{J}_K \\ \eta_j \neq 0}} e^{A(t_m-t_l)}
  \int_0^1 B''( \xi_2(Y_l,j,u) )\left(\frac{h}{2}P_NB(Y_l)\sqrt{\eta_j}\tilde{e}_j,\frac{h}{2}P_NB(Y_l)
  \sqrt{\eta_j}\tilde{e}_j\right)\sqrt{\eta_j}\tilde{e}_j (1-u) \, \mathrm{d}u \Bigg\|_H^2\Bigg] \\
  &\leq C \Bigg( \mathrm{E}\Bigg[\Bigg\| \su \frac{1}{\sqrt{h}}  e^{A(t_m-t_l)} \\ 
  &\quad \quad \times
  \int_0^1 B''( \xi_1(Y_l,u) )
  \bigg(\frac{\sqrt{h}}{2} P_N B(Y_l)\Delta W^K_l,\frac{\sqrt{h}}{2}P_NB(Y_l)\Delta W^K_l\bigg)\Delta W^K_l (1-u) \, \mathrm{d}u \bigg\|_H^2\Bigg]^\frac{1}{2}\Bigg)^2 \\
  &\quad + C \Bigg(\mathrm{E}\Bigg[\bigg\| \su \sum_{\substack{j \in \mathcal{J}_K \\ \eta_j \neq 0}} e^{A(t_m-t_l)} \\ 
  &\quad \quad \times 
  \int_0^1 B''( \xi_2(Y_l,j,u) )\left(\frac{h}{2}P_NB(Y_l)\sqrt{\eta_j}\tilde{e}_j,\frac{h}{2}P_NB(Y_l)\sqrt{\eta_j}\tilde{e}_j\right)
  \sqrt{\eta_j}\tilde{e}_j (1-u) \, \mathrm{d}u \Bigg\|_H^2\Bigg]^{\frac{1}{2}}\Bigg)^2 
\end{align*}
for all $m\in\{1,\ldots,M\}$. Assumptions (A1) and (A3) and the triangle inequality imply
\begin{align} \label{proof-eqn-Ito}
  & \E{\bar{Y}_{t_m}-\Y_m} 
  \nonumber 
  \\ 
  &\leq \Bigg(\su \frac{C}{\sqrt{h}}
  \Erw \Bigg[\bigg\|e^{A(t_m-t_l)} 
  \nonumber
  \\ 
  &\quad \quad \times 
  \int_0^1 B''( \xi_1(Y_l,u) ) \bigg(\frac{\sqrt{h}}{2}P_NB(Y_l)\Delta W^K_l,\frac{\sqrt{h}}{2}P_NB(Y_l)\Delta W^K_l\bigg)
  \Delta W^K_l (1-u) \, \mathrm{d}u \bigg\|_H^2\Bigg]^\frac{1}{2}\Bigg)^2 
  \nonumber
  \\
  &\quad + C \Bigg( \su \sum_{\substack{j \in \mathcal{J}_K \\ \eta_j \neq 0}} 
  \Erw \Bigg[\bigg\| e^{A(t_m-t_l)} 
  \nonumber
  \\ 
  &\quad \quad \times
  \int_0^1 B''( \xi_2(Y_l,j,u) )\left(\frac{h}{2}P_NB(Y_l)\sqrt{\eta_j}\tilde{e}_j, \frac{h}{2} P_N B(Y_l)\sqrt{\eta_j}\tilde{e}_j\right)\sqrt{\eta_j}\tilde{e}_j
  (1-u) \, \mathrm{d}u \bigg\|_H^2\Bigg]^{\frac{1}{2}}\Bigg)^2  
  \nonumber
  \\
  &\leq \Bigg(C \su  \frac{1}{\sqrt{h}} 
  \nonumber 
  \\
  &\quad \quad \times 
  \Erw \Bigg[ \bigg( \int_0^1 \| B''( \xi_1(Y_l,u) ) \|_{L^{(2)}(H,L(U,H))}
  \bigg\|\frac{\sqrt{h}}{2} P_NB(Y_l)\Delta W^K_l\bigg\|_{H}^2
  \|\Delta W^K_l\|_{U} (1-u) \, \mathrm{d}u \bigg)^2 \Bigg]^\frac{1}{2}\Bigg)^2 
  \nonumber
  \\
  &\quad + C \Bigg( \su \sum_{\substack{j \in \mathcal{J}_K \\ \eta_j \neq 0}}
  \mathrm{E}\left[ \bigg(\int_0^1
  \left\| B''( \xi_2(Y_l,j,u) ) \right\|_{L^{(2)}(H,L(U,H))} 
  \bigg\| \frac{h}{2} P_N B(Y_l)\sqrt{\eta_j}\tilde{e}_j \bigg\|_{H}^2 \right. 
  \nonumber 
  \\
  &\quad \quad \times \left.
  \|\sqrt{\eta_j}\tilde{e}\|_U (1-u) \, \mathrm{d}u \bigg)^2 \right]^{\frac{1}{2}}\Bigg)^2
\end{align}
for all $m\in\{1,\ldots,M\}$.
Since $Q$ is a trace class operator and by Assumptions (A1)--(A4) as well as by Lemma~\ref{Proof:Lemma-Moment}, 
we obtain for all $K,M \in\mathbb{N}$ and $m\in\{1,\ldots,M\}$
\begin{align*}
  & \E{\bar{Y}_{t_m}-\Y_m} \\
  &\leq \Bigg(C \su  \frac{\sqrt{h}}{4}  \mathrm{E}\left[\left\|B(Y_l)\right\|_{L(U,H_{\delta})}^4\left\|\Delta W^K_l\right\|_{U}^6\right]^\frac{1}{2}\Bigg)^2 
  + \Bigg(C  \su \sum_{\substack{j \in \mathcal{J}_K \\ \eta_j \neq 0}}\frac{h^2}{4}\eta_j^{\frac{3}{2}} 
  \mathrm{E}\left[\left\|B(Y_l)\right\|_{L(U,H_{\delta})}^4\right]^{\frac{1}{2}}\Bigg)^2\\
  &\leq  \Bigg(C \su  \sqrt{h} \left(1+\mathrm{E}\left[\left\|Y_l\right\|_{H_{\delta}}^4\right]\right)^\frac{1}{2} 
  \mathrm{E}\left[\left\|\Delta W^K_l\right\|_{U}^6\right]^\frac{1}{2}\Bigg)^2  
  + \Bigg(C  \su \sum_{\substack{j \in \mathcal{J}_K \\ \eta_j \neq 0}} h^2\eta_j^{\frac{3}{2}} 
  \left(1+\mathrm{E}\left[\left\|Y_l\right\|_{H_{\delta}}^4\right]\right)^\frac{1}{2}\Bigg)^2\\
  &\leq \Bigg(C \su  h^2 \Big(C\Big(1+\mathrm{E}\Big[\|Y_l\|_{H_{\delta}}^4\Big]\Big)\Big)^{\frac{1}{2}} \Bigg)^2 
  + \Bigg(C  \su \bigg(\sup_{j\in \mathcal{J}_K} \sqrt{\eta_j} \bigg) 
  \tr Q \, h^2\Big(1+\mathrm{E}\Big[\|Y_l\|_{H_{\delta}}^4\Big]\Big)^{\frac{1}{2}}\Bigg)^2\\
  & \leq \Bigg(C \su  h^2\Bigg)^2 +2 \Bigg(C  \su \bigg(\sup_{j\in \mathcal{J}} \sqrt{\eta_j}\bigg) \tr Q \frac{h^2}{4}\Bigg)^2\leq C_{T,Q} h^2.
\end{align*}
This proves the error estimate for the general case.

%
Finally, we consider the $\SESRK$ scheme~\eqref{ESRK-scheme-Special}--\eqref{ESRK-scheme-Special-bar-B}. 
Let $N,K,M \in\mathbb{N}$, $l\in\{0,\ldots,M\}$, and $j\in\mathcal{J}_K$. 
For
\begin{equation*}
  \bar{B}(Y_l,h,j) = \left(b\left(\cdot,Y_l-\frac{h}{2}P_Nb(\cdot,Y_l)\right)-b(\cdot,Y_l)\right)\eta_j\tilde{e}_j^2 ,
\end{equation*}
we use the Taylor expansion
\begin{align*}
  b\left(\cdot,Y_l-\frac{h}{2}P_Nb(\cdot,Y_l)\right)\eta_j\tilde{e}_j^2 &= b(\cdot,Y_l)\eta_j\tilde{e}_j^2   
  +b'(\cdot,Y_l) \left(-\frac{h}{2} P_N b(\cdot,Y_l)\right)\eta_j\tilde{e}_j^2  \\
  &\quad + \frac{1}{2} \int_0^1 b''(\cdot, \xi(Y_l,u) )
  \left(-\frac{h}{2} P_N b(\cdot,Y_l)\right) \left(-\frac{h}{2} P_N b(\cdot,Y_l)\right)
  \eta_j \tilde{e}_j^2 (1-u) \, \mathrm{d}u
\end{align*}
with $\xi(Y_l,u) = Y_l - u \frac{h}{2}P_Nb(\cdot,Y_l)$ and the estimate
\begin{align*}
     \E{\bar{Y}_{t_m}-\Y_m} \leq  C_{T,Q} h^2
\end{align*}
follows as above for all $m\in\{0,\ldots,M\}$.
\end{proof}
%
%
%
%
%
%
%

%
%
%
%
%
\bibliographystyle{abbrv}
\bibliography{Diss}
\end{document}